\documentclass[ps,preprint,twoside]{imsart}

\usepackage{amssymb,amsmath,amsthm}
\usepackage{graphpap}
\RequirePackage[colorlinks,citecolor=blue,urlcolor=blue]{hyperref}
\usepackage{helvet,mathrsfs,pifont,bm,tikz}
\usepackage{cite,stackrel}

\usetikzlibrary {shapes}

\newtheorem{thm}{Theorem}[section]
\newtheorem{cor}[thm]{Corollary}
\newtheorem{lem}[thm]{Lemma}
\newtheorem{prop}[thm]{Proposition}

\theoremstyle{definition}
\newtheorem*{rem}{Remark}

\def\Bi{{\rm Binom}}
\def\dtv{d_{\mathrm{TV}}}
\def\ve{\varepsilon}
\def\le{\leqslant}
\def\ge{\geqslant}
\def\tr#1{{\left\lfloor#1\right\rfloor}}
\def\rd#1{{\left\lceil#1\right\rceil}}
\def\dd#1{{\,\mathrm{d}}{#1}}

\doi{10.1214/12-PS213}
\pubyear{2014}
\volume{11}
\issue{0}
\firstpage{177}
\lastpage{236}

\allowdisplaybreaks

\begin{document}

\begin{frontmatter}
\title{Distribution of the sum-of-digits function of random integers: A survey}
\runtitle{Distribution of the sum-of-digits function}

\begin{aug}
\author{\fnms{Louis H. Y.} \snm{Chen}\ead[label=e1]{matchyl@nus.edu.sg}}
\address{Department of Mathematics\\
    National University of Singapore\\
    10 Lower Kent Ridge Road\\
    Singapore 119076 \\
    \printead{e1}}
\end{aug}
\begin{aug}
\author{\fnms{Hsien-Kuei} \snm{Hwang}\thanksref{m2}\ead[label=e2]{hkhwang@stat.sinica.edu.tw}}
\address{Institute of Statistical Science\\
    Institute of Information Science\\
    Academia Sinica\\
    Taipei 115\\
    Taiwan\\
    \printead{e2}}
\end{aug}
\medskip\textbf{\and}
\begin{aug}
\author{\fnms{Vytas} \snm{Zacharovas}\corref{}\thanksref{m3}\ead[label=e3]{vytas.zacharovas@mif.vu.lt}}
\address{Department of Mathematics and Informatics\\
    Vilnius University\\
    Naugarduko 24, Vilnius\\
    Lithuania\\
    \printead{e3}}

\runauthor{L. H. Y. Chen et al.}

\thankstext{m2}{Part of this author's work was done while visiting the Institute
for Mathematical Sciences, National University of Singapore,
supported by Grant C-389-000-012-101; he thanks the Institute
for its support and hospitality.}

\thankstext{m3}{Part of the work of this author was done during
his visit at Institute of Statistical Science, Academia Sinica,
Taipei, and the Institute for Mathematical Sciences, National
University of Singapore. He gratefully acknowledges the support
provided by both institutes.}
\end{aug}

\begin{abstract}
We review some probabilistic properties of the sum-of-digits
function of random integers. New asymptotic approximations to the
total variation distance and its refinements are also derived. Four
different approaches are used: a classical probability approach,
Stein's method, an analytic approach and a new approach based on
Krawtchouk polynomials and the Parseval identity. We also extend the
study to a simple, general numeration system for which similar
approximation theorems are derived.
\end{abstract}



\received{\smonth{12} \syear{2012}}

\end{frontmatter}

\tableofcontents

\section{Introduction}

Positional numeral systems have long been used in the history of
human civilizations, and the sum-of-digits function of an integer,
which equals the sum of all its digits in some given base, appeared
naturally in multitudinous applications such as divisibility check
or check-sum algorithms. Early publications dealing with
divisibility of integers using digit sums date back to at least
Blaise Pascal's \emph{\OE uvres}\footnote{Pascal's \emph{\OE uvres}
is freely available on Wikisource.} in the mid-1650's; see Glaser's
interesting account \cite{glaser1981a}. Numerous properties of the
sum-of-digits function have been extensively studied in the
literature since then; see Chapter XX of Dickson's \emph{History of
Number Theory} \cite{dickson1966a}, which contains a detailed
annotated bibliography for publications up to the early 20th century
dealing with the digits of an integer, and properties discussed
include relations between the digit structures between $n$ and~$n^2$,
iterated sum-of-digits function, general numeral bases, etc.
Modern reviews on digital sums and number systems can be found in
Stolarsky's paper \cite{stolarsky77a}, and the books by Knuth
\cite[\S 4.1]{knuth1997a}, Ifrah \cite{ifrah2000a}, Allouche and
Shallit \cite[Ch.\ 3]{allouche2003a}, S\'andor and Crstici \cite[\S~4.3]{sandor2004a},
Berth\'e\ and Rigo \cite{berthe-rigo2010a}. See
also the two papers by Barat and Grabner \cite{barat06a} and by
Mauduit and Rivat \cite{mauduit10a} for more useful pointers to
several directions relevant to the sum-of-digits function. We are
concerned in this paper with the distributional aspect of the
sum-of-digits function of random integers. Many other types of
results have been investigated in the literature and will not be
reviewed here; most of these results deal with dynamical properties,
exponential sums, Dirichlet series, block occurrences, Thue-Morse
sequence, congruential properties, connections to other structures,
additivity, uniform distribution and discrepancy, sum-of-digits
under special subsequences, etc.\looseness=-1

More precisely, let $q\ge2$ be a fixed integer and $n= \sum_{0\le
j\le \lambda}\ve_j q^j$, where $\ve_j\in\{0,\dots,q-1\}$ and
$\lambda = \tr{\log_q n}$. Then the sum-of-digits function
$\nu_q(n)$ of $n$ in base $q$ is defined as $\sum_{0\le j\le
\lambda}\ve_j$.\vspace*{1pt} When $q=2$, we write $\nu(n)=\nu_2(n)$, which is the
number of ones in the binary representation of $n$.

Since the distribution of $\nu_q(n)$ is irregular in
the sense that its values can vary between $(q-1)\tr{\log_q n}$ and
$1$ (see Figure \ref{fig1} for $q=2$), we consider $X_n=X_n(q)$, which
denotes the random variable equal to $\nu_q(U_n)$, where $U_n$
assumes each of the values $\{0,\dots,n-1\}$ with equal probability
$1/n$. The behavior of $X_n$ is then more smooth.

\begin{figure}[t!]
\begin{tikzpicture}[yscale=0.8,xscale=1.3]
\newcommand{\FONTSIZE}{\fontsize{8pt}{\baselineskip}\selectfont}
\draw[-latex] (-0.50, 0.000) -- (5.25, 0.000) node[right]
{\scriptsize $n$};
\draw (0.607843137, 0.1) -- (0.607843137, -0.1) node[below]
{\FONTSIZE{$32$}}%
(1.235294118, 0.1) -- (1.235294118, -0.1) node[below]
{\FONTSIZE{$64$}}%
(2.490196078,0.1) -- (2.490196078,-0.1) node[below] {\FONTSIZE{$128$}}%
(5, 0.1) -- (5, -0.1) node[below] {\FONTSIZE{$256$}};%
\draw[-latex] (0,-0.50) -- (0.000, 5.50) node[right] {$\nu_2(n)$};
\draw (1.23,5.3) node[right] {
}; 
\foreach \y/\ytext in {0.000, 0.357142857, ..., 5.000  }
    \draw[shift={(0.000, \y)}] (0.025,0) -- (-0.025, 0);
\draw  (0.050, 0.714285714 ) -- (-0.050, 0.714285714 ) node[left]
{\FONTSIZE{$2$}} (0.050, 1.428571429 ) -- (-0.050, 1.428571429 )
node[left] {\FONTSIZE{$3$}} (0.050, 2.142857143 ) --
(-0.050, 2.142857143 ) node[left] {\FONTSIZE{$4$}}
(0.050, 2.857142857 ) -- (-0.050, 2.857142857) node[left]
{\FONTSIZE{$5$}} (0.050, 3.571428571 ) -- (-0.050, 3.571428571 )
node[left] {\FONTSIZE{$6$}} (0.050, 4.285714286 ) -- (-0.050,
4.285714286 ) node[left] {\FONTSIZE{$7$}} (0.050, 5) -- (-0.050,
5) node[left] {\FONTSIZE{$8$}};
\definecolor{gen}{rgb}{0.2,0.5,0.2}
\filldraw [top color=blue!50,bottom color=gen!80] plot
coordinates{(0,0)(0.02,0)(0.039,0.714)(0.059,0)(0.078,0.714)
(0.098,0.714)(0.118,1.429)(0.137,0)(0.157,0.714)(0.176,0.714)
(0.196,1.429)(0.216,0.714)(0.235,1.429)(0.255,1.429)(0.275,2.143)
(0.294,0)(0.314,0.714)(0.333,0.714)(0.353,1.429)(0.373,0.714)
(0.392,1.429)(0.412,1.429)(0.431,2.143)(0.451,0.714)(0.471,1.429)
(0.49,1.429)(0.51,2.143)(0.529,1.429)(0.549,2.143)(0.569,2.143)
(0.588,2.857)(0.608,0)(0.627,0.714)(0.647,0.714)(0.667,1.429)
(0.686,0.714)(0.706,1.429)(0.725,1.429)(0.745,2.143)(0.765,0.714)
(0.784,1.429)(0.804,1.429)(0.824,2.143)(0.843,1.429)(0.863,2.143)
(0.882,2.143)(0.902,2.857)(0.922,0.714)(0.941,1.429)(0.961,1.429)
(0.98,2.143)(1,1.429)(1.02,2.143)(1.039,2.143)(1.059,2.857)
(1.078,1.429)(1.098,2.143)(1.118,2.143)(1.137,2.857)(1.157,2.143)
(1.176,2.857)(1.196,2.857)(1.216,3.571)(1.235,0)(1.255,0.714)
(1.275,0.714)(1.294,1.429)(1.314,0.714)(1.333,1.429)(1.353,1.429)
(1.373,2.143)(1.392,0.714)(1.412,1.429)(1.431,1.429)(1.451,2.143)
(1.471,1.429)(1.49,2.143)(1.51,2.143)(1.529,2.857)(1.549,0.714)
(1.569,1.429)(1.588,1.429)(1.608,2.143)(1.627,1.429)(1.647,2.143)
(1.667,2.143)(1.686,2.857)(1.706,1.429)(1.725,2.143)(1.745,2.143)
(1.765,2.857)(1.784,2.143)(1.804,2.857)(1.824,2.857)(1.843,3.571)
(1.863,0.714)(1.882,1.429)(1.902,1.429)(1.922,2.143)(1.941,1.429)
(1.961,2.143)(1.98,2.143)(2,2.857)(2.02,1.429)(2.039,2.143)
(2.059,2.143)(2.078,2.857)(2.098,2.143)(2.118,2.857)(2.137,2.857)
(2.157,3.571)(2.176,1.429)(2.196,2.143)(2.216,2.143)(2.235,2.857)
(2.255,2.143)(2.275,2.857)(2.294,2.857)(2.314,3.571)(2.333,2.143)
(2.353,2.857)(2.373,2.857)(2.392,3.571)(2.412,2.857)(2.431,3.571)
(2.451,3.571)(2.471,4.286)(2.49,0)(2.51,0.714)(2.529,0.714)
(2.549,1.429)(2.569,0.714)(2.588,1.429)(2.608,1.429)(2.627,2.143)
(2.647,0.714)(2.667,1.429)(2.686,1.429)(2.706,2.143)(2.725,1.429)
(2.745,2.143)(2.765,2.143)(2.784,2.857)(2.804,0.714)(2.824,1.429)
(2.843,1.429)(2.863,2.143)(2.882,1.429)(2.902,2.143)(2.922,2.143)
(2.941,2.857)(2.961,1.429)(2.98,2.143)(3,2.143)(3.02,2.857)
(3.039,2.143)(3.059,2.857)(3.078,2.857)(3.098,3.571)(3.118,0.714)
(3.137,1.429)(3.157,1.429)(3.176,2.143)(3.196,1.429)(3.216,2.143)
(3.235,2.143)(3.255,2.857)(3.275,1.429)(3.294,2.143)(3.314,2.143)
(3.333,2.857)(3.353,2.143)(3.373,2.857)(3.392,2.857)(3.412,3.571)
(3.431,1.429)(3.451,2.143)(3.471,2.143)(3.49,2.857)(3.51,2.143)
(3.529,2.857)(3.549,2.857)(3.569,3.571)(3.588,2.143)(3.608,2.857)
(3.627,2.857)(3.647,3.571)(3.667,2.857)(3.686,3.571)(3.706,3.571)
(3.725,4.286)(3.745,0.714)(3.765,1.429)(3.784,1.429)(3.804,2.143)
(3.824,1.429)(3.843,2.143)(3.863,2.143)(3.882,2.857)(3.902,1.429)
(3.922,2.143)(3.941,2.143)(3.961,2.857)(3.98,2.143)(4,2.857)
(4.02,2.857)(4.039,3.571)(4.059,1.429)(4.078,2.143)(4.098,2.143)
(4.118,2.857)(4.137,2.143)(4.157,2.857)(4.176,2.857)(4.196,3.571)
(4.216,2.143)(4.235,2.857)(4.255,2.857)(4.275,3.571)(4.294,2.857)
(4.314,3.571)(4.333,3.571)(4.353,4.286)(4.373,1.429)(4.392,2.143)
(4.412,2.143)(4.431,2.857)(4.451,2.143)(4.471,2.857)(4.49,2.857)
(4.51,3.571)(4.529,2.143)(4.549,2.857)(4.569,2.857)(4.588,3.571)
(4.608,2.857)(4.627,3.571)(4.647,3.571)(4.667,4.286)(4.686,2.143)
(4.706,2.857)(4.725,2.857)(4.745,3.571)(4.765,2.857)(4.784,3.571)
(4.804,3.571)(4.824,4.286)(4.843,2.857)(4.863,3.571)(4.882,3.571)
(4.902,4.286)(4.922,3.571)(4.941,4.286)(4.961,4.286)(4.98,5)(5,0)}--
(5,0) -- (0,0) -- cycle;
\end{tikzpicture}%
\caption{$\nu_2(n)$, $n=1,\dots,256$.}\label{fig1}
\end{figure}
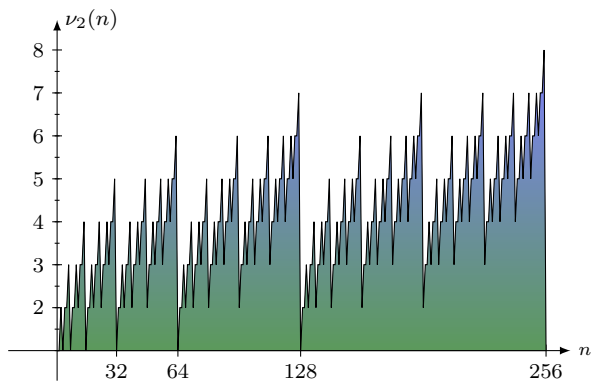

Obviously, when $n=q^k-1$, the distribution of $X_n$ is exactly
multinomial with parameters $k$ and $q$ identical probabilities
$1/q$. The general difficulty then lies in estimating the closeness
between the distribution of $X_n$ and a suitably chosen multinomial
distribution. Periodicities are then ubiquitous in the study of most
asymptotic problems involved.

We will mainly review known results for the mean, the variance, the
higher moments and the limit distribution of $X_n$, as well as
related asymptotic approximations. It turns out that many of such
results have been derived independently in the literature, and
rediscoveries are not uncommon. As Stolarsky \cite{stolarsky77a}
puts it\looseness=-1

\begin{quote}
    \normalsize
    \textit{``Whatever its mathematical virtues, the literature on
    sums of digital sums reflects a lack of communication
    between researchers.''}
\end{quote}\looseness=0
In view of the large number of independent discoveries it is likely
that we missed some papers in our attempt to give a more complete
collection of relevant known results.

In addition to reviewing known stochastic properties of $X_n$, we
will present new approximations to the distribution of $X_n$. For
simplicity, we focus on the binary case $q=2$, leaving the
straightforward extension to other numeration systems to the
interested reader. In particular, our results imply that the total
variation distance between the distribution of $X_n$ and a binomial
random variable $Y_\lambda$ of parameters $\lambda := \tr{\log_2n}$
and $\tfrac12$ is asymptotic to (see Figure~\ref{fig2})
\begin{align}
    \dtv(\mathscr{L}(X_n),\mathscr{L}(Y_\lambda))
    &=\frac12\sum_{k\ge0} \left|\mathbb{P}(X_n=k)-
    2^{-\lambda}\binom{\lambda}{k}\right| \nonumber\\
    &= \frac{\sqrt{2}\,|F(\log_2n)|}{\sqrt{\pi\lambda}}+
    O\left(\lambda^{-1}\right), \label{dtv}
\end{align}
where, interestingly,
\[
    F(\log_2n) = \mathbb{E}(X_n)-\frac{\lambda}{2}.
\]
The function $F$ is a bounded, periodic function (namely, $F(x)=
F(x+1)$) with discontinuities at integers; see (\ref{Fx}) for the
definition of $F$ for arbitrary $x$, and Figure~\ref{fig3} for 
a graphical rendering.

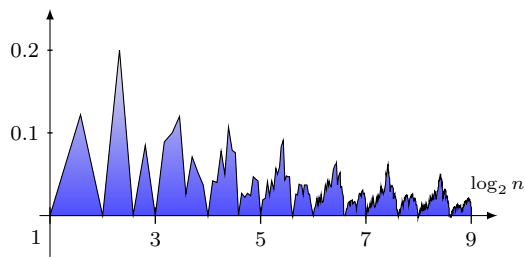
\begin{figure}[t!]
\begin{tikzpicture}[yscale=11,xscale=0.7]
\newcommand{\FONTSIZE}{\fontsize{8pt}{\baselineskip}\selectfont}
\draw[-latex] (0.8, 0) -- (9.5, 0) node[above=5pt]
{\scriptsize $\log_2n$};
\draw[line width=0.5pt] (1,0.01) -- (1, -0.01)
node[below=5pt,left] {\FONTSIZE{$1$}}%
(3, 0.01) -- (3, -0.01) node[below] {\FONTSIZE{$3$}}%
(5, 0.01) -- (5, -0.01) node[below] {\FONTSIZE{$5$}}%
(7,0.01) -- (7, -0.01) node[below] {\FONTSIZE{$7$}}%
(9, 0.01)-- (9, -0.01) node[below] {\FONTSIZE{$9$}};%
\draw[-latex] (1, -0.05) -- (1, 0.25); \draw (5.5, 0.27)
node{%
};%
\draw [line width=0.5pt] (1.05, 0.1) -- (0.95, 0.1) node[left]
{\FONTSIZE{$0.1$}}%
(1.05, 0.2) -- (0.95, 0.2) node[left] {\FONTSIZE{$0.2$}};%
\definecolor{gen}{rgb}{0,0,1};%
\filldraw [top color=gray!20,bottom color=blue!70] plot
coordinates{(1,0.000)(1.58,0.122)(2,0.000)(2.32, 0.200) (2.58, 0.001)
(2.81, 0.085)(3,0.000)(3.17 ,0.089)(3.32, 0.100) (3.46
,0.120)(3.58 ,0.026)(3.70 ,0.071)(3.81 ,0.052)(3.91 ,0.037) (4.00
,0.000)(4.09 ,0.042)(4.17 ,0.040)(4.25 ,0.078)(4.32 ,0.050) (4.39
,0.107)(4.46 ,0.079)(4.52 ,0.076)(4.58 ,0.001)(4.64 ,0.027) (4.70
,0.022)(4.75 ,0.027)(4.81 ,0.024)(4.86 ,0.047)(4.91 ,0.044) (4.95
,0.042)(5,0.000)(5.04 ,0.019)(5.09 ,0.021)(5.13 ,0.041) (5.17
,0.024)(5.21 ,0.042)(5.25 ,0.032)(5.29 ,0.057)(5.32 ,0.050) (5.36
,0.059)(5.39 ,0.083)(5.43 ,0.091)(5.46 ,0.043)(5.49 ,0.048) (5.52
,0.047)(5.55 ,0.047)(5.58 ,0.012)(5.61 ,-0.001)(5.64 ,0.013) (5.67
,0.026)(5.70 ,0.025)(5.73 ,0.038)(5.75 ,0.037)(5.78 ,0.037) (5.81
,0.025)(5.83 ,0.037)(5.86 ,0.037)(5.88 ,0.037)(5.91 ,0.026) (5.93
,0.026)(5.95 ,0.016)(5.98 ,0.006)(6,0.000)(6.02 ,0.008) (6.04
,0.011)(6.07 ,0.021)(6.09 ,0.012)(6.11 ,0.022)(6.13 ,0.016) (6.15
,0.026)(6.17 ,0.016)(6.19 ,0.017)(6.21 ,0.027)(6.23 ,0.037) (6.25
,0.031)(6.27 ,0.042)(6.29 ,0.036)(6.30 ,0.042)(6.32 ,0.037) (6.34
,0.035)(6.36 ,0.045)(6.38 ,0.055)(6.39 ,0.057)(6.41 ,0.061) (6.43
,0.062)(6.44 ,0.064)(6.46 ,0.037)(6.48 ,0.047)(6.49 ,0.049) (6.51
,0.052)(6.52 ,0.034)(6.54 ,0.035)(6.55 ,0.029)(6.57 ,0.023) (6.58
,0.001)(6.60 ,0.000)(6.61 ,0.001)(6.63 ,0.008)(6.64 ,0.008) (6.66
,0.015)(6.67 ,0.016)(6.69 ,0.017)(6.70 ,0.011)(6.71 ,0.018) (6.73
,0.018)(6.74 ,0.019)(6.75 ,0.014)(6.77 ,0.016)(6.78 ,0.017) (6.79
,0.017)(6.81 ,0.013)(6.82 ,0.007)(6.83 ,0.013)(6.85 ,0.019) (6.86
,0.020)(6.87 ,0.026)(6.88 ,0.027)(6.89 ,0.027)(6.91 ,0.023) (6.92
,0.029)(6.93 ,0.030)(6.94 ,0.031)(6.95 ,0.026)(6.97 ,0.027) (6.98
,0.023)(6.99 ,0.019)(7,0.000)(7.01 ,0.003)(7.02 ,0.006) (7.03
,0.011)(7.04 ,0.005)(7.06 ,0.010)(7.07 ,0.010)(7.08 ,0.015) (7.09
,0.010)(7.10 ,0.011)(7.11 ,0.016)(7.12 ,0.021)(7.13 ,0.016) (7.14
,0.021)(7.15 ,0.016)(7.16 ,0.011)(7.17 ,0.010)(7.18 ,0.011) (7.19
,0.016)(7.20 ,0.021)(7.21 ,0.017)(7.22 ,0.023)(7.23 ,0.019) (7.24
,0.024)(7.25 ,0.021)(7.26 ,0.022)(7.27 ,0.028)(7.28 ,0.033) (7.29
,0.030)(7.29 ,0.035)(7.30 ,0.032)(7.31 ,0.029)(7.32 ,0.031) (7.33
,0.026)(7.34 ,0.028)(7.35 ,0.029)(7.36 ,0.034)(7.37 ,0.037) (7.38
,0.044)(7.38 ,0.051)(7.39 ,0.049)(7.40 ,0.052)(7.41 ,0.059) (7.42
,0.062)(7.43 ,0.051)(7.43 ,0.054)(7.44 ,0.047)(7.45 ,0.045) (7.46
,0.029)(7.47 ,0.034)(7.48 ,0.035)(7.48 ,0.036)(7.49 ,0.035) (7.50
,0.036)(7.51 ,0.034)(7.52 ,0.033)(7.52 ,0.020)(7.53 ,0.024) (7.54
,0.026)(7.55 ,0.027)(7.55 ,0.022)(7.56 ,0.023)(7.57 ,0.020) (7.58
,0.018)(7.58 ,0.008)(7.59 ,0.007)(7.60 ,0.003)(7.61 ,0.000) (7.61
,0.003)(7.62 ,0.000)(7.63 ,0.003)(7.64 ,0.007)(7.64 ,0.008) (7.65
,0.004)(7.66 ,0.008)(7.67 ,0.011)(7.67 ,0.012)(7.68 ,0.015) (7.69
,0.016)(7.69 ,0.017)(7.70 ,0.014)(7.71 ,0.011)(7.71 ,0.014) (7.72
,0.018)(7.73 ,0.018)(7.73 ,0.022)(7.74 ,0.022)(7.75 ,0.023) (7.75
,0.021)(7.76 ,0.024)(7.77 ,0.025)(7.77 ,0.025)(7.78 ,0.023) (7.79
,0.024)(7.79 ,0.022)(7.80 ,0.020)(7.81 ,0.015)(7.81 ,0.012) (7.82
,0.015)(7.83 ,0.019)(7.83 ,0.019)(7.84 ,0.022)(7.85 ,0.023) (7.85
,0.023)(7.86 ,0.022)(7.86 ,0.025)(7.87 ,0.025)(7.88 ,0.026) (7.88
,0.024)(7.89 ,0.025)(7.89 ,0.023)(7.90 ,0.021)(7.91 ,0.017) (7.91
,0.020)(7.92 ,0.021)(7.92 ,0.021)(7.93 ,0.019)(7.94 ,0.020) (7.94
,0.018)(7.95 ,0.017)(7.95 ,0.013)(7.96 ,0.013)(7.97 ,0.012) (7.97
,0.010)(7.98 ,0.006)(7.98 ,0.005)(7.99 ,0.001)(7.99 ,-0.003) (8.00
,0.000)(8.01 ,0.001)(8.01 ,0.002)(8.02 ,0.005)(8.02 ,0.002) (8.03
,0.004)(8.03 ,0.006)(8.04 ,0.008)(8.04 ,0.006)(8.05 ,0.006) (8.06
,0.009)(8.06 ,0.012)(8.07 ,0.009)(8.07 ,0.012)(8.08 ,0.009) (8.08
,0.006)(8.09 ,0.006)(8.09 ,0.006)(8.10 ,0.009)(8.10 ,0.011) (8.11
,0.009)(8.11 ,0.011)(8.12 ,0.009)(8.12 ,0.011)(8.13 ,0.009) (8.13
,0.009)(8.14 ,0.012)(8.14 ,0.014)(8.15 ,0.012)(8.15 ,0.014) (8.16
,0.012)(8.16 ,0.009)(8.17 ,0.009)(8.17 ,0.008)(8.18 ,0.009) (8.18
,0.009)(8.19 ,0.012)(8.19 ,0.013)(8.20 ,0.016)(8.20 ,0.018) (8.21
,0.016)(8.21 ,0.017)(8.22 ,0.019)(8.22 ,0.022)(8.23 ,0.020) (8.23
,0.022)(8.24 ,0.020)(8.24 ,0.018)(8.25 ,0.018)(8.25 ,0.018) (8.26
,0.021)(8.26 ,0.023)(8.27 ,0.021)(8.27 ,0.024)(8.28 ,0.022) (8.28
,0.020)(8.29 ,0.020)(8.29 ,0.023)(8.29 ,0.024)(8.30 ,0.027) (8.30
,0.026)(8.31 ,0.030)(8.31 ,0.029)(8.32 ,0.027)(8.32 ,0.028) (8.33
,0.020)(8.33 ,0.020)(8.34 ,0.020)(8.34 ,0.022)(8.34 ,0.023) (8.35
,0.026)(8.35 ,0.029)(8.36 ,0.029)(8.36 ,0.029)(8.37 ,0.032) (8.37
,0.034)(8.38 ,0.035)(8.38 ,0.037)(8.38 ,0.040)(8.39 ,0.043) (8.39
,0.042)(8.40 ,0.042)(8.40 ,0.044)(8.41 ,0.046)(8.41 ,0.049) (8.41
,0.050)(8.42 ,0.050)(8.42 ,0.050)(8.43 ,0.043)(8.43 ,0.044) (8.43
,0.044)(8.44 ,0.044)(8.44 ,0.037)(8.45 ,0.039)(8.45 ,0.033) (8.46
,0.030)(8.46 ,0.022)(8.46 ,0.020)(8.47 ,0.022)(8.47 ,0.024) (8.48
,0.025)(8.48 ,0.027)(8.48 ,0.029)(8.49 ,0.030)(8.49 ,0.028) (8.50
,0.030)(8.50 ,0.030)(8.50 ,0.031)(8.51 ,0.030)(8.51 ,0.031) (8.52
,0.030)(8.52 ,0.029)(8.52 ,0.021)(8.53 ,0.023)(8.53 ,0.024) (8.54
,0.024)(8.54 ,0.024)(8.54 ,0.024)(8.55 ,0.023)(8.55 ,0.023) (8.55
,0.016)(8.56 ,0.017)(8.56 ,0.016)(8.57 ,0.015)(8.57 ,0.012) (8.57
,0.011)(8.58 ,0.009)(8.58 ,0.007)(8.58 ,0.000)(8.59 ,0.001) (8.59
,0.001)(8.60 ,0.000)(8.60 ,-0.001)(8.60 ,-0.002) (8.61 ,-0.002)(8.61
,0.000)(8.61 ,0.000)(8.62 ,-0.001)(8.62 ,0.001) (8.63 ,0.002)(8.63
,0.003)(8.63 ,0.005)(8.64 ,0.005)(8.64 ,0.006) (8.64 ,0.005)(8.65
,0.003)(8.65 ,0.005)(8.65 ,0.007)(8.66 ,0.007) (8.66 ,0.009)(8.67
,0.009)(8.67 ,0.010)(8.67 ,0.009)(8.68 ,0.011) (8.68 ,0.011)(8.68
,0.011)(8.69 ,0.011)(8.69 ,0.011)(8.69 ,0.010) (8.70 ,0.009)(8.70
,0.007)(8.70 ,0.005)(8.71 ,0.007)(8.71 ,0.009) (8.71 ,0.009)(8.72
,0.011)(8.72 ,0.011)(8.72 ,0.012)(8.73 ,0.011) (8.73 ,0.013)(8.73
,0.013)(8.74 ,0.014)(8.74 ,0.013)(8.74 ,0.013) (8.75 ,0.012)(8.75
,0.012)(8.75 ,0.009)(8.76 ,0.011)(8.76 ,0.012) (8.76 ,0.012)(8.77
,0.011)(8.77 ,0.012)(8.77 ,0.011)(8.78 ,0.010) (8.78 ,0.009)(8.78
,0.011)(8.79 ,0.012)(8.79 ,0.012)(8.79 ,0.011) (8.80 ,0.012)(8.80
,0.011)(8.80 ,0.010)(8.81 ,0.008)(8.81 ,0.008) (8.81 ,0.006)(8.82
,0.005)(8.82 ,0.006)(8.82 ,0.005)(8.83 ,0.006) (8.83 ,0.008)(8.83
,0.008)(8.84 ,0.007)(8.84 ,0.008)(8.84 ,0.010) (8.85 ,0.010)(8.85
,0.012)(8.85 ,0.012)(8.85 ,0.013)(8.86 ,0.012) (8.86 ,0.011)(8.86
,0.012)(8.87 ,0.014)(8.87 ,0.014)(8.87 ,0.016) (8.88 ,0.016)(8.88
,0.017)(8.88 ,0.016)(8.89 ,0.018)(8.89 ,0.018) (8.89 ,0.018)(8.89
,0.018)(8.90 ,0.018)(8.90 ,0.017)(8.90 ,0.017) (8.91 ,0.015)(8.91
,0.013)(8.91 ,0.015)(8.92 ,0.016)(8.92 ,0.017) (8.92 ,0.018)(8.92
,0.019)(8.93 ,0.019)(8.93 ,0.019)(8.93 ,0.020) (8.94 ,0.020)(8.94
,0.021)(8.94 ,0.020)(8.95 ,0.021)(8.95 ,0.020) (8.95 ,0.019)(8.95
,0.018)(8.96 ,0.019)(8.96 ,0.019)(8.96 ,0.020) (8.97 ,0.019)(8.97
,0.020)(8.97 ,0.019)(8.97 ,0.018)(8.98 ,0.017) (8.98 ,0.017)(8.98
,0.016)(8.99 ,0.016)(8.99 ,0.014)(8.99 ,0.013) (8.99 ,0.012)(9.00
,0.010)(9,0.000)}-- (1.000 ,0.000) --
(1 ,0.000) -- cycle;%
\end{tikzpicture}
\caption{$\dtv(\mathscr{L}(X_n),\mathscr{L}(Y_\lambda))
-\frac{\sqrt{2}\,|F(\log_2n)|}{\sqrt{\pi\lambda}}$.}\label{fig2}
\end{figure}

\begin{figure}[t]
\begin{tikzpicture}[yscale=6,xscale=5]
\newcommand{\FONTSIZE}{\fontsize{8pt}{\baselineskip}\selectfont}
\draw[-latex] (-0.05 , 0) -- (1.1 , 0) node[right] {};
\draw[line width=0.5pt] (0.2 ,0.01) -- (0.2 , -0.01) node[below]
{\FONTSIZE{$0.2$}}%
(0.4 , 0.01) -- (0.4 , -0.01) node[below] {\FONTSIZE{$0.4$}}%
(0.6 , 0.01) -- (0.6 , -0.01) node[below] {\FONTSIZE{$0.6$}}%
(0.8,0.01) -- (0.8 , -0.01) node[below] {\FONTSIZE{$0.8$}}%
(1.0 , 0.01)-- (1.0 , -0.01) node[below] {\FONTSIZE{$1$}};%
\draw[-latex] (0.000 , -0.04) -- (0.000 , 0.55);%
\draw (0.15 , 0.5) node[right]{
};%
\foreach \y/\ytext in { 0.000, 0.1, ..., 0.5  }%
\draw[shift={(0.000, \y)}] (0.001, 0) -- (-0.001, 0);%
\draw [line width=0.5pt] (0.01 , 0.1 ) -- (-0.01 , 0.1 ) node[left]
{\FONTSIZE{$0.1$}}%
(0.01 , 0.2 ) -- (-0.01 , 0.2 ) node[left] {\FONTSIZE{$0.2$}}%
(0.01 , 0.3 ) -- (-0.01 , 0.3 ) node[left] {\FONTSIZE{$0.3$}}%
(0.01 , 0.4 ) -- (-0.01 , 0.4 ) node[left] {\FONTSIZE{$0.4$}}%
(0.01 , 0.5 ) -- (-0.01 , 0.5 ) node[left] {\FONTSIZE{$0.5$}};%
\definecolor{gen}{rgb}{0,0,1};%
\filldraw [top color=black,bottom color=gray!80] plot
coordinates{(0.000 ,0.000)(0.006 ,0.012)(0.011 ,0.019)(0.017 ,0.027)
(0.022 ,0.031)(0.028 ,0.038)(0.033 ,0.042)(0.039 ,0.046)
(0.044 ,0.045)(0.050 ,0.053)(0.055 ,0.056)(0.061 ,0.060)
(0.066 ,0.060)(0.071 ,0.063)(0.077 ,0.063)(0.082 ,0.063)
(0.087 ,0.059)(0.093 ,0.066)(0.098 ,0.069)(0.103 ,0.073)
(0.109 ,0.072)(0.114 ,0.076)(0.119 ,0.076)(0.124 ,0.075)
(0.129 ,0.071)(0.134 ,0.075)(0.140 ,0.074)(0.145 ,0.074)
(0.150 ,0.070)(0.155 ,0.070)(0.160 ,0.066)(0.165 ,0.063)
(0.170 ,0.056)(0.175 ,0.062)(0.180 ,0.066)(0.185 ,0.069)
(0.190 ,0.068)(0.195 ,0.072)(0.200 ,0.071)(0.205 ,0.071)
(0.209 ,0.068)(0.214 ,0.071)(0.219 ,0.070)(0.224 ,0.070)
(0.229 ,0.067)(0.234 ,0.066)(0.238 ,0.063)(0.243 ,0.059)
(0.248 ,0.053)(0.253 ,0.056)(0.257 ,0.056)(0.262 ,0.055)
(0.267 ,0.052)(0.271 ,0.052)(0.276 ,0.048)(0.281 ,0.045)
(0.285 ,0.038)(0.290 ,0.038)(0.295 ,0.035)(0.299 ,0.032)
(0.304 ,0.025)(0.308 ,0.022)(0.313 ,0.016)(0.317 ,0.009)
(0.322 ,0.000)(0.326 ,0.006)(0.331 ,0.009)(0.335 ,0.012)
(0.340 ,0.012)(0.344 ,0.015)(0.349 ,0.015)(0.353 ,0.015)
(0.358 ,0.012)(0.362 ,0.015)(0.366 ,0.015)(0.371 ,0.015)
(0.375 ,0.012)(0.379 ,0.012)(0.384 ,0.009)(0.388 ,0.006)
(0.392 ,0.000)(0.397 ,0.003)(0.401 ,0.003)(0.405 ,0.003)
(0.409 ,0.000)(0.414 ,0.000)(0.418 ,0.003)(0.422 ,0.006)
(0.426 ,0.012)(0.430 ,0.012)(0.435 ,0.014)(0.439 ,0.017)
(0.443 ,0.023)(0.447 ,0.026)(0.451 ,0.031)(0.455 ,0.037)
(0.459 ,0.045)(0.464 ,0.042)(0.468 ,0.042)(0.472 ,0.042)
(0.476 ,0.045)(0.480 ,0.045)(0.484 ,0.047)(0.488 ,0.050)
(0.492 ,0.056)(0.496 ,0.055)(0.500 ,0.058)(0.504 ,0.061)
(0.508 ,0.066)(0.512 ,0.068)(0.516 ,0.074)(0.520 ,0.079)
(0.524 ,0.087)(0.527 ,0.087)(0.531 ,0.089)(0.535 ,0.092)
(0.539 ,0.097)(0.543 ,0.099)(0.547 ,0.104)(0.551 ,0.109)
(0.555 ,0.117)(0.558 ,0.119)(0.562 ,0.124)(0.566 ,0.129)
(0.570 ,0.137)(0.574 ,0.142)(0.577 ,0.149)(0.581 ,0.157)
(0.585 ,0.167)(0.589 ,0.161)(0.592 ,0.158)(0.596 ,0.155)
(0.600 ,0.155)(0.604 ,0.152)(0.607 ,0.151)(0.611 ,0.151)
(0.615 ,0.153)(0.618 ,0.150)(0.622 ,0.150)(0.626 ,0.149)
(0.629 ,0.152)(0.633 ,0.151)(0.637 ,0.153)(0.640 ,0.155)
(0.644 ,0.160)(0.647 ,0.157)(0.651 ,0.157)(0.655 ,0.156)
(0.658 ,0.158)(0.662 ,0.158)(0.665 ,0.160)(0.669 ,0.162)
(0.672 ,0.167)(0.676 ,0.166)(0.679 ,0.168)(0.683 ,0.170)
(0.687 ,0.175)(0.690 ,0.177)(0.693 ,0.181)(0.697 ,0.186)
(0.700 ,0.192)(0.704 ,0.189)(0.707 ,0.189)(0.711 ,0.189)
(0.714 ,0.190)(0.718 ,0.190)(0.721 ,0.192)(0.725 ,0.194)
(0.728 ,0.198)(0.731 ,0.198)(0.735 ,0.200)(0.738 ,0.201)
(0.741 ,0.206)(0.745 ,0.207)(0.748 ,0.212)(0.752 ,0.216)
(0.755 ,0.222)(0.758 ,0.222)(0.762 ,0.224)(0.765 ,0.225)
(0.768 ,0.229)(0.771 ,0.231)(0.775 ,0.235)(0.778 ,0.239)
(0.781 ,0.245)(0.785 ,0.247)(0.788 ,0.251)(0.791 ,0.255)
(0.794 ,0.261)(0.798 ,0.265)(0.801 ,0.271)(0.804 ,0.277)
(0.807 ,0.286)(0.811 ,0.283)(0.814 ,0.282)(0.817 ,0.282)
(0.820 ,0.283)(0.823 ,0.283)(0.827 ,0.284)(0.830 ,0.286)
(0.833 ,0.289)(0.836 ,0.289)(0.839 ,0.290)(0.842 ,0.292)
(0.845 ,0.296)(0.849 ,0.297)(0.852 ,0.301)(0.855 ,0.305)
(0.858 ,0.310)(0.861 ,0.310)(0.864 ,0.311)(0.867 ,0.313)
(0.870 ,0.316)(0.873 ,0.318)(0.877 ,0.321)(0.880 ,0.325)
(0.883 ,0.331)(0.886 ,0.332)(0.889 ,0.335)(0.892 ,0.339)
(0.895 ,0.345)(0.898 ,0.348)(0.901 ,0.354)(0.904 ,0.359)
(0.907 ,0.367)(0.910 ,0.366)(0.913 ,0.367)(0.916 ,0.369)
(0.919 ,0.372)(0.922 ,0.373)(0.925 ,0.377)(0.928 ,0.380)
(0.931 ,0.385)(0.934 ,0.387)(0.937 ,0.390)(0.940 ,0.393)
(0.943 ,0.398)(0.945 ,0.402)(0.948 ,0.407)(0.951 ,0.412)
(0.954 ,0.419)(0.957 ,0.421)(0.960 ,0.424)(0.963 ,0.427)
(0.966 ,0.432)(0.969 ,0.435)(0.972 ,0.440)(0.974 ,0.445)
(0.977 ,0.452)(0.980 ,0.455)(0.983 ,0.460)(0.986 ,0.465)
(0.989 ,0.472)(0.992 ,0.477)(0.994 ,0.484)(0.997 ,0.491)
(1 ,0)}-- (1 ,0) -- (0,0) -- cycle;%
\end{tikzpicture}
\caption{$|F(x)|$.}\label{fig3}
\end{figure}

We see that, up to an error of order $(\log_2n)^{-1/2}$, the total
variation distance is essentially asymptotic to the absolute
difference between the mean and $\lambda/2$. Finer approximations
will also be derived.

Four different proofs will be given for clarifying the total
variation distance, and each has its own generality; these include
an elementary probability approach, Stein's method, Fourier
analysis, and a new Krawtchouk-Parseval approach. Indeed, these
approaches easily extend to the consideration of more general
frameworks, a simple one being briefly considered that applies in
particular to the number of ones in the binary-reflected Gray codes.

\subsection{First moment of $X_n$}
The mean of $X_n$ is essentially the partial sum of $\nu_q(j)$
\[
    S_q(n) := n\mathbb{E}(X_n)
    =\sum_{0\le j<n} \nu_q(j),
\]
which, by the relation $\nu_q(qj+r)=\nu_q(j)+r$ for $0\le r<q$,
satisfies the following recurrence
\[
    S_q(n) = \sum_{1\le r\le q}
    S_q\left(\tr{\frac{n+r-1}{q}}\right)
    + \sum_{1\le r \le q} (q-r)\tr{\frac{n+r-1}{q}}
    \qquad(n\ge2),
\]
with $S_q(n)=0$ for $n\le 1$. In particular, when $q=2$, this
recurrence has the form
\[
    S_2(n) = S_2\left(\tr{\frac n2}\right) +
    S_2\left(\rd{\frac n2}\right) + \tr{\frac n2}.
\]
For many other recurrences for $S_2(n)$, see
\cite{mcilroy74a}. Interestingly, the quantity $S_2(n)$ appeared
naturally in a large number of concrete applications and is given as
\href{http://oeis.org/A000788}{A000788} in Sloane's
\emph{Encyclopedia of Integer Sequences}. A partial list when $q=2$
is given as follows.
\begin{itemize}
\item The number of bisecting strategies in certain games
\cite{gilbert62a};
\item Linear forms in number theory \cite{lindstrom65a};
\item Determinant of some matrix of order $n$ \cite{clements65a};
see also \cite{kano91a} for an extension to $q\ge2$;
\item Bounds for the number of edges in certain class of
graphs \cite{graham70a,hart76a,mehrabian13a};
\item The solution to the recurrence $f(n) =
\max_k\{f(k)+f(n-k)+\min\{k,\break n-k\}\}$ with $f(1)=0$ is exactly
$S_2(n)$; concrete instances where this recurrence arise can be
found in \cite[\S 2.2.1]{greene2008a} and \cite{mcilroy74a,hart76a};
see also \cite{agnarsson13a,mehrabian13a};
\item The number of comparators used by Batcher's bitonic sorting
network \cite{hong82a};
\item External left length of some binary trees \cite{li86a};
\item The minimum number of comparisons used by
    \begin{itemize}
    \item top-down recursive mergesort \cite{flajolet94b};
    \item bottom-up mergesort \cite{panny95a};
    \item queue-mergesort \cite{chen99a};
    \end{itemize}
\item The number of runs for the output sequence or recursive
mergesort with high erroneous comparisons; see \cite{hadjicostas11a}.
\end{itemize}
This list of concrete examples, albeit nonrandom in nature, shows
the richness and diversity of the sum-of-digits function.

Legendre, in his \emph{Th\'eorie des nombres} whose first edition
was published in 1798, derived the relation
\[
    \nu_q(n)
    = n - (q-1)\sum_{j\ge1}\left\lfloor
    \frac n{q^j}\right\rfloor;
\]
see \cite[Tome I, p.\ 12]{legendre1900a}. This relation has proved
useful in establishing many properties connected to $\nu_q(n)$,
including notably the identity (\ref{sum_of_nu}) below. On the other
hand, since $\sum_{j\ge1}\lfloor n/q^j\rfloor$ equals the $q$-adic
valuation of $n!$ (namely, the largest power of $q$ that divides
$n!$), the above relation has also been widely used in the $q$-adic
valuations of many famous numbers. For an extension of the
right-hand side, see \cite{prodinger82a}.

About nine decades later, d'Ocagne \cite{docagne86a} proved in 1886
an identity for $S_q(n)$ for $q=10$ (see also \cite[p.
457]{dickson1966a}); his identity easily extends to any base $q\ge2$
and can be rewritten as follows. Write $n=\sum_{0\le j\le k}\ve_j
q^j$, where $\ve_j = \ve_j(n)\in\{0,1,\dots,q-1\}$. Then d'Ocagne's
expression is identical to
\begin{equation}\label{docagne}
    \sum_{0\le j<n}\nu_q(j)
    = \sum_{0\le j\le k} \ve_j q^j \left(
    \frac{\ve_j -1+(q-1)j}{2}+ \sum_{j<\ell\le k}\ve_\ell
    \right).
\end{equation}
In particular, when $q=2$, we can write $n=\sum_{1\le j\le s}
2^{\lambda_j}$, where \mbox{$\lambda_1>\cdots>\lambda_s\ge0$}, and
(\ref{docagne}) has the alternative form
\begin{align} \label{docagne2}
    \sum_{0\le j<n}\nu_2(j)
    =\sum_{1\le j\le s} 2^{\lambda_j}
    \left(\frac{\lambda_j}2+j-1\right),
\end{align}
where $s=\nu_2(n)$. An extension of this expression can be found in
\cite{trollope67a,prodinger83a}. Since the proof of d'Ocagne's
expression is very simple (summing over all coefficients block by
block), it has remained almost unnoticed in the literature. Similar
expressions appeared and used in several later publications; see,
for example, \cite{bellman48a,cheo55a,trollope67a,shiokawa73a,
li86a,foster87a,laczay07a,hadjicostas11a}.

The first asymptotic result for $\mathbb{E}(X_n)$ was derived by
Bush \cite{bush40a} about half a century after d'Ocagne's 1886 paper
\cite{docagne86a}, and he proved that
\[
    \mathbb{E}(X_n)
    \sim \frac{q-1}{2}\,\log_q n,
\]
as $n\to\infty$, inspired by an expression derived earlier in
Bowden's book \cite{bowden1936a}\footnote{We were unable to find a
copy of this book.}. Note that, by (\ref{docagne}),
\[
    \mathbb{E}(X_{a q^k}) = \frac{(q-1)k+a-1}{2}
    \qquad(a=1,\dots,q-1).
\]
Bush proved his formula by providing upper and lower bounds for the
sum $\sum_{m<n}\ve_j(m)$ using the periodicity of $\ve_j$:
$\ve_j(m+q^{j+1})=\ve_j(m)$. In particular, when $q=2$, $\ve_j(m)$
is a sequence starting with a series of $2^j$ zeros followed by
$2^j$ ones. His estimates imply indeed a more precise result (see
Figure~\ref{fig4} for $q=2$)
\[
    \mathbb{E}(X_n)= \frac{q-1}{2}\,\log_qn+O(1),
\]
where the $O$-term is optimal. Note that this estimate can
also be derived easily from d'Ocagne's expression (\ref{docagne}) by
observing that the sum $\frac12(q-1)\sum_{0\le j\le \lambda}\ve_j j
q^j$ provides the major contribution, the others being of order
$O(n)$.

\begin{figure}
\begin{tikzpicture}[yscale=0.8]
\newcommand{\FONTSIZE}{\fontsize{8pt}{\baselineskip}\selectfont}
\draw[-latex] (-0.50 , 0) -- (5.5,0) node[right,yshift=0.3cm,xshift=-0.5cm]
{\scriptsize$\log_2 n$};%
\foreach \x/\xtext in { 0.000, 0.3125, ..., 5.000  }
\draw[shift={(\x,0.000)}] (0,0.025) -- (0,-0.025);%
\draw (0,0.05) -- (0,-0.05) node[below=4pt,left] {\FONTSIZE{$1$}}%
(0.625 ,0.05) -- (0.625 , -0.05) node[below] {\FONTSIZE{$2$}}%
(1.25 , 0.05) -- (1.25 , -0.05) node[below] {\FONTSIZE{$3$}}%
(1.875 , 0.05) -- (1.875 , -0.05) node[below] {\FONTSIZE{$4$}}%
(2.5 , 0.05) -- (2.5 , -0.05) node[below] {\FONTSIZE{$5$}}%
(3.125 , 0.05) -- (3.125 , -0.05) node[below] {\FONTSIZE{$6$}}%
(3.75 , 0.05) -- (3.75 , -0.05) node[below] {\FONTSIZE{$7$}}%
(4.375 , 0.05) -- (4.375, -0.05) node[below] {\FONTSIZE{$8$}}%
(5 , 0.05) -- (5 , -0.05) node[below] {\FONTSIZE{$9$}};%
\draw[-latex] (0, -0.50) -- (0, 5.5) node[right]
{
};%
\foreach \y/\ytext in { 0.000, 0.500, ..., 5.000  }
\draw[shift={(0, \y)}] (0.025, 0) -- (-0.025, 0);%
\foreach \y/\ytext in { 0.000, 0.100, ..., 5.000  }%
\draw[shift={(0, \y)}] (0.0125, 0) -- (-0.0125, 0);%
\draw (-0.05 , 1) -- (0.05 , 1) node[left] {\FONTSIZE{$0.04$}}%
(-0.05 , 2) -- (0.05 , 2) node[left] {\FONTSIZE{$0.08$}}%
(-0.05 , 3) -- (0.05 , 3) node[left] {\FONTSIZE{$0.12$}}%
(-0.05 , 4) -- (0.05 , 4) node[left] {\FONTSIZE{$0.17$}}%
(-0.05 , 5) -- (0.05 , 5) node[left] {\FONTSIZE{$0.21$}};%
\filldraw [top color=white,bottom color=gray,middle color=red!60]
plot coordinates{(0.000, 0.000)(0.366, 3.042)(0.625, 0.000)
(0.826, 3.892)(0.991, 3.042)(1.130, 2.852)(1.250, 0.000)
(1.356, 3.397)(1.451, 3.892)(1.537, 4.455)(1.616, 3.042)
(1.688, 3.818)(1.755, 2.852)(1.817, 2.098)(1.875, 0.000)
(1.930, 2.479)(1.981, 3.397)(2.030, 4.269)(2.076, 3.892)
(2.120, 4.742)(2.162, 4.455)(2.202, 4.227)(2.241, 3.042)
(2.277, 3.915)(2.313, 3.818)(2.347, 3.753)(2.380, 2.852)
(2.411, 2.868)(2.442, 2.098)(2.471, 1.396)(2.500, 0.000)
(2.528, 1.636)(2.555, 2.479)(2.581, 3.290)(2.606, 3.397)
(2.631, 4.165)(2.655, 4.269)(2.678, 4.380)(2.701, 3.892)
(2.723, 4.617)(2.745, 4.742)(2.766, 4.872)(2.787, 4.455)
(2.807, 4.602)(2.827, 4.227)(2.847, 3.875)(2.866, 3.042)
(2.884, 3.730)(2.902, 3.915)(2.920, 4.099)(2.938, 3.818)
(2.955, 4.010)(2.972, 3.753)(2.988, 3.511)(3.005, 2.852)
(3.021, 3.070)(3.036, 2.868)(3.052, 2.679)(3.067, 2.098)
(3.082, 1.937)(3.096, 1.396)(3.111, 0.877)(3.125, 0.000)
(3.139, 1.014)(3.153, 1.636)(3.166, 2.242)(3.180, 2.479)
(3.193, 3.064)(3.206, 3.290)(3.219, 3.513)(3.231, 3.397)
(3.244, 3.951)(3.256, 4.165)(3.268, 4.378)(3.280, 4.269)
(3.292, 4.481)(3.303, 4.380)(3.315, 4.284)(3.326, 3.892)
(3.337, 4.407)(3.348, 4.617)(3.359, 4.825)(3.370, 4.742)
(3.381, 4.949)(3.391, 4.872)(3.402, 4.799)(3.412, 4.455)
(3.422, 4.664)(3.432, 4.602)(3.442, 4.544)(3.452, 4.227)
(3.462, 4.178)(3.472, 3.875)(3.481, 3.580)(3.491, 3.042)
(3.500, 3.513)(3.509, 3.730)(3.518, 3.945)(3.527, 3.915)
(3.536, 4.127)(3.545, 4.099)(3.554, 4.074)(3.563, 3.818)
(3.571, 4.029)(3.580, 4.010)(3.588, 3.992)(3.597, 3.753)
(3.605, 3.741)(3.613, 3.511)(3.622, 3.287)(3.630, 2.852)
(3.638, 3.068)(3.646, 3.070)(3.653, 3.073)(3.661, 2.868)
(3.669, 2.876)(3.677, 2.679)(3.684, 2.487)(3.692, 2.098)
(3.699, 2.116)(3.707, 1.937)(3.714, 1.762)(3.721, 1.396)
(3.729, 1.230)(3.736, 0.877)(3.743, 0.529)(3.750, 0.000)
(3.757, 0.604)(3.764, 1.014)(3.771, 1.419)(3.778, 1.636)
(3.785, 2.032)(3.791, 2.242)(3.798, 2.451)(3.805, 2.479)
(3.811, 2.861)(3.818, 3.064)(3.824, 3.264)(3.831, 3.290)
(3.837, 3.487)(3.844, 3.513)(3.850, 3.539)(3.856, 3.397)
(3.862, 3.759)(3.869, 3.951)(3.875, 4.141)(3.881, 4.165)
(3.887, 4.353)(3.893, 4.378)(3.899, 4.403)(3.905, 4.269)
(3.911, 4.454)(3.917, 4.481)(3.923, 4.508)(3.928, 4.380)
(3.934, 4.408)(3.940, 4.284)(3.946, 4.162)(3.951, 3.892)
(3.957, 4.225)(3.962, 4.407)(3.968, 4.586)(3.973, 4.617)
(3.979, 4.794)(3.984, 4.825)(3.990, 4.855)(3.995, 4.742)
(4.001, 4.917)(4.006, 4.949)(4.011, 4.980)(4.016, 4.872)
(4.022, 4.904)(4.027, 4.799)(4.032, 4.694)(4.037, 4.455)
(4.042, 4.628)(4.047, 4.664)(4.052, 4.700)(4.057, 4.602)
(4.062, 4.640)(4.067, 4.544)(4.072, 4.450)(4.077, 4.227)
(4.082, 4.267)(4.087, 4.178)(4.092, 4.090)(4.097, 3.875)
(4.101, 3.790)(4.106, 3.580)(4.111, 3.373)(4.116, 3.042)
(4.120, 3.341)(4.125, 3.513)(4.130, 3.684)(4.134, 3.730)
(4.139, 3.899)(4.143, 3.945)(4.148, 3.990)(4.152, 3.915)
(4.157, 4.081)(4.161, 4.127)(4.166, 4.172)(4.170, 4.099)
(4.175, 4.145)(4.179, 4.074)(4.183, 4.003)(4.188, 3.818)
(4.192, 3.981)(4.196, 4.029)(4.201, 4.076)(4.205, 4.010)
(4.209, 4.057)(4.213, 3.992)(4.218, 3.928)(4.222, 3.753)
(4.226, 3.802)(4.230, 3.741)(4.234, 3.680)(4.238, 3.511)
(4.242, 3.453)(4.247, 3.287)(4.251, 3.122)(4.255, 2.852)
(4.259, 3.014)(4.263, 3.068)(4.267, 3.122)(4.271, 3.070)
(4.275, 3.124)(4.278, 3.073)(4.282, 3.022)(4.286, 2.868)
(4.290, 2.924)(4.294, 2.876)(4.298, 2.828)(4.302, 2.679)
(4.305, 2.633)(4.309, 2.487)(4.313, 2.342)(4.317, 2.098)
(4.321, 2.157)(4.324, 2.116)(4.328, 2.076)(4.332, 1.937)
(4.335, 1.898)(4.339, 1.762)(4.343, 1.627)(4.346, 1.396)
(4.350, 1.361)(4.354, 1.230)(4.357, 1.101)(4.361, 0.877)
(4.364, 0.750)(4.368, 0.529)(4.371, 0.311)(4.375, 0.000)
(4.379, 0.350)(4.382, 0.604)(4.386, 0.857)(4.389, 1.014)
(4.392, 1.264)(4.396, 1.419)(4.399, 1.574)(4.403, 1.636)
(4.406, 1.880)(4.410, 2.032)(4.413, 2.182)(4.416, 2.242)
(4.420, 2.392)(4.423, 2.451)(4.426, 2.510)(4.430, 2.479)
(4.433, 2.715)(4.436, 2.861)(4.440, 3.007)(4.443, 3.064)
(4.446, 3.208)(4.449, 3.264)(4.453, 3.320)(4.456, 3.290)
(4.459, 3.432)(4.462, 3.487)(4.465, 3.543)(4.469, 3.513)
(4.472, 3.568)(4.475, 3.539)(4.478, 3.510)(4.481, 3.397)
(4.484, 3.620)(4.487, 3.759)(4.491, 3.896)(4.494, 3.951)
(4.497, 4.087)(4.500, 4.141)(4.503, 4.194)(4.506, 4.165)
(4.509, 4.300)(4.512, 4.353)(4.515, 4.406)(4.518, 4.378)
(4.521, 4.430)(4.524, 4.403)(4.527, 4.376)(4.530, 4.269)
(4.533, 4.402)(4.536, 4.454)(4.539, 4.507)(4.542, 4.481)
(4.545, 4.533)(4.548, 4.508)(4.550, 4.482)(4.553, 4.380)
(4.556, 4.433)(4.559, 4.408)(4.562, 4.384)(4.565, 4.284)
(4.568, 4.261)(4.571, 4.162)(4.573, 4.064)(4.576, 3.892)
(4.579, 4.097)(4.582, 4.225)(4.585, 4.354)(4.587, 4.407)
(4.590, 4.534)(4.593, 4.586)(4.596, 4.639)(4.598, 4.617)
(4.601, 4.743)(4.604, 4.794)(4.607, 4.846)(4.609, 4.825)
(4.612, 4.876)(4.615, 4.855)(4.618, 4.835)(4.620, 4.742)
(4.623, 4.866)(4.626, 4.917)(4.628, 4.969)(4.631, 4.949)
(4.634, 5.000)(4.636, 4.980)(4.639, 4.961)(4.641, 4.872)
(4.644, 4.923)(4.647, 4.904)(4.649, 4.886)(4.652, 4.799)
(4.654, 4.781)(4.657, 4.694)(4.660, 4.609)(4.662, 4.455)
(4.665, 4.576)(4.667, 4.628)(4.670, 4.680)(4.672, 4.664)
(4.675, 4.716)(4.677, 4.700)(4.680, 4.685)(4.682, 4.602)
(4.685, 4.655)(4.687, 4.640)(4.690, 4.625)(4.692, 4.544)
(4.695, 4.530)(4.697, 4.450)(4.700, 4.371)(4.702, 4.227)
(4.705, 4.280)(4.707, 4.267)(4.710, 4.255)(4.712, 4.178)
(4.714, 4.166)(4.717, 4.090)(4.719, 4.014)(4.722, 3.875)
(4.724, 3.865)(4.726, 3.790)(4.729, 3.717)(4.731, 3.580)
(4.734, 3.508)(4.736, 3.373)(4.738, 3.238)(4.741, 3.042)
(4.743, 3.223)(4.745, 3.341)(4.748, 3.459)(4.750, 3.513)
(4.752, 3.630)(4.755, 3.684)(4.757, 3.738)(4.759, 3.730)
(4.761, 3.846)(4.764, 3.899)(4.766, 3.953)(4.768, 3.945)
(4.771, 3.998)(4.773, 3.990)(4.775, 3.983)(4.777, 3.915)
(4.780, 4.028)(4.782, 4.081)(4.784, 4.134)(4.786, 4.127)
(4.789, 4.179)(4.791, 4.172)(4.793, 4.165)(4.795, 4.099)
(4.797, 4.152)(4.800, 4.145)(4.802, 4.139)(4.804, 4.074)
(4.806, 4.068)(4.808, 4.003)(4.811, 3.939)(4.813, 3.818)
(4.815, 3.929)(4.817, 3.981)(4.819, 4.034)(4.821, 4.029)
(4.824, 4.081)(4.826, 4.076)(4.828, 4.071)(4.830, 4.010)
(4.832, 4.062)(4.834, 4.057)(4.836, 4.053)(4.838, 3.992)
(4.841, 3.988)(4.843, 3.928)(4.845, 3.868)(4.847, 3.753)
(4.849, 3.805)(4.851, 3.802)(4.853, 3.799)(4.855, 3.741)
(4.857, 3.738)(4.859, 3.680)(4.861, 3.623)(4.863, 3.511)
(4.865, 3.509)(4.867, 3.453)(4.869, 3.397)(4.872, 3.287)
(4.874, 3.231)(4.876, 3.122)(4.878, 3.014)(4.880, 2.852)
(4.882, 2.960)(4.884, 3.014)(4.886, 3.068)(4.888, 3.068)
(4.890, 3.122)(4.892, 3.122)(4.894, 3.122)(4.896, 3.070)
(4.898, 3.123)(4.900, 3.124)(4.901, 3.124)(4.903, 3.073)
(4.905, 3.074)(4.907, 3.022)(4.909, 2.971)(4.911, 2.868)
(4.913, 2.922)(4.915, 2.924)(4.917, 2.925)(4.919, 2.876)
(4.921, 2.878)(4.923, 2.828)(4.925, 2.779)(4.927, 2.679)
(4.929, 2.682)(4.930, 2.633)(4.932, 2.586)(4.934, 2.487)
(4.936, 2.440)(4.938, 2.342)(4.940, 2.245)(4.942, 2.098)
(4.944, 2.153)(4.946, 2.157)(4.947, 2.162)(4.949, 2.116)
(4.951, 2.121)(4.953, 2.076)(4.955, 2.031)(4.957, 1.937)
(4.959, 1.942)(4.960, 1.898)(4.962, 1.855)(4.964, 1.762)
(4.966, 1.719)(4.968, 1.627)(4.970, 1.536)(4.971, 1.396)
(4.973, 1.403)(4.975, 1.361)(4.977, 1.320)(4.979, 1.230)
(4.980, 1.190)(4.982, 1.101)(4.984, 1.012)(4.986, 0.877)
(4.988, 0.837)(4.989, 0.750)(4.991, 0.663)(4.993, 0.529)
(4.995, 0.444)(4.996, 0.311)(4.998, 0.179)(5.000, 0.000)}--
(5.000 ,0.000) -- (0 ,0.000) -- cycle;%
\end{tikzpicture}
\caption{$\frac12\log_2n-\mathbb{E}(X_n)$.}\label{fig4}
\end{figure}
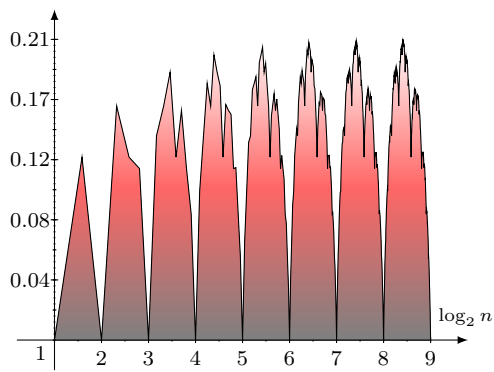

Bellman and Shapiro \cite{bellman48a} were primarily concerned with
the binary representation $q=2$ and provided an independent proof of
Bush's result
\[
    \mathbb{E}(X_n)
    = \frac{1}{2}\log_2n+O(\log\log n).
\]
They use two different proofs (one by generating functions and
Tauberian theorems and the other by recurrence) and briefly mention
in a footnote that the remainder can be improved to $O(1)$.

The same paper also initiated a very important notion called
``dyadically additive'', which has later on been fruitfully extended
and explored mostly under the name of $q$-additivity (and its
multiplicative counterpart $q$-multiplicativity); see
\cite{gelfond67a,delange72a,shiokawa74b} for the early publications
and \cite{mauduit05a} and the papers cited there for more recent
developments.

Mirsky \cite{mirsky49a}, following \cite{bellman48a}, proved that
\begin{equation}\label{mirsky}
    \mathbb{E}(X_n)
    = \frac{q-1}{2}\,\log_qn+O(1).
\end{equation}
His simple, half-page proof is based on the decompositions
\[
    \mathbb{E}(X_n)
    =\frac{1}{n}\sum_{0\le j<n}\nu_q(j)
    =\frac{1}{n}\sum_{0\le j<n}
    \sum_{\ell\ge0}\ve_\ell(j)
    =\frac{1}{n}\sum_{0\le r<q} r\sum_{\ell\ge1}f(n,\ell,r),
\]
where $f(n,\ell,r)$ denotes number of integers $0\le j<n$ such that
$\ve_\ell(j)=r$. Then (\ref{mirsky}) follows from the simple
estimate $f(n,\ell,r)=n/q+O(q^\ell)$.

Mirsky's result was independently re-derived by Cheo and Yien
\cite{cheo55a} and Tang \cite{tang63a} (judged to be virtually
identical to \cite{cheo55a} in MathSciNet), and referred to as Cheo
and Yien's theorem in \cite{cooper86a, kennedy91a}. Cheo and Yien
proved additionally in \cite{cheo55a} a theorem for the density of
$X_n$ of the form
\[
    \mathbb{P}\bigl(X_n=m\bigr)
    \sim \frac{1}{n}\cdot\frac{(\log_qn)^m}{m!},
\]
for each finite $m\ge0$.

Drazin and Griffith \cite{drazin52a} studied the sum of integer
powers of the digits and derived estimates similar to
(\ref{mirsky}). They also commenced the study of more precise
numerical bounds for the $O(1)$-term in (\ref{mirsky}), which was
followed later in \cite{clements65a,trollope68a,shiokawa73a,
mcilroy74a,foster87a,foster91a,foster92a,fang02a}. In particular, no
mention is made in \cite{clements65a,shiokawa73a,mcilroy74a,fang02a}
of known results for the $O(1)$-term in (\ref{mirsky}), and in
particular the bounds derived in \cite{fang02a} about half a century
later are weaker than those in \cite{drazin52a}.

The next stage of refinement was accomplished by Trollope in 1968
where he showed that the $O(1)$-term in (\ref{mirsky}) is indeed a
periodic function when $q=2$ for which an explicit expression is
also given. His proof is based on d'Ocagne's formula
(\ref{docagne2}), which he derived in \cite{trollope67a} in a more
general setting.

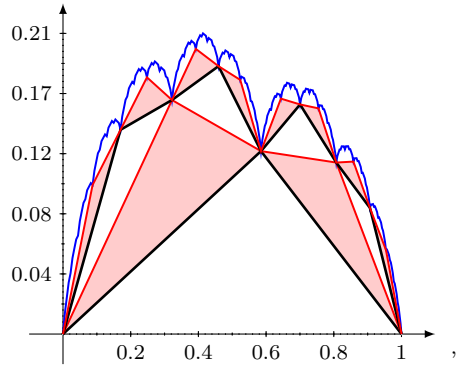
\begin{figure}[t!]
\begin{tikzpicture}[yscale=0.8,xscale=0.9]
\newcommand{\FONTSIZE}{\fontsize{8pt}{\baselineskip}\selectfont}
\draw[-latex] (-0.50 , 0.000) -- (5.50 , 0.000) node[right] {};%
\foreach \x/\xtext in { 0.000, 0.500, ..., 5.000  }%
\draw[shift={(\x,0.000)}] (0,0.025) -- (0,-0.025);%
\foreach \x/\xtext in { 0.000, 0.100, ..., 5.000  }%
\draw[shift={(\x,0.000)}] (0,0.0125) -- (0,-0.0125);%
\draw (1 , 0.05) -- (1 , -0.050) node[below] {\FONTSIZE{$0.2$}}%
(2 , 0.05) -- (2 , -0.050) node[below] {\FONTSIZE{$0.4$}}%
(3 , 0.05) -- (3 , -0.050) node[below] {\FONTSIZE{$0.6$}}%
(4 , 0.05) -- (4 , -0.050) node[below] {\FONTSIZE{$0.8$}}%
(5 , 0.05) -- (5 , -0.050) node[below] {\FONTSIZE{$1$}};%
\draw[-latex] (0.000 , -0.50) -- (0.000 , 5.50) node[right]
{
};%
\foreach \y/\ytext in { 0.000, 0.500, ..., 5.000  }%
\draw[shift={(0.000, \y)}] (0.025, 0) -- (-0.025, 0);%
\foreach \y/\ytext in { 0.000, 0.100, ..., 5.000  }%
\draw[shift={(0.000, \y)}] (0.0125, 0) -- (-0.0125, 0);%
\draw (0.050 , 1) -- (-0.050 , 1) node[left] {\FONTSIZE{$0.04$}}%
(0.050 , 2) -- (-0.050 , 2) node[left] {\FONTSIZE{$0.08$}}%
(0.050 , 3) -- (-0.050 , 3) node[left] {\FONTSIZE{$0.12$}}%
(0.050 , 4) -- (-0.050 , 4) node[left] {\FONTSIZE{$0.17$}}%
(0.050 , 5) -- (-0.050 , 5) node[left] {\FONTSIZE{$0.21$}};%
\filldraw[fill=red!20,draw=white]
(0,0) -- (1.610, 3.892) --(2.925, 3.042) -- cycle;
\filldraw[fill=red!20,draw=white]
(2.925, 3.042) -- (4.037, 2.852) --(5.000, 0.000) -- cycle;
\filldraw[fill=red!20,draw=white]
(0,0) -- (0.437, 2.479) --(0.850, 3.397) -- cycle;
\filldraw[fill=red!20,draw=white]
(0.850, 3.397) -- (1.240, 4.269) --(1.610, 3.892) -- cycle;
\filldraw[fill=red!20,draw=white]
(1.610, 3.892) -- (1.962, 4.742) --(2.297, 4.455) -- cycle;
\filldraw[fill=red!20,draw=white]
(2.297, 4.455) -- (2.618, 4.227) --(2.925, 3.042) -- cycle;
\filldraw[fill=red!20,draw=white]
(2.925, 3.042) -- (3.219, 3.915) --(3.502, 3.818) -- cycle;
\filldraw[fill=red!20,draw=white]
(3.502, 3.818) -- (3.774, 3.753) --(4.037, 2.852) -- cycle;
\filldraw[fill=red!20,draw=white]
(4.037, 2.852) -- (4.290, 2.868) --(4.534, 2.098) -- cycle;
\filldraw[fill=red!20,draw=white]
(4.534, 2.098) -- (4.771, 1.396) --(5.000, 0.000) -- cycle;
\definecolor{gen}{rgb}{0,0,1}%
\draw[gen, line width = 0.75pt] plot coordinates{(0.000, 0.000)
(0.028, 0.350)(0.056, 0.604)(0.084, 0.857)(0.112, 1.014)
(0.140, 1.264)(0.167, 1.419)(0.195, 1.574)(0.222, 1.636)
(0.249, 1.880)(0.276, 2.032)(0.303, 2.182)(0.330, 2.242)
(0.357, 2.392)(0.384, 2.451)(0.411, 2.510)(0.437, 2.479)
(0.464, 2.715)(0.490, 2.861)(0.516, 3.007)(0.543, 3.064)
(0.569, 3.208)(0.595, 3.264)(0.621, 3.320)(0.646, 3.290)
(0.672, 3.432)(0.698, 3.487)(0.723, 3.543)(0.749, 3.513)
(0.774, 3.568)(0.799, 3.539)(0.825, 3.510)(0.850, 3.397)
(0.875, 3.620)(0.900, 3.759)(0.924, 3.896)(0.949, 3.951)
(0.974, 4.087)(0.998, 4.141)(1.023, 4.194)(1.047, 4.165)
(1.072, 4.300)(1.096, 4.353)(1.120, 4.406)(1.144, 4.378)
(1.168, 4.430)(1.192, 4.403)(1.216, 4.376)(1.240, 4.269)
(1.263, 4.402)(1.287, 4.454)(1.310, 4.507)(1.334, 4.481)
(1.357, 4.533)(1.381, 4.508)(1.404, 4.482)(1.427, 4.380)
(1.450, 4.433)(1.473, 4.408)(1.496, 4.384)(1.519, 4.284)
(1.542, 4.261)(1.564, 4.162)(1.587, 4.064)(1.610, 3.892)
(1.632, 4.097)(1.655, 4.225)(1.677, 4.354)(1.699, 4.407)
(1.721, 4.534)(1.744, 4.586)(1.766, 4.639)(1.788, 4.617)
(1.810, 4.743)(1.832, 4.794)(1.853, 4.846)(1.875, 4.825)
(1.897, 4.876)(1.919, 4.855)(1.940, 4.835)(1.962, 4.742)
(1.983, 4.866)(2.004, 4.917)(2.026, 4.969)(2.047, 4.949)
(2.068, 5.000)(2.089, 4.980)(2.110, 4.961)(2.131, 4.872)
(2.152, 4.923)(2.173, 4.904)(2.194, 4.886)(2.215, 4.799)
(2.235, 4.781)(2.256, 4.694)(2.277, 4.609)(2.297, 4.455)
(2.318, 4.576)(2.338, 4.628)(2.358, 4.680)(2.379, 4.664)
(2.399, 4.716)(2.419, 4.700)(2.439, 4.685)(2.459, 4.602)
(2.479, 4.655)(2.499, 4.640)(2.519, 4.625)(2.539, 4.544)
(2.559, 4.530)(2.578, 4.450)(2.598, 4.371)(2.618, 4.227)
(2.637, 4.280)(2.657, 4.267)(2.676, 4.255)(2.696, 4.178)
(2.715, 4.166)(2.734, 4.090)(2.754, 4.014)(2.773, 3.875)
(2.792, 3.865)(2.811, 3.790)(2.830, 3.717)(2.849, 3.580)
(2.868, 3.508)(2.887, 3.373)(2.906, 3.238)(2.925, 3.042)
(2.944, 3.223)(2.962, 3.341)(2.981, 3.459)(3.000, 3.513)
(3.018, 3.630)(3.037, 3.684)(3.055, 3.738)(3.074, 3.730)
(3.092, 3.846)(3.110, 3.899)(3.129, 3.953)(3.147, 3.945)
(3.165, 3.998)(3.183, 3.990)(3.201, 3.983)(3.219, 3.915)
(3.237, 4.028)(3.255, 4.081)(3.273, 4.134)(3.291, 4.127)
(3.309, 4.179)(3.327, 4.172)(3.344, 4.165)(3.362, 4.099)
(3.380, 4.152)(3.397, 4.145)(3.415, 4.139)(3.433, 4.074)
(3.450, 4.068)(3.467, 4.003)(3.485, 3.939)(3.502, 3.818)
(3.520, 3.929)(3.537, 3.981)(3.554, 4.034)(3.571, 4.029)
(3.588, 4.081)(3.605, 4.076)(3.623, 4.071)(3.640, 4.010)
(3.657, 4.062)(3.674, 4.057)(3.690, 4.053)(3.707, 3.992)
(3.724, 3.988)(3.741, 3.928)(3.758, 3.868)(3.774, 3.753)
(3.791, 3.805)(3.808, 3.802)(3.824, 3.799)(3.841, 3.741)
(3.857, 3.738)(3.874, 3.680)(3.890, 3.623)(3.907, 3.511)
(3.923, 3.509)(3.940, 3.453)(3.956, 3.397)(3.972, 3.287)
(3.988, 3.231)(4.004, 3.122)(4.021, 3.014)(4.037, 2.852)
(4.053, 2.960)(4.069, 3.014)(4.085, 3.068)(4.101, 3.068)
(4.117, 3.122)(4.133, 3.122)(4.149, 3.122)(4.164, 3.070)
(4.180, 3.123)(4.196, 3.124)(4.212, 3.124)(4.227, 3.073)
(4.243, 3.074)(4.259, 3.022)(4.274, 2.971)(4.290, 2.868)
(4.305, 2.922)(4.321, 2.924)(4.336, 2.925)(4.352, 2.876)
(4.367, 2.878)(4.383, 2.828)(4.398, 2.779)(4.413, 2.679)
(4.428, 2.682)(4.444, 2.633)(4.459, 2.586)(4.474, 2.487)
(4.489, 2.440)(4.504, 2.342)(4.519, 2.245)(4.534, 2.098)
(4.549, 2.153)(4.564, 2.157)(4.579, 2.162)(4.594, 2.116)
(4.609, 2.121)(4.624, 2.076)(4.639, 2.031)(4.654, 1.937)
(4.668, 1.942)(4.683, 1.898)(4.698, 1.855)(4.713, 1.762)
(4.727, 1.719)(4.742, 1.627)(4.756, 1.536)(4.771, 1.396)
(4.786, 1.403)(4.800, 1.361)(4.814, 1.320)(4.829, 1.230)
(4.843, 1.190)(4.858, 1.101)(4.872, 1.012)(4.886, 0.877)
(4.901, 0.837)(4.915, 0.750)(4.929, 0.663)(4.943, 0.529)
(4.958, 0.444)(4.972, 0.311)(4.986, 0.179)(5.000, 0.000)};%
\definecolor{gen}{rgb}{0,0,0};%
\draw[gen, line width = 1pt] plot coordinates{(0.000, 0.000)
(2.925, 3.042)(5.000, 0.000)};%
\definecolor{gen}{rgb}{1,0,0};%
\draw[gen, line width = 0.75pt] plot coordinates{(0.000, 0.000)
(1.610, 3.892)(2.925, 3.042)(4.037, 2.852)(5.000, 0.000)};%
\definecolor{gen}{rgb}{0,0,0};%
\draw[gen, line width = 1pt] plot coordinates{(0.000, 0.000)
(0.850, 3.397)(1.610, 3.892)(2.297, 4.455)(2.925, 3.042)
(3.502, 3.818)(4.037, 2.852)(4.534, 2.098)(5.000, 0.000)};%
\definecolor{gen}{rgb}{1,0,0};%
\draw[gen, line width = 0.75pt] plot coordinates{(0.000, 0.000)
(0.437, 2.479)(0.850, 3.397)(1.240, 4.269)(1.610, 3.892)
(1.962, 4.742)(2.297, 4.455)(2.618, 4.227)(2.925, 3.042)
(3.219, 3.915)(3.502, 3.818)(3.774, 3.753)(4.037, 2.852)
(4.290, 2.868)(4.534, 2.098)(4.771, 1.396)(5.000, 0.000)};%
\end{tikzpicture}'
\caption{$-F_1(x)$: $q=2$.}\label{fig5}
\end{figure}

\begin{figure}[t!]
\begin{tikzpicture}[yscale=0.5,xscale=0.8]
\newcommand{\FONTSIZE}{\fontsize{8pt}{\baselineskip}\selectfont}
\draw[-latex] (-0.750 , 0.000) -- (5.50 , 0.000) node[right] {};
\foreach \x/\xtext in { 0.000, 0.500, ..., 5.000  }
    \draw[shift={(\x,0.000)}] (0,0.025) -- (0,-0.025);
\foreach \x/\xtext in { 0.000, 0.100, ..., 5.000  }
    \draw[shift={(\x,0.000)}] (0,0.0125) -- (0,-0.0125);
\draw (1 , 0.05) -- (1 , -0.050) node[below]
{\FONTSIZE{$0.2$}}(2 , 0.05) -- (2 , -0.050) node[below]
{\FONTSIZE{$0.4$}}(3 , 0.05) -- (3 , -0.050) node[below]
{\FONTSIZE{$0.6$}}(4 , 0.05) -- (4 , -0.050) node[below]
{\FONTSIZE{$0.8$}}(5 , 0.05) -- (5 , -0.050) node[below]
{\FONTSIZE{$1$}};%
\draw[-latex] (0.0 , -0.750) -- (0.000 , 5.750) node[right]
{
};
\foreach \y/\ytext in { 0.000, 0.500, ..., 5.000  }
    \draw[shift={(0.000, \y)}] (0.025, 0) -- (-0.025, 0);
\foreach \y/\ytext in { 0.000, 0.100, ..., 5.000  }
    \draw[shift={(0.000, \y)}] (0.0125, 0) -- (-0.0125, 0);
\draw (0.050 , 1) -- (-0.050 , 1) node[left] {\FONTSIZE{$0.1$}}
(0.050 , 2) -- (-0.050 , 2) node[left] {\FONTSIZE{$0.3$}}
(0.050 , 3) -- (-0.050 , 3) node[left] {\FONTSIZE{$0.4$}}
(0.050 , 4) -- (-0.050 , 4) node[left] {\FONTSIZE{$0.6$}}
(0.050 , 5) -- (-0.050 , 5) node[left] {\FONTSIZE{$0.8$}};%
\definecolor{gen}{rgb}{0,0,1}%
\draw[gen, line width = 0.75pt] plot coordinates{(0.000, 0.000)
(2.500, 1.667)(5.000, 0.000)};%
\draw[gen, line width = 0.75pt] plot coordinates{(0.000, 0.000)
(1.667, 2.222)(3.333, 2.222)(5.000, 0.000)};%
\draw[gen, line width = 0.75pt] plot coordinates{(0.000, 0.000)
(1.250, 2.500)(2.500, 3.333)(3.750, 2.500)(5.000, 0.000)};%
\draw[gen, line width = 0.75pt] plot coordinates{(0.000, 0.000)
(1, 2.667)(2.000, 4.000)(3.000, 4.000)(4.000, 2.667)
(5, 0)};%
\draw[gen, line width = 0.75pt] plot coordinates{(0.000, 0.000)
(0.833, 2.778)(1.667, 4.444)(2.500, 5.000)(3.333, 4.444)
(4.167, 2.778)(5.000, 0.000)};%
\end{tikzpicture}
\caption[]{$h(x): q=\begin{smallmatrix}6\\\vdots \\2\end{smallmatrix}$\ .}\label{fig6}
\end{figure}

Delange \cite{delange75a} made an important step towards the
ultimate understanding of the underlying periodic function. He
extended Trollope's result to any base $q\ge2$ and showed, by a very
simple, elegant, elementary proof, that (see Figure~\ref{fig4} for
$q=2$):
\begin{equation}\label{sum_of_nu}
    \mathbb{E}(X_n) - \frac{q-1}2\log_qn = F_1(\log_qn),
\end{equation}
where $F_1(x)=F_1(x+1)$ is a continuous, periodic, and
\emph{nowhere differentiable function} (see also
\cite{tenenbaum97a,thim2003a}).
His expression for $F_1$ is as follows; see Figure~\ref{fig5}
for a plot of $-F_1(x)$ and its first few approximations by
$\frac12\log_2n -\mathbb{E}(X_n)$.
\[
    F_1(x) = \frac{q-1}2\left(1-\{x\}\right)
    + q^{1-\{x\}} g(q^{-1+\{x\}}),
\]
where $\{x\}$ denotes the fractional part of $x$ and $g(x)$ is a
Takagi function \cite{hata84a,lagarias12a}
\[
    g(x) = \sum_{j\ge0} q^{-j}h(q^jx),
\]
with the $1$-periodic function $h$ defined by
(see Figure~\ref{fig6})
\[
    h(x) = \int_0^x
    \left(q\{t\}-\{qt\}-\frac{q-1}2\right) \dd t.
\]

Furthermore, the Fourier series expansion of $F$ is also computed;
see also \cite{flajolet94a} for a systematic approach by analytic
means. Delange's proof is based on the simple observation that
\begin{equation}\label{epsilon-j}
    \ve_j(n)
    =\left\lfloor\frac{n}{q^j}\right\rfloor
    -q\left\lfloor\frac{n}{q^{j+1}}\right\rfloor
    =\int_n^{n+1}\left(\left\lfloor\frac{t}{q^j}\right\rfloor
    -q\left\lfloor\frac{t}{q^{j+1}}\right\rfloor\right)\dd t.
\end{equation}
His paper \cite{delange75a} has since become a classic and has
stimulated much recent research on various themes related to digital
sums and different numeration systems; also different asymptotic
tools have been developed.

In particular, the Trollope-Delange formula (\ref{sum_of_nu}) for
$\mathbb{E}(X_n)$, which is not only an asymptotic expansion but
also an identity for all $n\ge1$, is not exceptional but a
distinguishing feature of many digital sums; see below and
\cite{flajolet94a,tenenbaum97a,grabner05a} for more examples.

\subsection{Beyond the mean: Variance, higher moments
and limit distribution of $X_n$}

The first paper dealing with the distribution of $X_n$ beyond the
mean value is by K\'atai and Mogyor\'odi \cite{katai68a} in 1968.
They derived the asymptotic normality of $X_n$ with a rate of the
form
\begin{align} \label{km-68}
    \sup_x\left|\mathbb{P}
    \left(\frac{X_n-\frac12(q-1)\log_qn}
    {\sqrt{\frac1{12}(q^2-1)\log_qn}}<x\right)
    -\Phi(x)\right| = O\left(\frac{\log\log n}
    {\sqrt{\log n}}\right),
\end{align}
where $\Phi$ denotes the standard normal distribution function and
the variance is implicit in their proof, namely,
\[
    \mathbb{V}(X_n) \sim \frac{q^2-1}{12}\,\log_qn.
\]
Their approach consists in decomposing $X_n$ into sums of suitable
number of independent random variables, each assuming the values
$\{0,1,\dots, q-1\}$ with equal probability. See (\ref{km-2}) below
for the binary case.

\begin{figure}
\includegraphics[width=5.4cm]{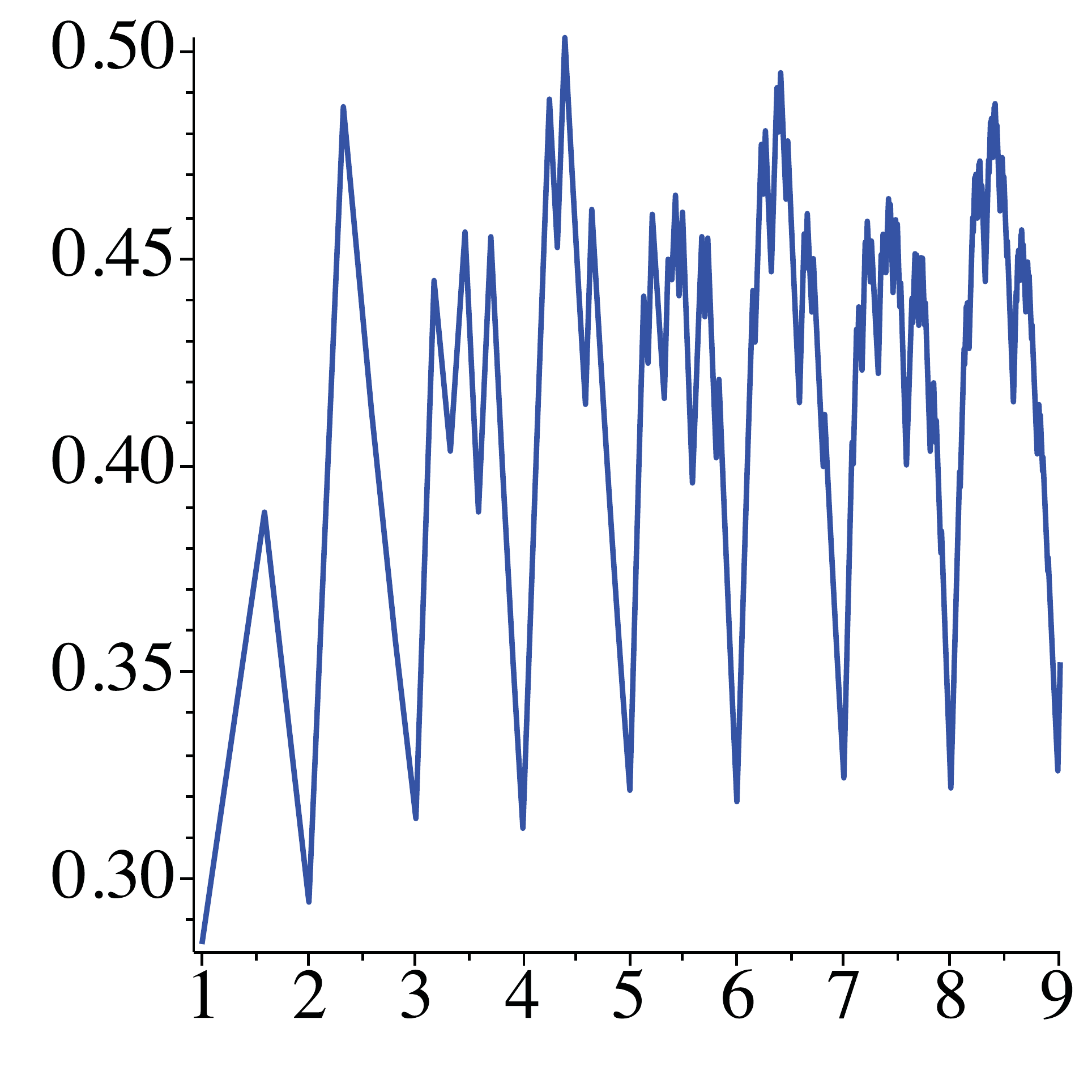}
\caption{(Kolmogorov dist.) $\sqrt{\log n}$.}\label{fig7}
\end{figure}

\begin{figure}
\begin{tikzpicture}[yscale=0.8]
\newcommand{\FONTSIZE}{\fontsize{8pt}{\baselineskip}\selectfont}
\draw[-latex] (-0.750 , 0.565) -- (5.50 , 0.565) node[right] {};%
\foreach \x/\xtext in { 0.000, 0.500, ..., 5.000  }%
   \draw[shift={(\x,0.565)}] (0,0.05) -- (0,-0.05);%
\foreach \x/\xtext in { 0.000, 0.100, ..., 5.000  }%
   \draw[shift={(\x,0.565)}] (0,0.0125) -- (0,-0.0125);%
\draw (1,0.465) -- (1,0.665) node[below=3pt] {\FONTSIZE{$0.2$}}%
(2 , 0.465) -- (2 , 0.665) node[below=3pt] {\FONTSIZE{$0.4$}}%
(3 ,0.465) -- (3 , 0.665) node[below=3pt] {\FONTSIZE{$0.6$}}%
(4 ,0.465) -- (4 , 0.665) node[below=3pt] {\FONTSIZE{$0.8$}}%
(5 ,0.465) -- (5 , 0.665) node[below=3pt] {\FONTSIZE{$1$}};%
\draw[-latex] (0.000 , -0.750) -- (0.000 , 5.750)
node[right]{
};%
\foreach \y/\ytext in { 0.000, 0.147, ...,5.000  }%
   \draw[shift={(0.000, \y)}] (0.025, 0) -- (-0.025, 0);%
\draw (.05,-.147 ) -- (-.05,-.147) node[left] {\FONTSIZE{$-0.05$}}%
(0.050 , 1.324 ) -- (-0.05, 1.324 ) node[left] {\FONTSIZE{$0.05$}}%
(0.050 , 2.059 ) -- (-0.050 , 2.059 ) node[left] {\FONTSIZE{$0.10$}}%
(0.050 , 2.794) -- (-0.050 , 2.794) node[left] {\FONTSIZE{$0.15$}}%
(0.050 , 3.529) -- (-0.050 , 3.529) node[left] {\FONTSIZE{$0.20$}}%
(0.050 , 4.265) -- (-0.050 , 4.265) node[left] {\FONTSIZE{$0.25$}}%
(0.050 , 5) -- (-0.050 , 5) node[left] {\FONTSIZE{$0.30$}};%
\definecolor{gen}{rgb}{0,0,1}%
\filldraw [top color=blue!40, bottom color=blue!80]
plot coordinates{(0.000, 0.565)
(0.028, 0.178)(0.056, 0.088)(0.084, 0.000)(0.112, 0.084) (0.140,
0.000)(0.167, 0.085)(0.195, 0.170)(0.222, 0.307) (0.249,
0.226)(0.276, 0.311)(0.303, 0.396)(0.330, 0.530) (0.357,
0.614)(0.384, 0.746)(0.411, 0.877)(0.437, 0.945) (0.464,
0.869)(0.490, 0.951)(0.516, 1.032)(0.543, 1.161) (0.569,
1.242)(0.595, 1.368)(0.621, 1.494)(0.646, 1.557) (0.672,
1.636)(0.698, 1.759)(0.723, 1.882)(0.749, 1.944) (0.774,
2.065)(0.799, 2.126)(0.825, 2.187)(0.850, 2.085) (0.875,
2.009)(0.900, 2.083)(0.924, 2.158)(0.949, 2.276) (0.974,
2.349)(0.998, 2.466)(1.023, 2.583)(1.047, 2.640) (1.072,
2.712)(1.096, 2.826)(1.120, 2.940)(1.144, 2.996) (1.168,
3.108)(1.192, 3.164)(1.216, 3.220)(1.240, 3.122) (1.263,
3.189)(1.287, 3.299)(1.310, 3.409)(1.334, 3.464) (1.357,
3.572)(1.381, 3.627)(1.404, 3.682)(1.427, 3.588) (1.450,
3.695)(1.473, 3.750)(1.496, 3.805)(1.519, 3.715) (1.542,
3.771)(1.564, 3.684)(1.587, 3.598)(1.610, 3.279) (1.632,
3.198)(1.655, 3.258)(1.677, 3.318)(1.699, 3.423) (1.721,
3.482)(1.744, 3.586)(1.766, 3.690)(1.788, 3.745) (1.810,
3.803)(1.832, 3.905)(1.853, 4.006)(1.875, 4.061) (1.897,
4.161)(1.919, 4.215)(1.940, 4.269)(1.962, 4.188) (1.983,
4.243)(2.004, 4.341)(2.026, 4.439)(2.047, 4.492) (2.068,
4.590)(2.089, 4.642)(2.110, 4.695)(2.131, 4.618) (2.152,
4.713)(2.173, 4.766)(2.194, 4.819)(2.215, 4.745) (2.235,
4.798)(2.256, 4.725)(2.277, 4.654)(2.297, 4.374) (2.318,
4.422)(2.338, 4.516)(2.358, 4.610)(2.379, 4.665) (2.399,
4.758)(2.419, 4.812)(2.439, 4.866)(2.459, 4.800) (2.479,
4.891)(2.499, 4.946)(2.519, 5.000)(2.539, 4.936) (2.559,
4.990)(2.578, 4.928)(2.598, 4.866)(2.618, 4.608) (2.637,
4.699)(2.657, 4.755)(2.676, 4.811)(2.696, 4.754) (2.715,
4.811)(2.734, 4.755)(2.754, 4.700)(2.773, 4.456) (2.792,
4.514)(2.811, 4.463)(2.830, 4.412)(2.849, 4.176) (2.868,
4.129)(2.887, 3.897)(2.906, 3.669)(2.925, 3.182) (2.944,
3.090)(2.962, 3.129)(2.981, 3.168)(3.000, 3.257) (3.018,
3.296)(3.037, 3.385)(3.055, 3.474)(3.074, 3.535) (3.092,
3.573)(3.110, 3.661)(3.129, 3.748)(3.147, 3.808) (3.165,
3.894)(3.183, 3.954)(3.201, 4.013)(3.219, 3.972) (3.237,
4.008)(3.255, 4.093)(3.273, 4.177)(3.291, 4.236) (3.309,
4.320)(3.327, 4.378)(3.344, 4.436)(3.362, 4.395) (3.380,
4.478)(3.397, 4.536)(3.415, 4.593)(3.433, 4.554) (3.450,
4.612)(3.467, 4.574)(3.485, 4.536)(3.502, 4.333) (3.520,
4.364)(3.537, 4.445)(3.554, 4.526)(3.571, 4.583) (3.588,
4.663)(3.605, 4.721)(3.623, 4.778)(3.640, 4.743) (3.657,
4.822)(3.674, 4.879)(3.690, 4.936)(3.707, 4.902) (3.724,
4.959)(3.741, 4.926)(3.758, 4.893)(3.774, 4.703) (3.791,
4.780)(3.808, 4.838)(3.824, 4.895)(3.841, 4.865) (3.857,
4.923)(3.874, 4.893)(3.890, 4.864)(3.907, 4.682) (3.923,
4.740)(3.940, 4.714)(3.956, 4.688)(3.972, 4.510) (3.988,
4.486)(4.004, 4.312)(4.021, 4.140)(4.037, 3.754) (4.053,
3.777)(4.069, 3.853)(4.085, 3.929)(4.101, 3.989) (4.117,
4.064)(4.133, 4.125)(4.149, 4.185)(4.164, 4.165) (4.180,
4.239)(4.196, 4.299)(4.212, 4.358)(4.227, 4.339) (4.243,
4.398)(4.259, 4.380)(4.274, 4.361)(4.290, 4.202) (4.305,
4.274)(4.321, 4.334)(4.336, 4.394)(4.352, 4.377) (4.367,
4.437)(4.383, 4.421)(4.398, 4.405)(4.413, 4.252) (4.428,
4.312)(4.444, 4.298)(4.459, 4.284)(4.474, 4.134) (4.489,
4.122)(4.504, 3.975)(4.519, 3.830)(4.534, 3.490) (4.549,
3.560)(4.564, 3.622)(4.579, 3.684)(4.594, 3.676) (4.609,
3.737)(4.624, 3.729)(4.639, 3.722)(4.654, 3.584) (4.668,
3.646)(4.683, 3.640)(4.698, 3.634)(4.713, 3.499) (4.727,
3.495)(4.742, 3.363)(4.756, 3.232)(4.771, 2.916) (4.786,
2.980)(4.800, 2.978)(4.814, 2.977)(4.829, 2.852) (4.843,
2.852)(4.858, 2.728)(4.872, 2.606)(4.886, 2.304) (4.901,
2.307)(4.915, 2.189)(4.929, 2.072)(4.943, 1.777)
(4.958, 1.663)(4.972, 1.373)(4.986, 1.085)(5.000, 0.565)};%
\end{tikzpicture}
\caption{$-F_2(x)$.}\label{fig8}
\end{figure}

About a decade later, Diaconis \cite{diaconis77a}
obtained, by Stein's method, an optimal Berry-Esseen bound for $q=2$
of the form
\begin{align}\label{diaconis-77}
\begin{split}
    &\sup_x\left|\mathbb{P}
    \left(\frac{X_{n}-\frac12\log_2n}
    {\sqrt{\frac1{4}\log_2n}}<x\right)
    -\Phi(x)\right| = O\left(\frac{1}
    {\sqrt{\log n}}\right);
\end{split}
\end{align}
(see Figure~\ref{fig7}) he also proved that ($q=2$)
\[
    \mathbb{V}(X_n)
    = \frac{\log_2n}4+O\left(\sqrt{\log n}\right).
\]

\paragraph{Moments of $X_n$.}
Stolarsky \cite{stolarsky77a}, in addition to giving a wide list of
references, carried out a systematic study of the asymptotics of the
moments of $X_n$ when $q=2$; in particular, he proved that
\begin{equation}
\label{stolarsky}
    \mathbb{E}(X_n^m)
    =\frac{1}{n}\sum_{0\le j<n}\nu_2(j)^m
    =\left(\frac{\log_2n}{2}\right)^m+O(\log_2^{m-1}n),
\end{equation}
for any positive integer $m$. The $O$-term is however too weak to
obtain a more precise asymptotic approximation to the central
moments of $X_n$ of order $\ge2$.

Later Coquet \cite{coquet86a} in 1986 improved Stolarsky's result by
providing a formula of the Trollope-Delange type ($q=2$)
\begin{equation}
\label{coquet}
    \mathbb{E}(X_n)^m
    =\left(\frac{\log_2n}{2}\right)^m+\sum_{0\le j<m}
    (\log_2n)^jF_{m,j}(\log_2n);
\end{equation}
here $F_{m,j}$ are bounded, continuous, $1$-periodic functions.
Coquet's method of proof starts from defining
\[
    S_2^{[m]}(n) := \sum_{0\le j<n}\nu_2(j)^m,
\]
and shows that the quantity $S_2^{[m]}(2n)-2S_2^{[m]}(n)$ is
expressible in terms of a sum of $S_2^{[j]}(n)$ with $j<m$; then an
induction is used. In particular, his result for the second moment
implies the identity
\[
   \mathbb{V}(X_n) = \frac{\log_2n}4+F_2(\log_2n),
\]
where $F_2$ is bounded, continuous and periodic of period $1$; 
see Figure~\ref{fig8}.
Coquet \cite{coquet86a} mentioned that the function $F_2$ is nowhere
differentiable and his proofs extend to any $q$-ary base. An
independent proof of the above identity \`a\ la Delange was given
later by Kirschenhofer \cite{kirschenhofer90a}; see also Osbaldestin
\cite{osbaldestin91a} for an interesting discussion of several
digital sums, as well as an alternative expression
for~$F_2$.\looseness=1

On the other hand, Coquet's expressions for the $F_{m,j}$'s (except
for $F_2$) are nonconstructive; see \cite{grabner93b} for the third
moment. Dumont and Thomas \cite{dumont93a} studied the moments of
$X_n$ in a general framework and derived more explicit expressions
for $F_{m,j}$, as well as properties such as continuity and nowhere
differentiability. Their approach relies on substitutions on finite
alphabet and matrix analysis; see \cite{dumont89a}. In addition to
the moments, they also considered in the same paper \cite{dumont93a}
the moments of $X_n-\frac12(q-1)\log_qn$ and showed that
\begin{equation}\label{hi-mm}
\begin{split}
    \mathbb{E}\left(X_n-\frac{q-1}{2}\log_qn\right)^m
    &=\frac{1+(-1)^m}{2}\cdot\frac{m!}{(m/2)!2^{m/2}}
    \left(\frac{q^2-1}{12}\,\log_qn\right)^{m/2}\\
    &\qquad\qquad
    +\sum_{0\le j<m/2}(\log_q n)^j\tilde{F}_{m,j}(\log_qn)+o(1),
\end{split}
\end{equation}
where the $\tilde{F}_{m,j}$'s are continuous, $1$-periodic, and
nowhere differentiable functions. These estimates imply of course
the asymptotic normality of $X_n$ by the method of moments, which
Dumont and Thomas later established in \cite{dumont97a} (in a more
general framework).

Other constructive expressions, together with interesting functional
properties, are derived by Okada et al.\ \cite{okada95a}, based on
binomial measures; see also \cite{muramoto00a} and the recent paper
\cite{kruppel09a}. Extensions of the same approach to cover the
moments of $X_n$ for any $q\ge2$ were carried out in
\cite{muramoto00a,muramoto03a}, the required tools being developed
in \cite{okada96a}.

Unaware of Stolarsky's and Coquet's results, Kennedy and Cooper
considered the cases when $q=10$: $m=2$ in \cite{kennedy91a} and any
positive integer $m$ in \cite{cooper92a} but with a non-optimal
error term in the corresponding expression of (\ref{stolarsky}) for
$q=10$; see also \cite{brown94a}. The optimal error term follows
indeed from Dumont and Thomas's result in \cite{dumont93a} (see also
\cite{muramoto00a}) and was later re-proved by Yu in \cite{yu96a}
(see also \cite{chen04a} for an extension).

A general procedure, based on the classical approach of Dirichlet
series and Mellin-Perron integral formula (fully discussed in
\cite{flajolet94a}), was developed in \cite{grabner05a} and leads to
absolutely convergent Fourier series expansions for $G_{m,j}$. The
approach there can be easily extended to $q$-ary case.

\paragraph{Probability generating function of $X_n$.}
By definition, the probability generating function of $X_n$ is given
by
\[
    \mathbb{E}\left(y^{X_n}\right)
    = \frac1n \sum_{0\le j<n} y^{\nu_q(j)}.
\]
The special cases when $q=y=2$ appeared as the total number of odd
numbers of $\binom{j}{i}$ for $0\le i\le j<n$, a result derived by
Glaisher \cite{glaisher99a} in 1899; earlier results of similar
character can be found in the papers by Kummer \cite{kummer52a} and
by Lucas \cite{lucas78a}. For another interesting occurrence in
cellular automata, see \cite{ettestad10a,wolfram83a,wolfram84a}.

The distribution of $X_n$ is closely connected to the notion of
$q$-additive and $q$-multiplicative functions, first introduced by
Bellman and Shapiro \cite{bellman48a}, and later systematically
investigated by Gel'fond \cite{gelfond67a} and Delange
\cite{delange72a}; see also \cite{barat06a,mauduit05a} and the
references cited there. We did not find a more complete survey on
$q$-additive or $q$-multiplicative functions but a simple search on
MathSciNet resulted in more than 152 papers (as of July 1,
2014); see \cite[Ch.\ 9]{berthe-rigo2010a} and \cite{barat06a}.

A function $f:\mathbb{N}\to \mathbb{C}$ is said to be
$q$-\emph{multiplicative} if
\[
    f\left(aq^r +b\right)
    = f(a q^r) f(b),
\]
for $1\le a\le q-1$ and $0\le b<q^r$, $r\ge1$. This implies that
$f(0)=1$. Similarly, one defines $q$-additive functions by
$f\left(aq^r +b\right)= f(a q^r) +f(b)$. By definition, one then
obtains, for a $q$-multiplicative function $f$ (see
\cite{gelfond67a,delange72a})
\[
    \sum_{0\le j<n}f(j)
    = \sum_{0\le j\le \lambda}\Biggl(\prod_{0\le r<j}\Biggl(1+
    \sum_{1\le \ell<q}f(\ell q^r)\Biggr)\Biggr) \Biggl(
    \prod_{j<r\le \lambda} f (\ve_rq^r)\Biggr)
    \sum_{0\le \ell<\ve_j}f(\ell q^j),
\]
where $n=\sum_{0\le j\le \lambda} \ve_j q^j$. Now taking
$f(n)=y^{\nu_q(n)}$, which is obviously a $q$-multiplicative
function, we obtain, by re-grouping nonzero summands,
\begin{equation}
\begin{split}\label{GBGF}
    \mathbb{E}\left(y^{X_n}\right)
    = \frac1n\sum_{1\le j\le s} y^{c_1+\cdots+c_{j-1}}
    \left(1+y+\cdots+y^{c_j-1}\right)
    \left(1+y+\cdots+y^{q-1}\right)^{\lambda_j} ,
\end{split}
\end{equation}
where
\[
    n=c_1q^{\lambda_1}+c_2q^{\lambda_2}+\cdots+c_sq^{\lambda_s},
\]
with $\lambda_1>\cdots>\lambda_s\ge0$ and $c_j\in\{1,\dots,q-1\}$.
The closed-form expression (\ref{GBGF}) was later derived and stated
explicitly by Stein \cite{stein86a}.

Special cases of (\ref{GBGF}) appeared in Roberts \cite{roberts57a}
for $q=y=2$ (later re-derived in \cite{stolarsky77a}), and in Stein
\cite{stein80a} for $q=2$, which has the form
\begin{align}\label{BGF}
    \mathbb{E}(y^{X_n})
    = \frac1n\sum_{1\le j\le s}y^{j-1}(1+y)^{\lambda_j},
\end{align}
when $n=2^{\lambda_1}+2^{\lambda_2}+\cdots+2^{\lambda_s}$, where
$\lambda_1>\lambda_2>\cdots>\lambda_s$.

In the same paper \cite{stein86a}, Stein also obtained many bounds
for the exponential sum (\ref{GBGF}); in particular, the function
\[
    G(\log_qn;y):=\frac{\mathbb{E}(y^{X_n})}
    {n^{\log_q(1+y+\cdots +y^{q-1})}}\qquad(y>0)
\]
is bounded and periodic ($G(x;y)=G(x+1;y)$).

Okada et al. \cite{okada95b,muramoto00a} later gave more explicit
expressions for the periodic function $G$ by multinomial measures. A
different approach was provided in \cite{kruppel09a}. A Fourier
expansion for $q=2$ was given in \cite{grabner05a}, which is
absolutely convergent when $\sqrt{2}-1<y<\sqrt{2}+1$.

The closed-form expression (\ref{GBGF}) contains much information;
for example, the d'Ocagne's formula (\ref{docagne}) follows from
(\ref{GBGF}) by taking derivative with respect to $y=1$ and then
substituting $y=1$. We will see later that (\ref{GBGF}) is also
helpful in proving effective approximations for distances between
$X_n$ and some binomials.

For other approaches to $q$-additive and $q$-multiplicative
functions, see \cite{mauclaire83a,mauclaire83b,murata88a,
grabner93a,manstavicius97a,drmota01a,alkauskas04a}.

\subsection{Asymptotic distribution of sum-of-digits function}

We mentioned K\'atai and Mogyor\'odi's \cite{katai68a} and
Diaconis's \cite{diaconis77a} Berry-Esseen bounds for $X_n$. We
group here known results concerning limit and approximation theorems
for $X_n$ according to the major approach used, focusing mostly on
the case $q=2$ for simplicity of presentation and comparison. See
Table~\ref{tab1} for a summary.

\begin{table}[!b]
\tabcolsep=5pt
\caption{A summary of known approaches leading to the asymptotic normality of $X_n$;
here CLT denotes ``central limit theorem'' and LLT ``local limit theorem''}\label{tab1}
\begin{tabular}{|c|c||c|c|c|}
\hline
Authors \& Papers & Year & Results & Approach & Notes \\
\hline\hline
K\'atai \& Mogyor\'odi \cite{katai68a} & 1968 & CLT$+$rate & Elementary & $q$-ary  \\
\hline
Diaconis \cite{diaconis77a} & 1977 & CLT$+$rate & Stein's method & binary \\
\hline
Schmidt \cite{schmidt83a} & 1983 & Multivariate CLT & Probabilistic & binary \\
\hline
Schmid \cite{schmid84a} & 1984 & $\begin{array}{c}\text{Multivariate LLT} \\+ \text{rate} \end{array} $ & $\begin{array}{c} \text{Matrix GF}\\ \text{Markov chain}\end{array}$  &  binary \\
\hline
Stein \cite{stein1986a} & 1986 & $\begin{array}{c}\text{Binomial}\\ \text{approximation} \end{array}$ & Stein's method & binary \\
\hline
Dumont \& Thomas \cite{dumont93a} & 1992 & CLT & Method of moments & general \\
\hline
Loh \cite{loh92a} & 1992 &  $\begin{array}{c}\text{Multinomial}\\ \text{approximation} \end{array}$ & Stein's method & $q$-ary \\
\hline
Barbour \& Chen \cite{barbour92a} & 1992 & $\begin{array}{c}\text{Approximation}\\ \text{by a mixture}\\ \text{of binomial} \end{array}$ & Stein's method & binary \\
\hline
Grabner \cite{grabner93a} & 1993 & (implicit) & Mellin transform & $q$-additive  \\
\hline
Bassily \& K\'atai \cite{bassily95a} & 1995 & CLT & Method of moments & $q$-additive \\
\hline
Manstavi\v{c}ius \cite{manstavicius97a} & 1997 & Functional CLT & Probabilistic & $q$-additive \\
\hline
Dumont \& Thomas \cite{dumont97a} & 1997 & LLT$+$rate & Markov chain & general \\
\hline
Drmota \& Gajdosik \cite{drmota98a} & 1998 & LLT$+$ rate & Generating function & general \\
\hline
Drmota et al.\ \cite{drmota03a} & 2003 & Functional CLT & Probabilistic & $q$-ary  \\
\hline
\end{tabular}
\end{table}

\subsubsection{Classical probabilistic approach}

K\'{a}tai and Mogyor\'{o}di's approach uses elementary probability
tools and relies their Berry-Esseen bound (\ref{km-68}) on the
following decomposition (for $q=2$)
\begin{equation}\label{km}
    \mathbb{P}\bigl(X_n=\ell\bigr)
    =\frac1n\sum_{1\le j\le s}2^{\lambda_j}
    \mathbb{P}(Y_{\lambda_j}=\ell-j+1),
\end{equation}
which follows immediately from (\ref{BGF}). Here $Y_j$ denotes the
sum of $j$ independent Bernoulli indicators, each assuming $0$ and
$1$ with equal probability $1/2$. The identity (\ref{km}) implies
that the random variable $X_n$ is itself a mixture of independent
binomial distributions. The remaining proof then proceeds along
standard classical lines (by using estimates for sums of independent
random variables).

Heppner \cite{heppner76a} later proved, in the same spirit, a simple
Chernoff-type inequality for $X_n$ when $q=2$ ($\lambda
=\tr{\log_2n}$)
\[
    \mathbb{P}(|X_n-(\lambda+1)/2|>C)
    \le2\mathbb{P}\bigl(|Y_{\lambda+1}
    -(\lambda+1)/2|>C\bigr);
\]
since the right-hand side of this inequality decreases exponentially
as $C$ grows, one concludes that $\nu_2(m)$ is close to $(\lambda
+1)/2$ for most $m< n$. This observation is useful in establishing
precise estimates for sums of the form $\sum_{m\in B_n}\nu_2(m)$,
where $B_n$ is an arbitrary subset of nonnegative integers $<n$. For
results concerning the distribution of $\nu_q(n)$ for given
subsequences of integers (such as prime numbers and squares), see
\cite{mauduit09a,mauduit10a} and the references therein. Similar
estimates will be used below.

A central limit theorem for the distribution of the values assumed
by the sequence $\nu_2(3n)-\nu_2(n)$ was derived by K\'atai
\cite{katai77a}, while the corresponding local limit theorem was
given independently by Stolarsky \cite{stolarsky80a}. The proof of
Stolarsky's local limit theorem starts from matrix generating
functions, obtaining a closed-form expression, and then applies the
saddle-point method for the corresponding sum. Schmidt
\cite{schmidt83a} then proved, motivated by Stolarsky's
\cite{stolarsky80a} result, a multidimensional central limit theorem
(the joint distribution of the values of $\nu_2(K_1n),\dots,
\nu_2(K_dn)$ for odd numbers $K_1,\dots,K_d$) using tools from
Markov chains. The intuition behind such a limit law is that the two
events $\nu_2(K_1U_n)$ and $\nu_2(K_2U_n)$ are more or less
independent, where $U_n \sim \text{Uniform}[0,n-1]$ and $K_1, K_2$
are odd numbers.

Dumont and Thomas \cite{dumont97a} use again Markov chains and large
deviations to characterize the asymptotic distribution of a class of
digital sums (covering in particular $X_n$) associated with
substitutions, a Berry-Esseen bound being also derived.

\subsubsection{$q$-multiplicative functions}

Since $\nu_q(n)$ is $q$-additive, the function $e^{it\nu_q(n)}$ is
$q$-multiplicative. The distribution of the values of $q$-additive
functions has been widely studied in the number-theoretic\vadjust{\goodbreak}
literature. We mention briefly an early result. Delange
\cite{delange72a} showed that
\[
    \frac1n\sum_{0\le m<n}f(m)
    = \prod_{1\le r\le \tr{\log_qn}}\frac{1+f(q^r)+\cdots
    + f((q-1)q^r)}{q} +o(1),
\]
for any $q$-multiplicative function $f$ with $|f|\le 1$ and (see
\cite{mauclaire93a})
\[
    \lim_{k\to\infty}\prod_{r_0\le r\le k}
    \left|\frac{1+f(q^r)+\cdots+ f((q-1)q^r)}{q}\right|>0.
\]
This result roughly says that \emph{the mean value of
$q$-multiplicative functions with bounded modulus is close to some
multinomial distribution.}

In particular, if one applies \emph{formally} this result to
$f(n)=e^{it\nu_q(n)}$, then the left-hand side corresponds to the
characteristic function of $X_n$, while the dominant term on the
right-hand side to a multinomial distribution. We cannot however
conclude directly from this result that $X_n$ is asymptotically
multinomially distributed due to lack of uniformity in $t$. For
asymptotic normality and related results for $q$-additive functions,
see \cite{bassily95a,bassily96a,manstavicius97a,steiner2002a} and
\cite[Ch.\ 9]{berthe-rigo2010a}.

\subsubsection{Stein's method}

Stein's method is a method of probability approximation invented by
Charles Stein in 1972 \cite{stein72a}. It does not involve Fourier
analysis but hinges on the solution of a functional equation. In a
nutshell, Stein's method can be described as follows. Let $W$ and
$Z$ be random variables. In approximating the distribution
$\mathscr{L}(W)$ of $W$ by the distribution $\mathscr{L}(Z)$ of $Z$,
the difference between $\mathbb{E}(h(W))$ and $\mathbb{E}(h(Z))$ for
a class of functions $h$ is expressed as
\[
    \mathbb{E}(h(W)) - \mathbb{E}(h(Z)) = \mathbb{E}\left(
    L[f_h](W)\right),
\]
where $L$ is a linear operator and $f_h$ a bounded solution of the
equation
\[
    L[f] = h - \mathbb{E}(h(Z)).
\]
The error $\mathbb{E}(L[f_h](W))$ is then bounded by studying the
solution $f_h$ and exploiting the probabilistic properties of $W$.
The operator $L$ has the property that $\mathbb{E}(L[f](Z)) = 0$ for
a sufficiently large class of $f$ and therefore characterizes
$\mathscr{L}(Z)$. Examples of $L$ are (i) $L[f](w) = f'(w)-wf(w)$
for normal approximation, that is, if $Z$ is the standard normal
distribution \cite{stein72a}, and (ii)~$L[f](w) = \lambda
f(w+1)-wf(w)$ for Poisson approximation, that is, if $Z$ has the
Poisson distribution with mean $\lambda$ \cite{chen75a}. The
operator $L$ is not unique. It can be chosen to be the generator of
a Markov process whose stationary distribution is the approximating
distribution $\mathscr{L}(Z)$. This generator approach to Stein's
method is due to Barbour \cite{barbour88a,barbour90a}.

Using Stein's method with $L[f](w) = f'(w)-wf(w)$, Diaconis
\cite{diaconis77a} proved that
\[
    \sup_x\left|\mathbb{P}\left(\frac{X_{n+1}
    -(\lambda+1)/2}{\sqrt{(\lambda+1)/4}}
    \le x\right) -\Phi(x) \right|
    \le\frac{c_1}{\sqrt{\lambda}},
\]
which implies (\ref{diaconis-77}) since $\lambda=\tr{\log_2n}$. Chen
and Shao \cite{chen05a} refined Diaconis's proof to obtain
\[
    \sup_x\left|\mathbb{P}\left(\frac{X_n
    -\lambda_0/2}{\sqrt{\lambda_0/4}}
    \le x\right) -\Phi(x) \right|
    \le\frac{6.2}{\sqrt{\lambda_0}},
\]
where $\lambda_0 := \lceil \log_2 n \rceil$.

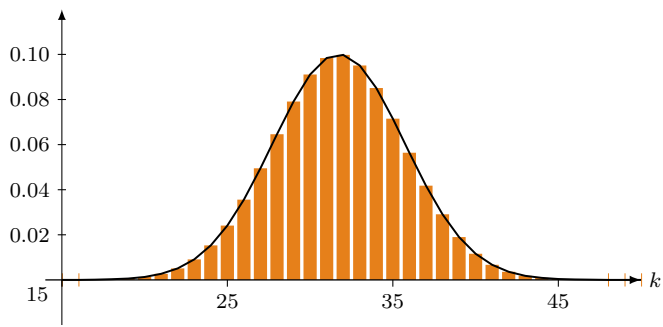
\begin{figure}
\begin{tikzpicture}[yscale=30,xscale=0.22]
\newcommand{\FONTSIZE}{\fontsize{8pt}{\baselineskip}\selectfont}
\definecolor{gen}{rgb}{0.9,0.5,0.1}%
\draw[ycomb,color=gen,line width=5pt] plot coordinates{(15.00,0)
(16,0.0000)(17,0.0001)(18,0.0003)(19,0.0006) (20.00,0.0013)(21.00
,0.0027)(22,0.0051)(23,0.0091) (24,0.0153)(25,0.0241)(26.00
,0.0356)(27,0.0495) (28,0.0646)(29,0.0791)(30,0.0911)(31,0.0984)
(32,0.0998)(33,0.0951)(34,0.0851)(35,0.0715) (36,0.0564)(37.00
,0.0418)(38,0.0291)(39,0.0190) (40,0.0116)(41,0.0067)(42.00
,0.0036)(43,0.0018) (44,0.0009)(45,0.0004)(46,0.0002)(47,0.0001)
(48,0.0000)(49,0.0000)(50,0.0000)};%
\definecolor{gen}{rgb}{0,0,0};%
\draw[gen, line width = 0.75pt] plot coordinates{(15,0.0000) (16.00
,0.0000)(17,0.0001)(18,0.0003)(19,0.0006) (20,0.0013)(21.00
,0.0027)(22,0.0051)(23,0.0091) (24,0.0153)(25,0.0241)(26.00
,0.0356)(27,0.0495) (28,0.0646)(29,0.0791)(30,0.0911)(31,0.0984)
(32,0.0998)(33,0.0951)(34,0.0851)(35,0.0715) (36,0.0564)(37.00
,0.0418)(38,0.0291)(39,0.0190) (40,0.0116)(41,0.0067)(42.00
,0.0036)(43,0.0018) (44,0.0009)(45,0.0004)(46,0.0002)(47,0.0001)
(48,0.0000)(49,0.0000)(50,0.0000)};%
\draw[-latex,line width=0.5pt] (14 , 0.000) -- (50 , 0.000)
node[right] {$k$};
\draw (15 ,0) -- (15 , -0.003)
node[below=3pt,left=2pt]{\FONTSIZE{$15$}}%
(25 ,0) -- (25 , -0.003) node[below]{\FONTSIZE{$25$}}%
(35,0) -- (35 , -0.003) node[below]{\FONTSIZE{$35$}}%
(45 ,0)--(45 , -0.003) node[below] {\FONTSIZE{$45$}};%
\draw[-latex] (15 , -0.02) -- (15 , 0.12) node[right]{};
\foreach \y/\ytext in { 0.02,0.04,0.06,0.08,0.10 }%
\draw[shift={(15, \y)}] (0.25, 0) -- (-0.25, 0)
node[left]{\FONTSIZE{$\y$}};%
\end{tikzpicture}
\caption{The histogram of $X_{31415926535897932384}$, where the black smooth curve
represents the density with the same mean and the same variance.}\label{fig9}
\end{figure}

In his book \cite{stein1986a}, Stein considered binomial
approximation for $X_n$, and using the equation
\begin{align}\label{bino-stein-eq}
    L[f](w) = (k-w)f(w)-wf(w-1)
\end{align}
for $f$ defined on $\{0, 1,\ldots, k\}$, he obtained
\[
    \max_\ell\left|\mathbb{P}\bigl(X_n=\ell\bigr)
    -\frac{1}{2^{\lambda}}\binom{\lambda}{\ell}\right|
    \le \frac{4}{\lambda};
\]
see also \cite{holmes04a}. By using the generator approach, Loh
\cite{loh92a} extended the binary expansion for $X_n$ to $q$-ary
expansion for any base $q \ge 2$, and proved that
\[
    \dtv(\mathscr{L}(X_n),\mathscr{L}(Z))
    \le \frac{3.3q^{3/2}(q-1)}{\sqrt{\lceil \log_q n\rceil}},
\]
where $X_n$ denotes the $q$-dimensional random vector whose $i$-th
component is the number of the $i$-th digit in the $q$-ary expansion
and
\[
    Z\sim \text{Multinomial} (\lceil \log_q n\rceil, 1/q, \dots,1/q).
\]

Barbour and Chen \cite{barbour92a} also used the generator approach
to improve the error bound in the binary expansion case to
$1/\lambda$ if the approximating binomial distribution
$Y_{\lambda_0}$ is replaced by $\Bi(2e(n), \tfrac12)$, where $e(n)$
is the mean of $X_n$ or by a mixture of $Y_{\lambda_0 - 1}$ with
either $Y_{\lambda_0}$ or $Y_{\lambda_0- 2}$ chosen to have mean
$e(n)$.

\subsubsection{Generating functions and analytic approach}

Schmid \cite{schmid84a} derived, improving earlier results by
Stolarsky \cite{stolarsky80a} and by Schmidt \cite{schmidt83a}, a
very precise multidimensional local limit theorem of the form
\begin{align}\label{schmid-84}
\begin{split}
    &\frac1n\#\left\{m\,:\,
    0\le m<n, \nu_2(K_jm)= k_j, j=1,\dots,d\right\} \\
    &\quad
    = \frac{\exp\left(-\frac1{2\log_2n}
    \left(\mathbf{k}-\tfrac12\log_2n\right)\mathbf{V}^{-1}
    \left(\mathbf{k}-\tfrac12\log_2n\right)^{\text{tr}}\right)}
    {(2\pi\log_2n)^{d/2}\det(\mathbf{V})^{1/2}}
    +O\left((\log n)^{-(d+1)/2}\right),
\end{split}
\end{align}
where $d\ge1$, the $K_j$'s are odd integers $>1$, $\mathbf{k}
=(k_1,\dots,k_d)$ and $\mathbf{V}$ is the positive-definite $d\times
d$ matrix with entries
\[
    v_{j,\ell} := \frac{\gcd(K_j,K_\ell)^2}{4K_jK_\ell}
    \qquad(1\le j,\ell\le d).
\]
His proof builds on matrix generating functions and uses tools from
Markov chains, following Stolarsky and Schmidt. In addition to
providing optimal convergence rate for the corresponding
multidimensional central limit theorem (\mbox{derived} in
\cite{schmidt83a}), his result implies very tight estimates for the
distribution of $\nu_2(kn)-\nu_2(n)$, a problem receiving much
attention in the literature; see the recent paper \cite{dartyge09a},
the Ph.D.\ Dissertation \cite{steiner2002a} and the references
therein.

In particular, (\ref{schmid-84}) also leads to a local limit theorem
for $X_n$ with optimal rate when $q=2$.

Drmota and Gajdosik \cite{drmota98a} use generating functions and
complex-analytic method to prove a local limit theorem for the
sum-of-digits function in more general numeration systems; see the
paper by Madritsch \cite{madritsch10a} and the references cited
there for more recent developments.

\subsubsection{Other approaches}

We mentioned the result (\ref{hi-mm}) by Dumont and Thomas
\cite{dumont93a} for the central moments of $X_n$, which implies the
asymptotic normality of $X_n$ by the Frechet-Shohat moment
convergence theorem.

The same method of moments was later applied by Bassily and K\'atai
\cite{bassily95a} to derive the asymptotic normality of $q$-additive
functionals; see also
\cite{gittenberger00a,madritsch10a,madritsch11a}.

Manstavi\v{c}ius \cite{manstavicius97a}, Drmota et al.
\cite{drmota03a} obtained a functional limit theorem for~$\nu_q(n)$.

\section{New results}

We will derive a few approximation theorems for the distribution of
$X_n$; different approaches will be developed, each having its own
advantages and constraints. In particular, an expansion for a
refined version of the total variation distance will be given, which
will cover (\ref{dtv}) as a special case.

\emph{Here and throughout this paper}, we consider only the case
$q=2$ for simplicity. The following notations will be consistently
used. Let $\lambda=\lambda_1=\tr{\log_2n}$. We then write
$n=\sum_{1\le j\le s}2^{\lambda_j}$ with $\lambda_1>
\cdots>\lambda_s \ge0$. Let $Y_\lambda\sim \Bi(\lambda,1/2)$.

Let $H_m(x)$ denote the Hermite polynomials
\[
    H_m(x) = (-1)^m e^{x^2/2}\frac{\text{d}^m}{\dd{x}^m}
    \,e^{-x^2/2}\qquad(m=0,1,\dots).
\]
Define a sequence $\{h_m\}$ by
\begin{align}\label{hm}
    h_m := \frac{2^{m/2}}{\sqrt{2\pi}}
    \int_{-\infty}^{\infty}
    \left|H_m(x)\right| e^{-x^2/2}\dd{x}
    \qquad(m=0,1,\dots).
\end{align}
\begin{thm}\label{thm-digits}
Let $X_n$ denote the number of $1$s in the binary representation of
a random integer, where each of the integers $\{0,1,\ldots,n-1\}$ is
chosen with equal probability $\mathbb{P}(X_n=m)=1/n$. Then
\begin{align*}
    &\sum_{0\le k\le \lambda}\left|\mathbb{P}
    \bigl(X_n=k\bigr)-\sum_{0\le r<m}(-1)^r a_r(n)
    2^{-\lambda}\Delta^r \binom{\lambda}{k}\right|\\
    &\qquad = \frac{h_m |a_m(n)|}{(\log_2 n)^{m/2}}+
    O\left((\log n)^{-(m+1)/2}\right),
\end{align*}
for $m=1,2,\dots$, where the sequence $a_r(n)=a_r(2n)$ is defined by
(see (\ref{BGF}))
\begin{align} \label{brn}
    \mathbb{E}(y^{X_n})
    \left(\frac{1+y}{2}\right)^{-\lambda}
    = \frac1n\sum_{1\le j\le s}2^{\lambda_j}
    y^{j-1}\left(\frac{1+y}2\right)^{\lambda_j-\lambda}
    = \sum_{r\ge0} a_r(n) (y-1)^r,
\end{align}
and $\Delta$ denotes the difference operator
\[
    \Delta^r\binom{\lambda}{k}=
    \sum_{0\le \ell \le r}\binom{r}{\ell}
    (-1)^\ell \binom{\lambda}{k-\ell}\qquad(r=0,1,\dots).
\]
\end{thm}

Two explicit expressions for $a_r(n)$ are as follows.
\begin{align*}
    a_r(n) &= \sum_{0\le \ell \le r}
    \binom{\lambda+r-\ell-1}{r-\ell}\frac{(-1)^{r-\ell}}
    {\ell! 2^{r-\ell}}\mathbb{E}(X_n
    \cdots(X_n-\ell+1)) \\
    & = \frac{2^\lambda}n\sum_{1\le j\le s}
    2^{-(\lambda-\lambda_j)} \sum_{0\le \ell \le r}
    \binom{\lambda-\lambda_j+\ell-1}{\ell}\binom{j-1}{r-\ell}
    (-1)^\ell 2^{-\ell},
\end{align*}
for $r=0,1,\dots$.

Taking $m=1$ and dividing the left-hand side by $2$, we obtain an
asymptotic approximation to the total variation distance; see
\cite{chen94a,soon1993a} for similar results.
\begin{cor} \label{dtv0}
The total variation distance between the distribution of $X_n$ and
that of the binomial random variable $Y_\lambda$ satisfies
\begin{align*}
    \dtv(\mathscr{L}(X_n),\mathscr{L}(Y_\lambda))
    = \frac{\sqrt{2}\,|F(\log_2n)|}
    {\sqrt{\pi\log_2n}}+O\left((\log n)^{-1}\right),
\end{align*}
where the bounded periodic function $F$ is defined by
\begin{align}\label{Fx}
    F(x) = 2^{-x}\sum_{j\ge0}2^{-d_j}
    \left(j-1-\frac{d_j}2\right),
\end{align}
for $x\in(0,1]$, when writing $2^x=\sum_{j\ge0}2^{-d_j}\in(1,2]$
with $0=d_0<d_1< \cdots$, and $F(x+1)=F(x)$ for other values of $x$.
\end{cor}
\begin{proof}
Observe first that $a_0(n)=1$ and by the definition of $F$
\[
    F(\log_2n) = a_1(n) = \mathbb{E}(X_n)-
    \frac{\lambda}2,
\]
which has the form
\begin{align} \label{Flogn}
    F(\log_2n)= \frac{2^\lambda}n
    \sum_{1\le j\le s} 2^{-(\lambda-\lambda_j)}
    \left(j-1-\frac{\lambda-\lambda_j}{2}\right).
\end{align}
On the other hand, since $H_1(x)=x$, we then get $h_1=\sqrt{2/\pi}$.
By considering the values of
\[
    2^x = \sum_{0\le j\le k} 2^{-d_j}
    = \sum_{0\le j<k} 2^{-d_j} + \sum_{j>d_k}2^{-j},
\]
we see that $F$ is continuous except at the end points (integers).
\end{proof}
By (\ref{Flogn}), we see that if $\lambda_j = \lambda-2(j-1)$ for
$j=1,\dots,s$, then $F(\log_2n)=0$. This yields the sequence
$\{\log_2(\sum_{0\le j\le k}4^{-j})
\}_{k=0,1,\dots}$ for the locations of the zeros of $|F(x)|$;
see Figure~\ref{fig10}.

\begin{figure}[t!]
\hspace*{-4mm}{ 
\scalebox{1.09}{
\begin{tikzpicture}
\node (A)[above=8cm,left=7cm] {
\begin{tikzpicture}[yscale=8,xscale=4.5]
\newcommand{\FONTSIZE}{\fontsize{5pt}{\baselineskip}\selectfont}
\foreach \x in { 0,0.2,0.4,0.6,0.8,1 }%
\draw[shift={(\x,0)}] (0,-0.01) -- (0,0)
node[below=2pt]{\FONTSIZE{$\x$}};%
\foreach \x in {0.1,0.3,...,0.9 }%
\draw[shift={(\x,0)}] (0,-0.005) -- (0,0);%
\foreach \y in {0,0.1,0.2,0.3,0.4}%
\draw[shift={(0,\y)}] (-0.02,0) -- (0,0)
node[left=2pt]{\FONTSIZE{$\y$}};%
\foreach \y in {0.05,0.15,...,0.5}%
\draw[shift={(0,\y)}] (-0.01,0) -- (0,0);%
\filldraw[blue!60, line width = 0.5pt] plot coordinates{(0.000,0)
(0.003,0.007)(0.006,0.012)(0.008,0.017)(0.011,0.019)(0.014,0.024)
(0.017,0.027)(0.020,0.030)(0.022,0.031)(0.025,0.036)(0.028,0.038)
(0.031,0.041)(0.033,0.042)(0.036,0.045)(0.039,0.046)(0.042,0.046)
(0.044,0.045)(0.047,0.050)(0.050,0.053)(0.053,0.056)(0.055,0.056)
(0.058,0.059)(0.061,0.060)(0.063,0.061)(0.066,0.060)(0.069,0.062)
(0.071,0.063)(0.074,0.064)(0.077,0.063)(0.079,0.064)(0.082,0.063)
(0.085,0.062)(0.087,0.059)(0.090,0.063)(0.093,0.066)(0.095,0.069)
(0.098,0.069)(0.101,0.072)(0.103,0.073)(0.106,0.074)(0.109,0.072)
(0.111,0.075)(0.114,0.076)(0.116,0.077)(0.119,0.076)(0.122,0.076)
(0.124,0.075)(0.127,0.074)(0.129,0.071)(0.132,0.074)(0.134,0.075)
(0.137,0.075)(0.140,0.074)(0.142,0.075)(0.145,0.074)(0.147,0.073)
(0.150,0.070)(0.152,0.071)(0.155,0.070)(0.157,0.069)(0.160,0.066)
(0.162,0.065)(0.165,0.063)(0.167,0.060)(0.170,0.056)(0.172,0.060)
(0.175,0.062)(0.177,0.065)(0.180,0.066)(0.182,0.068)(0.185,0.069)
(0.187,0.069)(0.190,0.068)(0.192,0.071)(0.195,0.072)(0.197,0.072)
(0.200,0.071)(0.202,0.072)(0.205,0.071)(0.207,0.070)(0.209,0.068)
(0.212,0.070)(0.214,0.071)(0.217,0.071)(0.219,0.070)(0.222,0.071)
(0.224,0.070)(0.226,0.069)(0.229,0.067)(0.231,0.067)(0.234,0.066)
(0.236,0.066)(0.238,0.063)(0.241,0.062)(0.243,0.059)(0.246,0.057)
(0.248,0.053)(0.250,0.055)(0.253,0.056)(0.255,0.056)(0.257,0.056)
(0.260,0.056)(0.262,0.055)(0.264,0.054)(0.267,0.052)(0.269,0.053)
(0.271,0.052)(0.274,0.051)(0.276,0.048)(0.278,0.048)(0.281,0.045)
(0.283,0.043)(0.285,0.038)(0.288,0.039)(0.290,0.038)(0.292,0.037)
(0.295,0.035)(0.297,0.034)(0.299,0.032)(0.301,0.029)(0.304,0.025)
(0.306,0.024)(0.308,0.022)(0.311,0.020)(0.313,0.016)(0.315,0.013)
(0.317,0.009)(0.320,0.005)}-- (0.320,0);%

\filldraw[blue!60, line width = 0.5pt] plot coordinates{
(0.445,0)(0.445,0.024)
(0.447,0.026)(0.449,0.028)(0.451,0.031)(0.453,0.034)(0.455,0.037)
(0.457,0.041)(0.459,0.045)(0.461,0.043)(0.464,0.042)(0.466,0.042)
(0.468,0.042)(0.470,0.042)(0.472,0.042)(0.474,0.043)(0.476,0.045)
(0.478,0.044)(0.480,0.045)(0.482,0.045)(0.484,0.047)(0.486,0.048)
(0.488,0.050)(0.490,0.052)(0.492,0.056)(0.494,0.055)(0.496,0.055)
(0.498,0.056)(0.500,0.058)(0.502,0.059)(0.504,0.061)(0.506,0.063)
(0.508,0.066)(0.510,0.067)(0.512,0.068)(0.514,0.070)(0.516,0.074)
(0.518,0.076)(0.520,0.079)(0.522,0.082)(0.524,0.087)(0.526,0.086)
(0.527,0.087)(0.529,0.087)(0.531,0.089)(0.533,0.090)(0.535,0.092)
(0.537,0.094)(0.539,0.097)(0.541,0.097)(0.543,0.099)(0.545,0.101)
(0.547,0.104)(0.549,0.106)(0.551,0.109)(0.553,0.113)(0.555,0.117)
(0.557,0.118)(0.558,0.119)(0.560,0.121)(0.562,0.124)(0.564,0.126)
(0.566,0.129)(0.568,0.132)(0.570,0.137)(0.572,0.139)(0.574,0.142)
(0.576,0.145)(0.577,0.149)(0.579,0.152)(0.581,0.157)(0.583,0.161)
(0.585,0.167)(0.587,0.163)(0.589,0.161)(0.591,0.159)(0.592,0.158)
(0.594,0.156)(0.596,0.155)(0.598,0.154)(0.600,0.155)(0.602,0.153)
(0.604,0.152)(0.605,0.151)(0.607,0.151)(0.609,0.150)(0.611,0.151)
(0.613,0.151)(0.615,0.153)(0.617,0.151)(0.618,0.150)(0.620,0.149)
(0.622,0.150)(0.624,0.149)(0.626,0.149)(0.628,0.150)(0.629,0.152)
(0.631,0.151)(0.633,0.151)(0.635,0.152)(0.637,0.153)(0.638,0.154)
(0.640,0.155)(0.642,0.157)(0.644,0.160)(0.646,0.158)(0.647,0.157)
(0.649,0.156)(0.651,0.157)(0.653,0.156)(0.655,0.156)(0.656,0.157)
(0.658,0.158)(0.660,0.158)(0.662,0.158)(0.664,0.158)(0.665,0.160)
(0.667,0.161)(0.669,0.162)(0.671,0.164)(0.672,0.167)(0.674,0.166)
(0.676,0.166)(0.678,0.167)(0.679,0.168)(0.681,0.169)(0.683,0.170)
(0.685,0.172)(0.687,0.175)(0.688,0.175)(0.690,0.177)(0.692,0.178)
(0.693,0.181)(0.695,0.183)(0.697,0.186)(0.699,0.188)(0.700,0.192)
(0.702,0.190)(0.704,0.189)(0.706,0.189)(0.707,0.189)(0.709,0.188)
(0.711,0.189)(0.713,0.189)(0.714,0.190)(0.716,0.190)(0.718,0.190)
(0.719,0.190)(0.721,0.192)(0.723,0.192)(0.725,0.194)(0.726,0.195)
(0.728,0.198)(0.730,0.197)(0.731,0.198)(0.733,0.198)(0.735,0.200)
(0.736,0.200)(0.738,0.201)(0.740,0.203)(0.741,0.206)(0.743,0.206)
(0.745,0.207)(0.747,0.209)(0.748,0.212)(0.750,0.213)(0.752,0.216)
(0.753,0.218)(0.755,0.222)(0.757,0.221)(0.758,0.222)(0.760,0.222)
(0.762,0.224)(0.763,0.224)(0.765,0.225)(0.767,0.227)(0.768,0.229)
(0.770,0.230)(0.771,0.231)(0.773,0.233)(0.775,0.235)(0.776,0.237)
(0.778,0.239)(0.780,0.242)(0.781,0.245)(0.783,0.246)(0.785,0.247)
(0.786,0.249)(0.788,0.251)(0.790,0.253)(0.791,0.255)(0.793,0.258)
(0.794,0.261)(0.796,0.263)(0.798,0.265)(0.799,0.268)(0.801,0.271)
(0.803,0.274)(0.804,0.277)(0.806,0.281)(0.807,0.286)(0.809,0.284)
(0.811,0.283)(0.812,0.282)(0.814,0.282)(0.815,0.281)(0.817,0.282)
(0.819,0.282)(0.820,0.283)(0.822,0.282)(0.823,0.283)(0.825,0.283)
(0.827,0.284)(0.828,0.284)(0.830,0.286)(0.831,0.287)(0.833,0.289)
(0.834,0.289)(0.836,0.289)(0.838,0.289)(0.839,0.290)(0.841,0.291)
(0.842,0.292)(0.844,0.293)(0.845,0.296)(0.847,0.296)(0.849,0.297)
(0.850,0.298)(0.852,0.301)(0.853,0.302)(0.855,0.305)(0.856,0.307)
(0.858,0.310)(0.860,0.309)(0.861,0.310)(0.863,0.310)(0.864,0.311)
(0.866,0.311)(0.867,0.313)(0.869,0.314)(0.870,0.316)(0.872,0.316)
(0.873,0.318)(0.875,0.319)(0.877,0.321)(0.878,0.323)(0.880,0.325)
(0.881,0.327)(0.883,0.331)(0.884,0.331)(0.886,0.332)(0.887,0.333)
(0.889,0.335)(0.890,0.337)(0.892,0.339)(0.893,0.341)(0.895,0.345)
(0.896,0.346)(0.898,0.348)(0.899,0.350)(0.901,0.354)(0.902,0.356)
(0.904,0.359)(0.905,0.362)(0.907,0.367)(0.908,0.366)(0.910,0.366)
(0.911,0.366)(0.913,0.367)(0.914,0.367)(0.916,0.369)(0.917,0.370)
(0.919,0.372)(0.920,0.372)(0.922,0.373)(0.923,0.374)(0.925,0.377)
(0.926,0.378)(0.928,0.380)(0.929,0.382)(0.931,0.385)(0.932,0.385)
(0.934,0.387)(0.935,0.388)(0.937,0.390)(0.938,0.391)(0.940,0.393)
(0.941,0.395)(0.943,0.398)(0.944,0.399)(0.945,0.402)(0.947,0.404)
(0.948,0.407)(0.950,0.409)(0.951,0.412)(0.953,0.415)(0.954,0.419)
(0.956,0.419)(0.957,0.421)(0.959,0.422)(0.960,0.424)(0.961,0.425)
(0.963,0.427)(0.964,0.429)(0.966,0.432)(0.967,0.433)(0.969,0.435)
(0.970,0.437)(0.972,0.440)(0.973,0.442)(0.974,0.445)(0.976,0.448)
(0.977,0.452)(0.979,0.453)(0.980,0.455)(0.982,0.457)(0.983,0.460)
(0.984,0.462)(0.986,0.465)(0.987,0.468)(0.989,0.472)(0.990,0.474)
(0.992,0.477)(0.993,0.480)(0.994,0.484)(0.996,0.487)(0.997,0.491)
(0.999,0.495)(1.000,0.000)}--(1,0);%
\filldraw[black!80, line width = 0.5pt] plot coordinates{(0.322,0)
(0.324,0.004)(0.326,0.006)(0.329,0.009)(0.331,0.009)(0.333,0.012)
(0.335,0.012)(0.338,0.013)(0.340,0.012)(0.342,0.015)(0.344,0.015)
(0.347,0.016)(0.349,0.015)(0.351,0.016)(0.353,0.015)(0.355,0.015)
(0.358,0.012)(0.360,0.014)(0.362,0.015)(0.364,0.016)(0.366,0.015)
(0.369,0.016)(0.371,0.015)(0.373,0.014)(0.375,0.012)(0.377,0.013)
(0.379,0.012)(0.382,0.011)(0.384,0.009)(0.386,0.008)(0.388,0.006)
(0.390,0.004)(0.392,0.000)(0.394,0.002)(0.397,0.003)(0.399,0.004)
(0.401,0.003)(0.403,0.004)(0.405,0.003)(0.407,0.002)(0.409,0.000)
(0.412,0.001)(0.414,0.000)(0.416,0.001)(0.418,0.003)(0.420,0.004)
(0.422,0.006)(0.424,0.008)(0.426,0.012)(0.428,0.011)(0.430,0.012)
(0.433,0.012)(0.435,0.014)(0.437,0.015)(0.439,0.017)(0.441,0.019)
(0.443,0.023)}--(0.443,0);%
\draw[-latex,line width=.5pt] (0,0) -- (1.1,0)
node[right] {\FONTSIZE{$x$}};%
\draw[-latex,line width=.5pt] (0,0)-- (0,.5)
node[right]{\FONTSIZE{$F(x)$}};%
\end{tikzpicture}
};

\node (B) [above=3.35cm,left=10.05cm]{
\begin{tikzpicture}[yscale=142.5,xscale=36.5]
\newcommand{\FONTSIZE}{\fontsize{5pt}{\baselineskip}\selectfont}
\foreach \x in { 0.34,0.36,0.38,0.40,0.42}%
\draw[shift={(\x,0)}] (0,-.0007) -- (0,0)
node[below=1pt]{\FONTSIZE{$\x$}};%
\foreach \x in { 0.33,0.35,...,0.44 }%
\draw[shift={(\x,0)}] (0, -.0005) -- (0, 0);%
\foreach \y in {0,0.01,0.02}%
\draw[shift={(.3219,\y)}] (-.003,0) -- (0,0)
node[left=2pt]{\FONTSIZE{$\y$}};%
\foreach \y in { 0.005,0.015,0.02 }%
\draw[shift={(0.3219, \y)}] (-0.0015, 0) -- (0, 0);%
\filldraw[black!80, line width = 0.5pt] plot coordinates{(0.3219,0)
(0.3225,0.0014)(0.3231,0.0023)(0.3236,0.0033)(0.3242,0.0039)
(0.3247,0.0049)(0.3253,0.0055)(0.3259,0.0060)(0.3264,0.0062)
(0.3270,0.0072)(0.3276,0.0078)(0.3281,0.0084)(0.3287,0.0086)
(0.3292,0.0091)(0.3298,0.0093)(0.3304,0.0095)(0.3309,0.0093)
(0.3315,0.0103)(0.3320,0.0109)(0.3326,0.0114)(0.3332,0.0116)
(0.3337,0.0122)(0.3343,0.0124)(0.3348,0.0126)(0.3354,0.0124)
(0.3359,0.0130)(0.3365,0.0131)(0.3371,0.0133)(0.3376,0.0131)
(0.3382,0.0133)(0.3387,0.0131)(0.3393,0.0129)(0.3399,0.0123)
(0.3404,0.0133)(0.3410,0.0139)(0.3415,0.0145)(0.3421,0.0146)
(0.3426,0.0152)(0.3432,0.0154)(0.3437,0.0156)(0.3443,0.0154)
(0.3449,0.0160)(0.3454,0.0161)(0.3460,0.0163)(0.3465,0.0161)
(0.3471,0.0163)(0.3476,0.0161)(0.3482,0.0159)(0.3487,0.0153)
(0.3493,0.0159)(0.3498,0.0161)(0.3504,0.0163)(0.3509,0.0161)
(0.3515,0.0163)(0.3520,0.0161)(0.3526,0.0159)(0.3531,0.0153)
(0.3537,0.0155)(0.3542,0.0153)(0.3548,0.0151)(0.3554,0.0145)
(0.3559,0.0143)(0.3565,0.0137)(0.3570,0.0132)(0.3576,0.0122)
(0.3581,0.0131)(0.3587,0.0137)(0.3592,0.0143)(0.3597,0.0145)
(0.3603,0.0150)(0.3608,0.0152)(0.3614,0.0154)(0.3619,0.0152)
(0.3625,0.0158)(0.3630,0.0159)(0.3636,0.0161)(0.3641,0.0159)
(0.3647,0.0161)(0.3652,0.0159)(0.3658,0.0157)(0.3663,0.0152)
(0.3669,0.0157)(0.3674,0.0159)(0.3680,0.0161)(0.3685,0.0159)
(0.3691,0.0161)(0.3696,0.0159)(0.3701,0.0157)(0.3707,0.0151)
(0.3712,0.0153)(0.3718,0.0151)(0.3723,0.0149)(0.3729,0.0143)
(0.3734,0.0141)(0.3740,0.0136)(0.3745,0.0130)(0.3750,0.0120)
(0.3756,0.0126)(0.3761,0.0128)(0.3767,0.0130)(0.3772,0.0128)
(0.3778,0.0130)(0.3783,0.0128)(0.3788,0.0126)(0.3794,0.0120)
(0.3799,0.0122)(0.3805,0.0120)(0.3810,0.0118)(0.3815,0.0112)
(0.3821,0.0111)(0.3826,0.0105)(0.3832,0.0099)(0.3837,0.0090)
(0.3842,0.0092)(0.3848,0.0090)(0.3853,0.0088)(0.3859,0.0082)
(0.3864,0.0080)(0.3869,0.0075)(0.3875,0.0069)(0.3880,0.0060)
(0.3886,0.0058)(0.3891,0.0052)(0.3896,0.0047)(0.3902,0.0037)
(0.3907,0.0032)(0.3912,0.0022)(0.3918,0.0013)}-- (0.3918,0);%

\filldraw[black!80, line width = 0.5pt] plot coordinates{ (0.4236,0)
(0.4242,0.0080)
(0.4247,0.0085)(0.4252,0.0095)(0.4257,0.0104)(0.4263,0.0116)
(0.4268,0.0111)(0.4273,0.0109)(0.4278,0.0107)(0.4284,0.0109)
(0.4289,0.0107)(0.4294,0.0109)(0.4299,0.0111)(0.4305,0.0116)
(0.4310,0.0114)(0.4315,0.0116)(0.4320,0.0118)(0.4325,0.0123)
(0.4331,0.0125)(0.4336,0.0130)(0.4341,0.0136)(0.4346,0.0145)
(0.4351,0.0143)(0.4357,0.0144)(0.4362,0.0146)(0.4367,0.0152)
(0.4372,0.0153)(0.4378,0.0159)(0.4383,0.0164)(0.4388,0.0173)
(0.4393,0.0175)(0.4398,0.0180)(0.4404,0.0185)(0.4409,0.0194)
(0.4414,0.0200)(0.4419,0.0208)(0.4424,0.0217)(0.4429,0.0230)
(0.4435,0.0228) }--(0.4435,0);%
\filldraw[blue!60,line width = 0.5pt] plot coordinates{(.3923,0)
(0.3929,0.0009)(0.3934,0.0015)(0.3939,0.0020)(0.3945,0.0022)
(0.3950,0.0028)(0.3955,0.0030)(0.3961,0.0032)(0.3966,0.0030)
(0.3971,0.0035)(0.3977,0.0037)(0.3982,0.0039)(0.3987,0.0037)
(0.3993,0.0039)(0.3998,0.0037)(0.4003,0.0035)(0.4009,0.0030)
(0.4014,0.0035)(0.4019,0.0037)(0.4025,0.0039)(0.4030,0.0037)
(0.4035,0.0039)(0.4041,0.0037)(0.4046,0.0035)(0.4051,0.0029)
(0.4057,0.0031)(0.4062,0.0029)(0.4067,0.0028)(0.4073,0.0022)
(0.4078,0.0020)(0.4083,0.0015)(0.4089,0.0009)(0.4094,0.0000)
(0.4099,0.0006)(0.4105,0.0007)(0.4110,0.0009)(0.4115,0.0007)
(0.4120,0.0009)(0.4126,0.0007)(0.4131,0.0006)(0.4136,0.0000)
(0.4142,0.0002)(0.4147,0.0000)(0.4152,0.0002)(0.4157,0.0007)
(0.4163,0.0009)(0.4168,0.0015)(0.4173,0.0020)(0.4179,0.0029)
(0.4184,0.0027)(0.4189,0.0029)(0.4194,0.0031)(0.4200,0.0036)
(0.4205,0.0038)(0.4210,0.0044)(0.4215,0.0049)(0.4221,0.0058)
(0.4226,0.0060)(0.4231,0.0066)(0.4236,0.0071)}--(0.4236,0);%
\draw[-latex,line width=0.5pt] (0.3219, 0) -- (0.4435,0)
    node[right]{\FONTSIZE{$x$}};%
\draw[-latex,line width=0.5pt] (0.3219, 0) -- (0.3219, 0.0228)
    node[right]{\FONTSIZE{$F(x)$}};%
\end{tikzpicture}
};

\node (C) [below=0.5cm,left=10cm]{
\begin{tikzpicture}[yscale=240,xscale=610]
\newcommand{\FONTSIZE}{\fontsize{5pt}{\baselineskip}\selectfont}
\foreach \x in {0.410,0.411,0.412,0.413,0.414,0.415,0.416}%
\draw[shift={(\x, 0)}] (0,-.0003) -- (0,0)
node[below=2pt]{\FONTSIZE{$\x$}};%
\foreach \y/\ytext in {0/0,0.002/0.0002,0.004/0.0004,0.006/0.0006,
0.008/0.0008,0.01/0.0010,0.012/0.0012}%
\draw[shift={(.40939,\y)}] (-.0001,0) --
(0,0)node[left=1pt]{\FONTSIZE{$\ytext$}};
\foreach \y in { 0.001,0.003,...,0.013 }%
\draw[shift={(0.40939, \y)}] (-0.00005, 0) -- (0, 0);%
\filldraw[black!80,line width=.5pt] plot coordinates{(.40939,0)
(0.40942,0.000804)(0.40946,0.001379)(0.40949,0.001953)
(0.40952,0.002298)(0.40956,0.002872)(0.40959,0.003216)
(0.40962,0.003561)(0.40966,0.003676)(0.40969,0.004250)
(0.40972,0.004595)(0.40976,0.004939)(0.40979,0.005054)
(0.40982,0.005398)(0.40985,0.005513)(0.40989,0.005628)
(0.40992,0.005513)(0.40995,0.006087)(0.40999,0.006431)
(0.41002,0.006776)(0.41005,0.006890)(0.41009,0.007235)
(0.41012,0.007349)(0.41015,0.007464)(0.41019,0.007349)
(0.41022,0.007693)(0.41025,0.007808)(0.41029,0.007922)
(0.41032,0.007807)(0.41035,0.007922)(0.41039,0.007807)
(0.41042,0.007692)(0.41045,0.007348)(0.41048,0.007921)
(0.41052,0.008266)(0.41055,0.008610)(0.41058,0.008724)
(0.41062,0.009069)(0.41065,0.009183)(0.41068,0.009298)
(0.41072,0.009183)(0.41075,0.009527)(0.41078,0.009641)
(0.41082,0.009756)(0.41085,0.009641)(0.41088,0.009756)
(0.41092,0.009641)(0.41095,0.009526)(0.41098,0.009181)
(0.41101,0.009525)(0.41105,0.009640)(0.41108,0.009754)
(0.41111,0.009639)(0.41115,0.009754)(0.41118,0.009639)
(0.41121,0.009524)(0.41125,0.009179)(0.41128,0.009294)
(0.41131,0.009179)(0.41135,0.009064)(0.41138,0.008720)
(0.41141,0.008605)(0.41144,0.008260)(0.41148,0.007916)
(0.41151,0.007342)(0.41154,0.007916)(0.41158,0.008260)
(0.41161,0.008603)(0.41164,0.008718)(0.41168,0.009062)
(0.41171,0.009176)(0.41174,0.009291)(0.41178,0.009176)
(0.41181,0.009520)(0.41184,0.009634)(0.41188,0.009749)
(0.41191,0.009634)(0.41194,0.009748)(0.41197,0.009633)
(0.41201,0.009519)(0.41204,0.009174)(0.41207,0.009518)
(0.41211,0.009633)(0.41214,0.009747)(0.41217,0.009632)
(0.41221,0.009747)(0.41224,0.009632)(0.41227,0.009517)
(0.41231,0.009173)(0.41234,0.009287)(0.41237,0.009172)
(0.41240,0.009057)(0.41244,0.008713)(0.41247,0.008598)
(0.41250,0.008254)(0.41254,0.007910)(0.41257,0.007337)
(0.41260,0.007680)(0.41264,0.007795)(0.41267,0.007909)
(0.41270,0.007795)(0.41274,0.007909)(0.41277,0.007794)
(0.41280,0.007679)(0.41283,0.007335)(0.41287,0.007450)
(0.41290,0.007335)(0.41293,0.007220)(0.41297,0.006876)
(0.41300,0.006762)(0.41303,0.006418)(0.41307,0.006074)
(0.41310,0.005501)(0.41313,0.005615)(0.41317,0.005500)
(0.41320,0.005386)(0.41323,0.005042)(0.41326,0.004927)
(0.41330,0.004583)(0.41333,0.004239)(0.41336,0.003666)
(0.41340,0.003552)(0.41343,0.003208)(0.41346,0.002864)
(0.41350,0.002291)(0.41353,0.001948)(0.41356,0.001375)
(0.41359,0.000802)(0.41363,0.000000)(0.41366,0.000573)
(0.41369,0.000916)(0.41373,0.001260)(0.41376,0.001375)
(0.41379,0.001718)(0.41383,0.001833)(0.41386,0.001947)
(0.41389,0.001833)(0.41393,0.002176)(0.41396,0.002291)
(0.41399,0.002405)(0.41402,0.002290)(0.41406,0.002405)
(0.41409,0.002290)(0.41412,0.002176)(0.41416,0.001832)
(0.41419,0.002176)(0.41422,0.002290)(0.41426,0.002405)
(0.41429,0.002290)(0.41432,0.002404)(0.41435,0.002290)
(0.41439,0.002175)(0.41442,0.001832)(0.41445,0.001946)
(0.41449,0.001832)(0.41452,0.001717)(0.41455,0.001374)
(0.41459,0.001259)(0.41462,0.000916)(0.41465,0.000572)
(0.41469,0.000000)(0.41472,0.000343)(0.41475,0.000458)
(0.41478,0.000572)(0.41482,0.000458)(0.41485,0.000572)
(0.41488,0.000458)(0.41492,0.000343)(0.41495,0.000000)
(0.41498,0.000114)(0.41502,0.000000)(0.41505,0.000114)
(0.41508,0.000458)(0.41511,0.000572)(0.41515,0.000915)
(0.41518,0.001259)(0.41521,0.001831)(0.41525,0.001716)
(0.41528,0.001831)(0.41531,0.001945)(0.41535,0.002288)
(0.41538,0.002403)(0.41541,0.002746)(0.41544,0.003089)
(0.41548,0.003661)(0.41551,0.003775)(0.41554,0.004118)
(0.41558,0.004462)(0.41561,0.005033)(0.41564,0.005376)
(0.41568,0.005948)(0.41571,0.006520)(0.41574,0.007321)
(0.41577,0.006977)(0.41581,0.006863)(0.41584,0.006748)
(0.41587,0.006862)(0.41591,0.006748)(0.41594,0.006862)
(0.41597,0.006976)(0.41601,0.007319)(0.41604,0.007205)
(0.41607,0.007319)(0.41610,0.007433)(0.41614,0.007776)
(0.41617,0.007890)(0.41620,0.008233)(0.41624,0.008576)
(0.41627,0.009147)(0.41630,0.009033)(0.41634,0.009147)
(0.41637,0.009261)(0.41640,0.009604)(0.41643,0.009718)
(0.41647,0.010061)(0.41650,0.010404)(0.41653,0.010975)
(0.41657,0.011089)(0.41660,0.011432)(0.41663,0.011774)
(0.41667,0.012346)(0.41670,0.012688)(0.41673,0.013260)
(0.41676,0.013831)(0.41680,0.014631)(0.41683,0.014516)
}--(.41683,0);%
\filldraw[blue!60, line width = 0.5pt] plot coordinates{(.41363,0)
(0.41366,0.000573)(0.41369,0.000916)(0.41373,0.001260)
(0.41376,0.001375)(0.41379,0.001718)(0.41383,0.001833)
(0.41386,0.001947)(0.41389,0.001833)(0.41393,0.002176)
(0.41396,0.002291)(0.41399,0.002405)(0.41402,0.002290)
(0.41406,0.002405)(0.41409,0.002290)(0.41412,0.002176)
(0.41416,0.001832)(0.41419,0.002176)(0.41422,0.002290)
(0.41426,0.002405)(0.41429,0.002290)(0.41432,0.002404)
(0.41435,0.002290)(0.41439,0.002175)(0.41442,0.001832)
(0.41445,0.001946)(0.41449,0.001832)(0.41452,0.001717)
(0.41455,0.001374)(0.41459,0.001259)(0.41462,0.000916)
(0.41465,0.000572)(0.41469,0.000000)(0.41472,0.000343)
(0.41475,0.000458)(0.41478,0.000572)(0.41482,0.000458)
(0.41485,0.000572)(0.41488,0.000458)(0.41492,0.000343)
(0.41495,0.000000)(0.41498,0.000114)(0.41502,0.000000)
(0.41505,0.000114)(0.41508,0.000458)(0.41511,0.000572)
(0.41515,0.000915)(0.41518,0.001259)(0.41521,0.001831)
(0.41525,0.001716)(0.41528,0.001831)(0.41531,0.001945)
(0.41535,0.002288)(0.41538,0.002403)(0.41541,0.002746)
(0.41544,0.003089)(0.41548,0.003661)}--(0.41548,0);%
\draw[-latex,line width=0.5pt]
(0.40939,0) -- (0.41683,0)node[right]{\FONTSIZE{$x$}};%
\draw[-latex,line width=0.5pt]
(0.40939,0) -- (0.40939,0.014516)node[right]{\FONTSIZE{$F(x)$}};%
\end{tikzpicture}
};

\node (D) [below=0.5cm,left=4.5cm]{
\begin{tikzpicture}[yscale=0.97,xscale=243]
\newcommand{\FONTSIZE}{\fontsize{5pt}{\baselineskip}\selectfont}
\foreach \x/\xtext in {0.140/0.4140, 0.145/0.4145,0.150/0.4150}%
\draw[line width = 0.4pt,shift={(\x, 0)}]
(0, -0.08) -- (0, 0)node[below=1pt]{\FONTSIZE{$\xtext$}};%
\foreach \x in {0.140, 0.150}%
\draw[line width = 0.4pt,shift={(\x, 0)}] (0, -0.0001) -- (0, 0);%
\foreach \y/\ytext in {0,1/0.0001,2/0.0002,3/0.0003}%
\draw[line width = 0.4pt, shift={(0.136279, \y)}]
(-0.0003, 0) -- (0, 0)node[left]{\FONTSIZE{$\ytext$}};%
\filldraw[blue!60,line width=.5pt] plot coordinates{(.136279,0)
(0.136362,0.200466)(0.136445,0.343654)(0.136527,0.486841)
(0.136610,0.572751)(0.136692,0.715934)(0.136775,0.801842)
(0.136858,0.887748)(0.136940,0.916380)(0.137023,1.059559)
(0.137106,1.145462)(0.137188,1.231365)(0.137271,1.259994)
(0.137353,1.345895)(0.137436,1.374523)(0.137519,1.403151)
(0.137601,1.374507)(0.137684,1.517677)(0.137767,1.603574)
(0.137849,1.689470)(0.137932,1.718095)(0.138014,1.803989)
(0.138097,1.832614)(0.138180,1.861238)(0.138262,1.832593)
(0.138345,1.918484)(0.138428,1.947107)(0.138510,1.975730)
(0.138593,1.947085)(0.138675,1.975707)(0.138758,1.947063)
(0.138841,1.918419)(0.138923,1.832509)(0.139006,1.975662)
(0.139089,2.061549)(0.139171,2.147434)(0.139254,2.176054)
(0.139336,2.261938)(0.139419,2.290557)(0.139502,2.319176)
(0.139584,2.290531)(0.139667,2.376412)(0.139749,2.405030)
(0.139832,2.433647)(0.139915,2.405002)(0.139997,2.433619)
(0.140080,2.404975)(0.140162,2.376331)(0.140245,2.290426)
(0.140328,2.376303)(0.140410,2.404920)(0.140493,2.433536)
(0.140576,2.404892)(0.140658,2.433508)(0.140741,2.404865)
(0.140823,2.376222)(0.140906,2.290321)(0.140989,2.318937)
(0.141071,2.290295)(0.141154,2.261653)(0.141236,2.175755)
(0.141319,2.147115)(0.141402,2.061218)(0.141484,1.975323)
(0.141567,1.832173)(0.141649,1.975300)(0.141732,2.061171)
(0.141815,2.147041)(0.141897,2.175656)(0.141980,2.261524)
(0.142062,2.290138)(0.142145,2.318751)(0.142228,2.290111)
(0.142310,2.375977)(0.142393,2.404589)(0.142475,2.433201)
(0.142558,2.404562)(0.142641,2.433174)(0.142723,2.404534)
(0.142806,2.375895)(0.142888,2.290006)(0.142971,2.375868)
(0.143054,2.404479)(0.143136,2.433090)(0.143219,2.404452)
(0.143301,2.433062)(0.143384,2.404424)(0.143467,2.375786)
(0.143549,2.289902)(0.143632,2.318512)(0.143714,2.289875)
(0.143797,2.261239)(0.143879,2.175357)(0.143962,2.146721)
(0.144045,2.060841)(0.144127,1.974961)(0.144210,1.831837)
(0.144292,1.917694)(0.144375,1.946305)(0.144458,1.974916)
(0.144540,1.946283)(0.144623,1.974893)(0.144705,1.946260)
(0.144788,1.917628)(0.144870,1.831753)(0.144953,1.860364)
(0.145036,1.831732)(0.145118,1.803101)(0.145201,1.717230)
(0.145283,1.688599)(0.145366,1.602729)(0.145449,1.516860)
(0.145531,1.373752)(0.145614,1.402364)(0.145696,1.373736)
(0.145779,1.345109)(0.145861,1.259244)(0.145944,1.230618)
(0.146027,1.144754)(0.146109,1.058892)(0.146192,0.915793)
(0.146274,0.887169)(0.146357,0.801310)(0.146439,0.715451)
(0.146522,0.572357)(0.146605,0.486501)(0.146687,0.343411)
(0.146770,0.200322)(0.146852,0.000000)(0.146935,0.143085)
(0.147018,0.228935)(0.147100,0.314784)(0.147183,0.343399)
(0.147265,0.429246)(0.147348,0.457860)(0.147430,0.486473)
(0.147513,0.457854)(0.147595,0.543699)(0.147678,0.572312)
(0.147761,0.600924)(0.147843,0.572305)(0.147926,0.600917)
(0.148008,0.572298)(0.148091,0.543680)(0.148173,0.457834)
(0.148256,0.543674)(0.148339,0.572285)(0.148421,0.600896)
(0.148504,0.572279)(0.148586,0.600889)(0.148669,0.572272)
(0.148751,0.543656)(0.148834,0.457813)(0.148917,0.486423)
(0.148999,0.457807)(0.149082,0.429192)(0.149164,0.343352)
(0.149247,0.314737)(0.149329,0.228898)(0.149412,0.143061)
(0.149494,0.000000)(0.149577,0.085835)(0.149660,0.114447)
(0.149742,0.143057)(0.149825,0.114445)(0.149907,0.143056)
(0.149990,0.114444)(0.150072,0.085832)(0.150155,0.000000)
(0.150237,0.028611)(0.150320,0.000000)(0.150403,0.028610)
(0.150485,0.114440)(0.150568,0.143049)(0.150650,0.228877)
(0.150733,0.314705)(0.150815,0.457750)(0.150898,0.429138)
(0.150980,0.457744)(0.151063,0.486351)(0.151145,0.572174)
(0.151228,0.600779)(0.151311,0.686601)(0.151393,0.772422)
(0.151476,0.915457)(0.151558,0.944060)(0.151641,1.029878)
(0.151723,1.115695)(0.151806,1.258725)(0.151888,1.344540)
(0.151971,1.487567)(0.152053,1.630593)(0.152136,1.830831)
(0.152219,1.745001)(0.152301,1.716385)(0.152384,1.687769)
(0.152466,1.716365)(0.152549,1.687749)(0.152631,1.716345)
(0.152714,1.744941)(0.152796,1.830747)(0.152879,1.802132)
(0.152961,1.830726)(0.153044,1.859321)(0.153126,1.945125)
(0.153209,1.973718)(0.153292,2.059520)(0.153374,2.145321)
(0.153457,2.288330)(0.153539,2.259712)(0.153622,2.288303)
(0.153704,2.316894)(0.153787,2.402691)(0.153869,2.431281)
(0.153952,2.517076)(0.154034,2.602871)(0.154117,2.745870)
(0.154199,2.774457)(0.154282,2.860248)(0.154364,2.946039)
(0.154447,3.089033)(0.154529,3.174821)(0.154612,3.317812)
(0.154695,3.460801)(0.154777,3.660992)(0.154860,3.632370)
}--(0.154860,0);%
\draw[-latex,line width=0.5pt]
(0.136279,0) -- (0.154860,0)node[right]{\FONTSIZE{$x$}};%
\draw[-latex,line width=0.5pt] (.136279,0) -- (.136279,3.6323700)
node[right]{\FONTSIZE{$F(x)$}};%
\end{tikzpicture}
};

\node(E) [above=3.5cm,left=4.5cm]{
\begin{tikzpicture}[yscale=600,xscale=150]
\newcommand{\FONTSIZE}{\fontsize{5pt}{\baselineskip}\selectfont}
\foreach \x in { 0.395,0.400,0.405,0.410,0.415,0.420 }%
\draw[shift={(\x, 0)}]
(0, -0.0001) -- (0, 0)node[below=1pt]{\FONTSIZE{$\x$}};%
\foreach \y in { 0,0.001,0.002,0.003,0.004,0.005}%
\draw[shift={(0.39232, \y)}]
(-0.0004, 0) -- (0, 0)node[left]{\FONTSIZE{$\y$}};%
\filldraw[blue!60,line width=.5pt] plot coordinates{(.39232,0)
(0.39245,0.000325)(0.39259,0.000558)(0.39272,0.000790)
(0.39285,0.000930)(0.39299,0.001162)(0.39312,0.001301)
(0.39326,0.001441)(0.39339,0.001487)(0.39352,0.001719)
(0.39366,0.001858)(0.39379,0.001998)(0.39393,0.002044)
(0.39406,0.002183)(0.39419,0.002229)(0.39433,0.002275)
(0.39446,0.002229)(0.39460,0.002461)(0.39473,0.002600)
(0.39486,0.002739)(0.39500,0.002785)(0.39513,0.002924)
(0.39527,0.002970)(0.39540,0.003016)(0.39553,0.002970)
(0.39567,0.003108)(0.39580,0.003155)(0.39594,0.003201)
(0.39607,0.003154)(0.39620,0.003200)(0.39634,0.003153)
(0.39647,0.003107)(0.39660,0.002967)(0.39674,0.003199)
(0.39687,0.003338)(0.39701,0.003476)(0.39714,0.003522)
(0.39727,0.003661)(0.39741,0.003707)(0.39754,0.003753)
(0.39767,0.003706)(0.39781,0.003845)(0.39794,0.003891)
(0.39808,0.003937)(0.39821,0.003890)(0.39834,0.003936)
(0.39848,0.003890)(0.39861,0.003843)(0.39874,0.003704)
(0.39888,0.003842)(0.39901,0.003888)(0.39914,0.003934)
(0.39928,0.003887)(0.39941,0.003933)(0.39954,0.003887)
(0.39968,0.003840)(0.39981,0.003701)(0.39995,0.003747)
(0.40008,0.003700)(0.40021,0.003654)(0.40035,0.003515)
(0.40048,0.003468)(0.40061,0.003329)(0.40075,0.003190)
(0.40088,0.002959)(0.40101,0.003189)(0.40115,0.003328)
(0.40128,0.003466)(0.40141,0.003512)(0.40155,0.003650)
(0.40168,0.003696)(0.40181,0.003742)(0.40195,0.003695)
(0.40208,0.003834)(0.40221,0.003880)(0.40235,0.003925)
(0.40248,0.003879)(0.40261,0.003925)(0.40275,0.003878)
(0.40288,0.003832)(0.40301,0.003693)(0.40315,0.003831)
(0.40328,0.003877)(0.40341,0.003922)(0.40354,0.003876)
(0.40368,0.003922)(0.40381,0.003875)(0.40394,0.003829)
(0.40408,0.003690)(0.40421,0.003736)(0.40434,0.003689)
(0.40448,0.003643)(0.40461,0.003504)(0.40474,0.003458)
(0.40488,0.003319)(0.40501,0.003181)(0.40514,0.002950)
(0.40527,0.003088)(0.40541,0.003134)(0.40554,0.003179)
(0.40567,0.003133)(0.40581,0.003179)(0.40594,0.003132)
(0.40607,0.003086)(0.40621,0.002948)(0.40634,0.002993)
(0.40647,0.002947)(0.40660,0.002901)(0.40674,0.002762)
(0.40687,0.002716)(0.40700,0.002578)(0.40713,0.002439)
(0.40727,0.002209)(0.40740,0.002255)(0.40753,0.002209)
(0.40767,0.002163)(0.40780,0.002024)(0.40793,0.001978)
(0.40806,0.001840)(0.40820,0.001702)(0.40833,0.001472)
(0.40846,0.001426)(0.40860,0.001287)(0.40873,0.001149)
(0.40886,0.000919)(0.40899,0.000781)(0.40913,0.000552)
(0.40926,0.000322)(0.40939,0.000000)(0.40952,0.000230)
(0.40966,0.000368)(0.40979,0.000505)(0.40992,0.000551)
(0.41005,0.000689)(0.41019,0.000735)(0.41032,0.000781)
(0.41045,0.000735)(0.41058,0.000872)(0.41072,0.000918)
(0.41085,0.000964)(0.41098,0.000918)(0.41111,0.000964)
(0.41125,0.000918)(0.41138,0.000872)(0.41151,0.000734)
(0.41164,0.000872)(0.41178,0.000918)(0.41191,0.000963)
(0.41204,0.000917)(0.41217,0.000963)(0.41231,0.000917)
(0.41244,0.000871)(0.41257,0.000734)(0.41270,0.000779)
(0.41283,0.000734)(0.41297,0.000688)(0.41310,0.000550)
(0.41323,0.000504)(0.41336,0.000367)(0.41350,0.000229)}--
(0.41350,0);%

\filldraw[blue!60,line width=.5pt] plot
coordinates{(0.41706,0)(0.41706,0.001646)
(0.41719,0.001783)(0.41733,0.002011)(0.41746,0.002148)
(0.41759,0.002376)(0.41772,0.002604)(0.41785,0.002924)
(0.41798,0.002787)(0.41812,0.002741)(0.41825,0.002695)
(0.41838,0.002740)(0.41851,0.002694)(0.41864,0.002740)
(0.41877,0.002785)(0.41891,0.002922)(0.41904,0.002876)
(0.41917,0.002921)(0.41930,0.002967)(0.41943,0.003103)
(0.41957,0.003149)(0.41970,0.003285)(0.41983,0.003422)
(0.41996,0.003650)(0.42009,0.003604)(0.42022,0.003649)
(0.42036,0.003694)(0.42049,0.003831)(0.42062,0.003876)
(0.42075,0.004012)(0.42088,0.004149)(0.42101,0.004376)
(0.42114,0.004422)(0.42128,0.004558)(0.42141,0.004694)
(0.42154,0.004922)(0.42167,0.005058)(0.42180,0.005285)
(0.42193,0.005513)(0.42206,0.005831)(0.42220,0.005785)
(0.42233,0.005830)}--(0.42233,0);%
\filldraw[black!80,line width=.5pt] plot coordinates{(.40939,0)
(0.40952,0.000230)(0.40966,0.000368)(0.40979,0.000505)
(0.40992,0.000551)(0.41005,0.000689)(0.41019,0.000735)
(0.41032,0.000781)(0.41045,0.000735)(0.41058,0.000872)
(0.41072,0.000918)(0.41085,0.000964)(0.41098,0.000918)
(0.41111,0.000964)(0.41125,0.000918)(0.41138,0.000872)
(0.41151,0.000734)(0.41164,0.000872)(0.41178,0.000918)
(0.41191,0.000963)(0.41204,0.000917)(0.41217,0.000963)
(0.41231,0.000917)(0.41244,0.000871)(0.41257,0.000734)
(0.41270,0.000779)(0.41283,0.000734)(0.41297,0.000688)
(0.41310,0.000550)(0.41323,0.000504)(0.41336,0.000367)
(0.41350,0.000229)(0.41363,0.000000)(0.41376,0.000137)
(0.41389,0.000183)(0.41402,0.000229)(0.41416,0.000183)
(0.41429,0.000229)(0.41442,0.000183)(0.41455,0.000137)
(0.41469,0.000000)(0.41482,0.000046)(0.41495,0.000000)
(0.41508,0.000046)(0.41521,0.000183)(0.41535,0.000229)
(0.41548,0.000366)(0.41561,0.000503)(0.41574,0.000732)
(0.41587,0.000686)(0.41601,0.000732)(0.41614,0.000778)
(0.41627,0.000915)(0.41640,0.000960)(0.41653,0.001097)
(0.41667,0.001235)(0.41680,0.001463)(0.41693,0.001509)
}--(0.41693,0);%
\draw[-latex,line width=0.5pt] (0.39232,0) -- (0.42233,0)
node[right]{\FONTSIZE{$x$}};%
\draw[-latex,line width=0.5pt] (0.39232,0) -- (0.39232, 0.005830)
node[right]{\FONTSIZE{$F(x)$}};%
\end{tikzpicture}
};
\draw[-latex,red!60,line width=.5pt] (-10.8,6.2) -- (-10.8,4.6);%
\draw[-latex,red!60,line width=.5pt] (-11.5,2.1) -- (-9.5,2.1);%
\draw[-latex,red!60,line width=.5pt] (-6.6,2.1) -- (-10.5,-1.1);%
\draw[-latex,red!60,line width=.5pt] (-11.3,-1.6) -- (-9.5,-1.6);%
\end{tikzpicture}}}
\caption{The fractal nature of the function $|F(x)|$.}\label{fig10}
\end{figure}
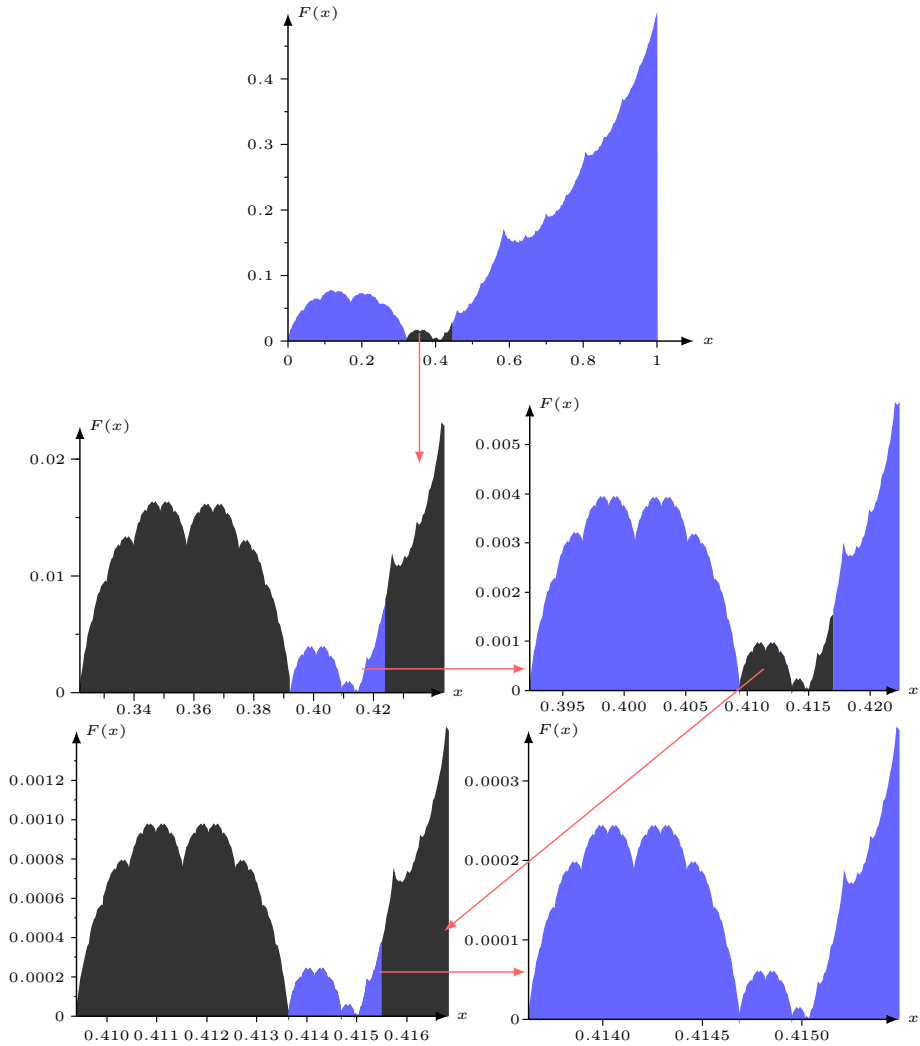

Taking $m=2$, we obtain a refined estimate with smaller errors.
\begin{cor}
\begin{align} \label{dtv3}
\begin{split}
    &\sum_{0\le k\le \lambda}\left|\mathbb{P}
    \bigl(X_n=k\bigr)-2^{-\lambda}\binom{\lambda}{k}
    + \left(\mathbb{E}(X_n)-\frac{\lambda}2\right)
    2^{-\lambda}\binom{\lambda}{k}\frac{\lambda+1-2k}
    {\lambda+1-k}\right|\\
    &\qquad= \frac{16|F_2(\log_2n)|}{\sqrt{2\pi e}\,\log_2 n}+
    O\left((\log n)^{-3/2}\right),
\end{split}
\end{align}
where $F_2(x)$ is defined for $x\in(0,1]$ by
\begin{align*}
    F_2(x) &= 2^{-x}\sum_{j\ge0} 2^{-d_j}
    \left(\frac{d_j(d_j+5)}8
  -\frac{jd_j}2+
    \frac{j(j-3)}2+1\right),
\end{align*}
by writing $2^x = \sum_{j\ge0}2^{-d_j}$ as above, and $F_2(x+1)=
F_2(x)$ for other values of~$x$
(see Figure~\ref{fig11} for a plot).
\end{cor}

\begin{figure}[t]
\begin{tikzpicture}[yscale=0.7]
\newcommand{\FONTSIZE}{\fontsize{8pt}{\baselineskip}\selectfont}
\draw[-latex] (-0.75, 0) -- (5.5, 0) node[right] {};
\foreach \x/\xtext in { 0.000, 0.500, ..., 5.000  }
    \draw[shift={(\x,0.000)}] (0,0.025) -- (0,-0.025);
\foreach \x/\xtext in { 0.000, 0.100, ..., 5.000  }
    \draw[shift={(\x,0.000)}] (0,0.0125) -- (0,-0.0125);
\draw (1 , 0.05) -- (1 , -0.050) node[below] {\FONTSIZE{$0.2$}}
(2 , 0.05) -- (2 , -0.050) node[below] {\FONTSIZE{$0.4$}}
(3 , 0.05) -- (3 , -0.050) node[below] {\FONTSIZE{$0.6$}}
(4 , 0.05) -- (4 , -0.050) node[below] {\FONTSIZE{$0.8$}}
(5 , 0.05) -- (5 , -0.050) node[below] {\FONTSIZE{$1$}};%
\draw[-latex] (0, -0.75) -- (0 , 5.75) node[right]
{
};
\foreach \y/\ytext in { 0.000, 0.500, ..., 5.000  }
    \draw[shift={(0.000, \y)}] (0.025, 0) -- (-0.025, 0);
\foreach \y/\ytext in { 0.000, 0.100, ..., 5.000  }
    \draw[shift={(0.000, \y)}] (0.0125, 0) -- (-0.0125, 0);
\draw (0.050 , 1) -- (-0.050 , 1) node[left] {\FONTSIZE{$0.03$}}
(0.050 , 2) -- (-0.050 , 2) node[left] {\FONTSIZE{$0.06$}}
(0.050 , 3) -- (-0.050 , 3) node[left] {\FONTSIZE{$0.08$}}
(0.050 , 4) -- (-0.050 , 4) node[left] {\FONTSIZE{$0.11$}}
(0.050 , 5) -- (-0.050 , 5) node[left] {\FONTSIZE{$0.14$}};%
\definecolor{gen}{rgb}{0,0,0}%
\filldraw[top color=yellow!20,bottom color=black] plot coordinates{
(0.000, 0.000)(0.014, 0.477)(0.028, 0.705)(0.042, 0.932)(0.056, 0.983)
(0.070, 1.209)(0.084, 1.259)(0.098, 1.309)(0.112, 1.254)(0.126, 1.478)
(0.140, 1.527)(0.153, 1.576)(0.167, 1.521)(0.181, 1.570)(0.195, 1.515)
(0.208, 1.461)(0.222, 1.372)(0.236, 1.592)(0.249, 1.641)(0.263, 1.689)
(0.276, 1.635)(0.290, 1.682)(0.303, 1.628)(0.317, 1.575)(0.330, 1.487)
(0.344, 1.535)(0.357, 1.482)(0.371, 1.428)(0.384, 1.342)(0.397, 1.289)
(0.411, 1.203)(0.424, 1.118)(0.437, 1.066)(0.451, 1.280)(0.464, 1.327)
(0.477, 1.374)(0.490, 1.322)(0.503, 1.369)(0.516, 1.318)(0.530, 1.266)
(0.543, 1.181)(0.556, 1.228)(0.569, 1.177)(0.582, 1.126)(0.595, 1.043)
(0.608, 0.992)(0.621, 0.909)(0.634, 0.826)(0.646, 0.776)(0.659, 0.823)
(0.672, 0.774)(0.685, 0.724)(0.698, 0.642)(0.711, 0.593)(0.723, 0.512)
(0.736, 0.431)(0.749, 0.383)(0.761, 0.334)(0.774, 0.254)(0.787, 0.174)
(0.799, 0.127)(0.812, 0.047)(0.825, 0.000)(0.837, 0.047)(0.850, 0.000)
(0.862, 0.204)(0.875, 0.251)(0.887, 0.297)(0.900, 0.250)(0.912, 0.296)
(0.924, 0.249)(0.937, 0.202)(0.949, 0.124)(0.961, 0.170)(0.974, 0.124)
(0.986, 0.077)(0.998, 0.000)(1.011, 0.046)(1.023, 0.123)(1.035, 0.199)
(1.047, 0.245)(1.059, 0.199)(1.072, 0.244)(1.084, 0.289)(1.096, 0.365)
(1.108, 0.410)(1.120, 0.485)(1.132, 0.560)(1.144, 0.604)(1.156, 0.648)
(1.168, 0.722)(1.180, 0.796)(1.192, 0.840)(1.204, 0.913)(1.216, 0.957)
(1.228, 1.000)(1.240, 0.953)(1.251, 0.907)(1.263, 0.950)(1.275, 0.993)
(1.287, 1.066)(1.299, 1.108)(1.310, 1.180)(1.322, 1.252)(1.334, 1.294)
(1.346, 1.336)(1.357, 1.407)(1.369, 1.478)(1.381, 1.519)(1.392, 1.590)
(1.404, 1.631)(1.415, 1.672)(1.427, 1.626)(1.439, 1.667)(1.450, 1.736)
(1.462, 1.806)(1.473, 1.846)(1.485, 1.915)(1.496, 1.955)(1.507, 1.995)
(1.519, 1.949)(1.530, 2.018)(1.542, 2.057)(1.553, 2.097)(1.564, 2.051)
(1.576, 2.090)(1.587, 2.044)(1.598, 1.999)(1.610, 1.812)(1.621, 1.625)
(1.632, 1.580)(1.643, 1.535)(1.655, 1.575)(1.666, 1.531)(1.677, 1.570)
(1.688, 1.610)(1.699, 1.677)(1.710, 1.633)(1.721, 1.672)(1.733, 1.711)
(1.744, 1.778)(1.755, 1.817)(1.766, 1.884)(1.777, 1.950)(1.788, 1.988)
(1.799, 1.944)(1.810, 1.982)(1.821, 2.021)(1.832, 2.086)(1.843, 2.124)
(1.853, 2.189)(1.864, 2.254)(1.875, 2.292)(1.886, 2.329)(1.897, 2.394)
(1.908, 2.458)(1.919, 2.495)(1.929, 2.559)(1.940, 2.596)(1.951, 2.632)
(1.962, 2.588)(1.972, 2.544)(1.983, 2.580)(1.994, 2.617)(2.004, 2.680)
(2.015, 2.716)(2.026, 2.779)(2.036, 2.841)(2.047, 2.877)(2.058, 2.913)
(2.068, 2.975)(2.079, 3.037)(2.089, 3.072)(2.100, 3.134)(2.110, 3.169)
(2.121, 3.204)(2.131, 3.160)(2.142, 3.195)(2.152, 3.256)(2.163, 3.316)
(2.173, 3.351)(2.184, 3.411)(2.194, 3.446)(2.204, 3.480)(2.215, 3.436)
(2.225, 3.496)(2.235, 3.530)(2.246, 3.564)(2.256, 3.520)(2.266, 3.553)
(2.277, 3.510)(2.287, 3.466)(2.297, 3.294)(2.307, 3.251)(2.318, 3.284)
(2.328, 3.318)(2.338, 3.378)(2.348, 3.411)(2.358, 3.470)(2.369, 3.529)
(2.379, 3.562)(2.389, 3.595)(2.399, 3.654)(2.409, 3.712)(2.419, 3.745)
(2.429, 3.803)(2.439, 3.835)(2.449, 3.868)(2.459, 3.824)(2.469, 3.857)
(2.479, 3.914)(2.489, 3.971)(2.499, 4.004)(2.509, 4.060)(2.519, 4.092)
(2.529, 4.124)(2.539, 4.081)(2.549, 4.138)(2.559, 4.169)(2.569, 4.201)
(2.578, 4.158)(2.588, 4.189)(2.598, 4.146)(2.608, 4.104)(2.618, 3.938)
(2.628, 3.970)(2.637, 4.026)(2.647, 4.082)(2.657, 4.113)(2.667, 4.168)
(2.676, 4.199)(2.686, 4.230)(2.696, 4.188)(2.705, 4.243)(2.715, 4.274)
(2.725, 4.305)(2.734, 4.263)(2.744, 4.293)(2.754, 4.251)(2.763, 4.209)
(2.773, 4.047)(2.783, 4.102)(2.792, 4.133)(2.802, 4.163)(2.811, 4.122)
(2.821, 4.152)(2.830, 4.111)(2.840, 4.070)(2.849, 3.909)(2.859, 3.940)
(2.868, 3.899)(2.878, 3.858)(2.887, 3.699)(2.897, 3.659)(2.906, 3.500)
(2.915, 3.342)(2.925, 3.019)(2.934, 2.862)(2.944, 2.823)(2.953, 2.784)
(2.962, 2.816)(2.972, 2.777)(2.981, 2.809)(2.990, 2.840)(3.000, 2.895)
(3.009, 2.856)(3.018, 2.887)(3.027, 2.919)(3.037, 2.973)(3.046, 3.004)
(3.055, 3.058)(3.064, 3.112)(3.074, 3.143)(3.083, 3.104)(3.092, 3.135)
(3.101, 3.165)(3.110, 3.219)(3.119, 3.249)(3.129, 3.302)(3.138, 3.355)
(3.147, 3.385)(3.156, 3.415)(3.165, 3.468)(3.174, 3.521)(3.183, 3.550)
(3.192, 3.603)(3.201, 3.632)(3.210, 3.662)(3.219, 3.623)(3.228, 3.585)
(3.237, 3.614)(3.246, 3.643)(3.255, 3.695)(3.264, 3.724)(3.273, 3.776)
(3.282, 3.827)(3.291, 3.856)(3.300, 3.885)(3.309, 3.936)(3.318, 3.987)
(3.327, 4.016)(3.336, 4.067)(3.344, 4.095)(3.353, 4.123)(3.362, 4.085)
(3.371, 4.113)(3.380, 4.164)(3.389, 4.214)(3.397, 4.242)(3.406, 4.292)
(3.415, 4.320)(3.424, 4.347)(3.433, 4.309)(3.441, 4.359)(3.450, 4.386)
(3.459, 4.414)(3.467, 4.376)(3.476, 4.403)(3.485, 4.365)(3.494, 4.327)
(3.502, 4.181)(3.511, 4.143)(3.520, 4.171)(3.528, 4.198)(3.537, 4.247)
(3.545, 4.275)(3.554, 4.324)(3.563, 4.372)(3.571, 4.400)(3.580, 4.427)
(3.588, 4.475)(3.597, 4.524)(3.605, 4.550)(3.614, 4.599)(3.623, 4.625)
(3.631, 4.652)(3.640, 4.614)(3.648, 4.641)(3.657, 4.689)(3.665, 4.737)
(3.674, 4.763)(3.682, 4.810)(3.690, 4.837)(3.699, 4.863)(3.707, 4.825)
(3.716, 4.872)(3.724, 4.898)(3.733, 4.924)(3.741, 4.887)(3.749, 4.913)
(3.758, 4.876)(3.766, 4.839)(3.774, 4.697)(3.783, 4.723)(3.791, 4.770)
(3.799, 4.816)(3.808, 4.842)(3.816, 4.889)(3.824, 4.914)(3.833, 4.940)
(3.841, 4.903)(3.849, 4.949)(3.857, 4.975)(3.866, 5.000)(3.874, 4.963)
(3.882, 4.989)(3.890, 4.952)(3.899, 4.915)(3.907, 4.776)(3.915, 4.822)
(3.923, 4.847)(3.931, 4.873)(3.940, 4.836)(3.948, 4.862)(3.956, 4.825)
(3.964, 4.789)(3.972, 4.651)(3.980, 4.677)(3.988, 4.641)(3.996, 4.605)
(4.004, 4.468)(4.013, 4.433)(4.021, 4.296)(4.029, 4.160)(4.037, 3.882)
(4.045, 3.847)(4.053, 3.873)(4.061, 3.899)(4.069, 3.945)(4.077, 3.971)
(4.085, 4.017)(4.093, 4.063)(4.101, 4.088)(4.109, 4.114)(4.117, 4.159)
(4.125, 4.204)(4.133, 4.230)(4.141, 4.275)(4.149, 4.300)(4.157, 4.325)
(4.164, 4.291)(4.172, 4.316)(4.180, 4.361)(4.188, 4.405)(4.196, 4.430)
(4.204, 4.475)(4.212, 4.499)(4.220, 4.524)(4.227, 4.490)(4.235, 4.534)
(4.243, 4.558)(4.251, 4.583)(4.259, 4.549)(4.267, 4.573)(4.274, 4.539)
(4.282, 4.505)(4.290, 4.373)(4.298, 4.397)(4.305, 4.441)(4.313, 4.485)
(4.321, 4.510)(4.329, 4.553)(4.336, 4.577)(4.344, 4.602)(4.352, 4.568)
(4.360, 4.611)(4.367, 4.635)(4.375, 4.659)(4.383, 4.625)(4.390, 4.649)
(4.398, 4.616)(4.406, 4.582)(4.413, 4.452)(4.421, 4.495)(4.428, 4.519)
(4.436, 4.543)(4.444, 4.510)(4.451, 4.534)(4.459, 4.500)(4.467, 4.467)
(4.474, 4.339)(4.482, 4.363)(4.489, 4.330)(4.497, 4.297)(4.504, 4.169)
(4.512, 4.136)(4.519, 4.009)(4.527, 3.882)(4.534, 3.623)(4.542, 3.648)
(4.549, 3.691)(4.557, 3.734)(4.564, 3.758)(4.572, 3.802)(4.579, 3.826)
(4.587, 3.850)(4.594, 3.818)(4.602, 3.861)(4.609, 3.885)(4.617, 3.909)
(4.624, 3.877)(4.631, 3.901)(4.639, 3.869)(4.646, 3.837)(4.654, 3.712)
(4.661, 3.755)(4.668, 3.779)(4.676, 3.803)(4.683, 3.771)(4.691, 3.795)
(4.698, 3.763)(4.705, 3.732)(4.713, 3.608)(4.720, 3.632)(4.727, 3.601)
(4.735, 3.570)(4.742, 3.447)(4.749, 3.416)(4.756, 3.294)(4.764, 3.172)
(4.771, 2.922)(4.778, 2.965)(4.786, 2.989)(4.793, 3.013)(4.800, 2.983)
(4.807, 3.007)(4.814, 2.977)(4.822, 2.947)(4.829, 2.826)(4.836, 2.850)
(4.843, 2.820)(4.851, 2.791)(4.858, 2.670)(4.865, 2.641)(4.872, 2.521)
(4.879, 2.402)(4.886, 2.157)(4.894, 2.181)(4.901, 2.152)(4.908, 2.123)
(4.915, 2.005)(4.922, 1.976)(4.929, 1.858)(4.936, 1.740)(4.943, 1.498)
(4.951, 1.470)(4.958, 1.352)(4.965, 1.236)(4.972, 0.995)(4.979, 0.878)
(4.986, 0.638)(4.993, 0.398)(5.000, 0.000)};%
\end{tikzpicture}
\caption{$|F_2(x)|$.}\label{fig11}
\end{figure}

The two functions $F(x)$ and $F_2(x)$ may assume the value zero when
$x$ is not an integer; see Figures~\ref{fig10} and \ref{fig11}.
This means that in such cases the error term
is of a smaller order, and the right-hand side of our result gives
simply an $O$-estimate. One naturally wonders if there are other
simple uniform approximants for the total variation distance. We
propose a simple one in the following.

\begin{thm} \label{spec-seq}
The total variation distance between the distribution of $X_n$ and
the binomial distribution of $Y_\lambda$ satisfies\vadjust{\eject}
\[
    \dtv(\mathscr{L}(X_n),\mathscr{L}(Y_\lambda))
    \asymp \frac1{2^{\lambda-\lambda_2}}
    \min\left\{1, \frac{\lambda-\lambda_2}
    {\sqrt{\lambda}}\right\},
\]
whenever $\lambda-\lambda_2\ge c$, where $c>0$ is sufficiently large.
\end{thm}
This result is similar to the estimate proved by Soon
\cite{soon1993a} (see also \cite{chen94a}), where he considered the
distance $\dtv(\mathscr{L}(X_n), \mathscr{L}(Y_{\lambda+1}))$
instead of $\dtv(\mathscr{L}(X_n),\break \mathscr{L}(Y_{\lambda}))$ by
using Stein's method.

We see roughly that the wider the gap between $\lambda$ and
$\lambda_2$, the smaller the total variation distance is.

On the other hand, the theorem fails when $c=2$. In this case,
$n=2^\lambda+2^{\lambda-2}$ and by (\ref{dtv3}) or by a direct
calculation,
\begin{align} \label{dtv4}
    \dtv(\mathscr{L}(X_n),\mathscr{L}(Y_\lambda))
    \asymp\lambda^{-1}.
\end{align}

More generally, if
\[
    n = 2^{\lambda} + 2^{\lambda-2} +\cdots + 2^{\lambda-2d}
    + n_0,
\]
where $d\ge1$ and $n_0 = O(n/\lambda^{3/2})$, then $F(\log_2n)
=O(\lambda^{-1/2})$, and (\ref{dtv4}) holds.

All these results can be extended to $\nu_q(n)$. The major
difference is to use generating function (\ref{GBGF}) instead of
(\ref{BGF}).

\section{Proofs}

We first prove Theorem~\ref{thm-digits} by a direct analytic
approach based on Fourier analysis. A closely connected semigroup
approach (first developed by Deheuvels and Pfeifer for Poisson
distribution; see \cite{deheuvels86a}), but relies on more algebraic
formulations and manipulations, can also be used for the same
purpose; see \cite{roos01a}. Then Theorem~\ref{spec-seq} is proved
by two different approaches: one by Stein's method, and the other by
a standard probability argument, which starts from decomposing the
distribution of $X_n$ into a sum of binomial distributions. Our
adaptation of Stein's method indeed leads to a refinement of
Theorem~\ref{spec-seq}, which will be given in
Section~\ref{sec:stein}. For more methodological interests, we also
include another approach using the Krawtchouk polynomials and the
Parseval identity, which is the binomial analogue of the
Charlier-Parseval approach developed earlier in detail in
\cite{zacharovas10a}.

\begin{table}[!t]
\caption{A summary of approaches used and results proved in this section}\label{tab2}
\begin{tabular}{|c|c|c|}
\hline
Approach               & Result                                     & Section         \\\hline\hline
Analytic               & Thm \ref{thm-digits} (for extended $\dtv$) & \ref{sec:ana}   \\\hline
Elementary Probability & Thm \ref{spec-seq} (for $\dtv$)            & \ref{sec:ele}   \\\hline
Stein's method         & Thm \ref{spec-seq} (for $\dtv$)            & \ref{sec:stein} \\\hline
Krawtchouk-Parseval    & $\chi^2$-distance                          & \ref{sec:kraw}  \\\hline
\end{tabular}
\end{table}

\subsection{Analytic approach: Proof of Theorem \ref{thm-digits}}
\label{sec:ana}

We now prove Theorem \ref{thm-digits} and write the proof in a more
general way that can be readily amended for dealing with other cases
such as Gray codes; see Section~\ref{sec-gray} below.

\subsubsection{Probability generating function}
Let
\begin{align}\label{Pny}
    P_n(y)
    :=\mathbb{E}\left(y^{X_n}\right)
    =\frac{1}{n}\sum_{0\le j<n}y^{\nu_2(j)}.
\end{align}
Then
\begin{align} \label{Jnz}
    P_{2n}(y)
    =\frac{1+y}{2}P_n(y),
\end{align}
and $P_{2^k}(y)=(1+y)^k/2^k$. Note that $\nu_2(n) \le \lambda+1$.

For convenience, let
\[
    Q_n(y)
    =nP_n(y).
\]
In terms of $Q_n$, the relation (\ref{Jnz}) has the form
\[
    Q_{2n}(y)
    =(1+y)Q_n(y).
\]
For odd numbers, we have
\[
    Q_{2n+1}(y)
    =(1+y)Q_n(y)+y^{\nu_2(n)}.
\]
These two recurrences can be written as
\[
    Q_n(y)
    =(1+y)Q_{\tr{n/2}}(y)+\delta_n y^{\nu_2(\tr{n/2})},
\]
for all $n\ge 0$, where
\[
    \delta_n
    =\frac{1-(-1)^n}{2}.
\]
By iteration, we then get
\begin{align}\label{Qnz}
\begin{split}
    Q_n(y)
    &=\sum_{0\le j\le \lambda} \delta_{\tr{n/2^j}}
    y^{\nu_2\left(\tr{n/2^{j+1}}\right)} (1+y)^j\\
    &=(1+y)^{\lambda}\sum_{0\le j\le \lambda}
    \delta_{\tr{2^{\{\log_2n\}+j}}}
    \frac{y^{\nu_2\left(\tr{2^{\{\log_2 n\}+j-1}}\right)}}
    {(1+y)^{j}},
\end{split}
\end{align}
for any $n\ge 1$; compare (\ref{BGF}). This means that $P_n$ has the
form
\[
    P_n(y)
    =\left(\frac{1+y}2\right)^{\lambda}\phi_n(y),
\]
where
\[
    \phi_n(y)
    =\frac{2^\lambda}n \sum_{0\le j\le \lambda}
    \delta_j\frac{y^{\rho_j}}{(1+y)^j},
\]
and $\rho_j$ are nonnegative integers such that $\rho_j\le j$ and
$|\delta_j|\le 1$.

\subsubsection{Local expansion of \texorpdfstring{$\phi_n(y)$}{phi n(y)}}

The approach we use here relies on the intuition that if $\phi_n$ is
sufficiently ``smooth'' then $X_n$ is close to the binomial
distribution $Y_\lambda$. More precisely, let
\[
    \phi_n(y)
    =\sum_{j\ge0} a_j(n)(y-1)^j;
\]
cf.\ (\ref{brn}).

\begin{lem}
\label{Lestima_a_n} For each $m\ge 1$, we have
\begin{align} \label{phi-exp}
    \left|\phi_n(y)-\sum_{0\le r<m}a_r(n)(y-1)^r\right|
    \le\frac 32\cdot\frac{2^\lambda(2|y-1|)^m}{n(1-2|y-1|)},
\end{align}
if $|y-1|\le 1/2-\ve$, $\ve>0$ being an arbitrarily small number.
\end{lem}
\begin{proof}
We indeed prove a stronger estimate
\[
    \frac{n}{2^\lambda}|a_r(n)|\le 3\cdot 2^{r-1},
\]
for all $r\ge1$, which then implies (\ref{phi-exp}).

Let $[y^r]f(y)$ denote the coefficient of $y^r$ in the Taylor
expansion of $f$. Since $\rho_j\le j$, we have, for $r\ge1$,
\begin{align*}
    \frac{n}{2^\lambda}|a_r(n)|
    &=\left|[w^r]\sum_{0\le j\le \lambda}
    \delta_j\frac{(1+w)^{\rho_j}}{(2+w)^j}\right|\\
    &\le[w^r]\sum_{j\ge1}
    \left(\frac{1+w}{2-w}\right)^j \\
    &=[w^r]\frac{1+w}{1-2w}\\
    &= 3\cdot 2^{r-1},
\end{align*}
as required.
\end{proof}

\subsubsection{An asymptotic expansion for $\mathbb{P}(X_n=k)$}
\begin{prop}\label{loc-exp}
For all integer $0\le r\le \lambda$ and each $m\ge1$, we have
\[
    \mathbb{P}(X_n=k)
    =\frac{1}{2^\lambda}\sum_{0\le r<m}(-1)^r
    a_r(n)\Delta^r\binom{\lambda}{k}
    +O\left(\frac{2^{3m/2}\Gamma((m+1)/2)}{\lambda^{(m+1)/2}}
    \right),
\]
uniformly in $k$.
\end{prop}
\begin{proof}
By Cauchy's integral formula for the coefficient of an analytic
function, we have
\begin{align*}
    \mathbb{P}(X_n=k)
    &=\frac{1}{2\pi i}\oint_{|y|=1}y^{-n-1}P_n(y)\dd y\\
    &= \frac{1}{2\pi}\left(\int_{-1/2}^{1/2}
    +\int_{1/2\le|t|\le\pi}\right)e^{-kit}
    \left(\frac{1+e^{it}}2\right)^{\lambda}\phi_n(e^{it})\dd t\\
    &=:I_1+I_2.
\end{align*}
Since $e^{i/2}$ lies inside the circle $|y-1|=1/2$, we evaluate
$I_1$ by applying Lemma~\ref{Lestima_a_n} and obtain
\begin{align*}
    I_1
    &=\frac{1}{2\pi}\int_{-1/2}^{1/2}
    \left(\frac{1+e^{it}}2\right)^{\lambda}\sum_{0\le r<m}
    a_r(n) e^{-kit}(e^{it}-1)^r\dd t \\
    &\qquad + O\left(\int_{-1/2}^{1/2}
    \left|\frac{1+e^{it}}2\right|^{\lambda}
    |1-e^{it}|^m\dd t\right).
\end{align*}
The integral in the $O$-term is then estimated as follows.
\begin{align} \label{I-O}
\begin{split}
    \int_{-1/2}^{1/2} \left|\frac{1+e^{it}}2\right|^{\lambda}
    |1-e^{it}|^m\dd t
    &=2^m\int_{-1/2}^{1/2} \left(\cos
    \tfrac t2\right)^{\lambda}\left|\sin\tfrac t2\right|^m\dd t\\
    &\le 2^{m+1}\int_{0}^{1}
    (1-t)^{(\lambda-1)/2}t^{(m-1)/2}\dd t\\
    &=O\left(\frac{2^{3m/2}\Gamma((m+1)/2)}
    {\lambda^{(m+1)/2}}\right).
\end{split}
\end{align}
Now substituting this estimate into the expression of $I_1$ and
using the relation
\begin{align}  \label{Drk}
    \Delta^r\binom{\lambda}{k}
    &=\frac{1}{2\pi}\int_{-\pi}^{\pi}e^{-kit}
    (1+e^{it})^{\lambda}(1-e^{it})^r \dd t,
\end{align}
we obtain
\begin{align*}
    I_1
    &=\sum_{0\le r<m}\frac{a_r(n)}{2\pi }\int_{-1/2}^{1/2}
    \left(\frac{1+e^{it}}2\right)^{\lambda}
    (e^{it}-1)^r e^{-kit}\dd t \\ &\qquad +
    O\left(\frac{2^{3m/2}\Gamma((m+1)/2)}
    {\lambda^{(m+1)/2}}\right)\\
    &=\frac{1}{2^\lambda}\sum_{0\le r<m}(-1)^r
    a_r(n)\Delta^r\binom{\lambda}{k} \\ &\qquad +
    \sum_{0\le r<k}\frac{|a_r(n)|}{2\pi}\int_{1/2\le|t|\le\pi}
    \left|\frac{1+e^{it}}2\right|^{\lambda}|e^{it}-1|^r\dd t\\
    &\qquad+O\left(\frac{2^{3m/2}\Gamma((m+1)/2)}
    {\lambda^{(m+1)/2}}\right)\\
    &=\frac{1}{2^\lambda}\sum_{0\le r<m}(-1)^r a_r(n)
    \Delta^r\binom{\lambda}{k}
    +O\left(4^m \left(\cos \tfrac{1}{4}\right)^{\lambda}\right)
    \\ &\qquad + O\left(\frac{2^{3m}\Gamma((m+1)/2)}
    {\lambda^{(m+1)/2}}\right).
\end{align*}
On the other hand, since (by (\ref{BGF}))
\begin{align*}
    \max_{1/2\le |t|\le\pi}|P_n(e^{it})|
    &\le \frac1n \sum_{1\le j\le s}
    \left|1+e^{i/2}\right|^{\lambda_j}\\
    &\le \frac1n\sum_{1\le j\le s} 2^{\lambda_j}
    \exp\left(-\frac{\lambda_j}{8\pi^2}\right)\\
    &\le \frac1n\sum_{\lambda_j\ge c_0\lambda}
    2^{\lambda_j}  \exp\left(-\frac{c_0\lambda}
    {8\pi^2}\right) + \frac1n\sum_{\lambda_j<c_0\lambda}
    2^{ c_0 \lambda}\\
    &= O\left(\exp\left(-\frac{c_0\lambda}
    {8\pi^2}\right) + \lambda 2^{-(1-c_0)\lambda}\right).
\end{align*}
Choosing
\[
     c_0 = \frac{\log2}{\log2 + 1/(8\pi^2)},
\]
so as to balance the two terms in the $O$-symbol, we obtain
\[
    \max_{1/2\le |t|\le\pi}|P_n(e^{it})|
    = O\left(\lambda e^{-c_0'\lambda} \right),
\]
where
\[
    c_0' = \frac{\log 2}{1+8\pi^2\log 2}.
\]
Thus
\[
    I_2 = O\left(\lambda e^{-c_0'\lambda}\right).
\]
This proves the proposition.
\end{proof}

\subsubsection{Estimates for the differences of binomial coefficients}

\begin{lem} For $r\ge 0$, we have
\begin{align}
    2^{-\lambda}\max_{0\le k \le \lambda}
    \left|\Delta^r\binom{\lambda}{k}\right|
    &=O\left(\frac{2^{3r/2}\Gamma((r+1)/2)}
    {\lambda^{(r+1)/2}}\right),\nonumber \\
    2^{-\lambda}\sum_{0\le k\le \lambda}\left|\Delta^r
    \binom{\lambda}{k}\right|
    &=\frac{h_r}{\lambda^{r/2}}
    \left(1+O\left(\lambda^{-1}\right)\right), \label{hj2}
\end{align}
where $h_r$ is defined in (\ref{hm}).
\end{lem}
\begin{proof}
By (\ref{Drk}) and an analysis similar to that used in (\ref{I-O}),
we obtain
\begin{align*}
    \max_{0\le k \le \lambda}
    \left|\Delta^r\binom{\lambda}{k}\right|
    &\le\frac{2^{\lambda+r}}{2\pi}\int_{-\pi}^{\pi}
    \left(\cos\tfrac t2\right)^{\lambda}
    \left|\sin \tfrac t2\right|^r \dd t,\\
    &=\frac{2^{\lambda+r}}{\pi}
    \int_{0}^{1} (1-t)^{(\lambda-1)/2}t^{(r-1)/2} \dd t\\
    &=O\left(\frac{2^{\lambda+3r/2}
    \Gamma((r+1)/2)}{\lambda^{(r+1)/2}}\right).
\end{align*}
For the proof of (\ref{hj2}), we apply the standard saddle-point
method and obtain
\begin{align*}
    &2^{-\lambda}\sum_{0\le k\le \lambda}
    \left|\Delta^r \binom{\lambda}{k}
    \right|\\ &\quad= \sum_{0\le k\le \lambda}
    \left|\frac{1}{2\pi}\int_{-\pi}^{\pi}
    \left(\frac{1+e^{it}}{2}\right)^{\lambda}(1-e^{it})^r
    e^{-kit}\dd t \right|\\
    &\quad= \sum_{\substack{k=\lambda/2+x\sqrt{\lambda}/2\\
    x=o(\lambda^{1/6})}} \left|\frac{2^{r+1}}{2\pi
    \lambda^{(r+1)/2}}\int_{-\infty}^\infty
    (-it)^r e^{-t^2/2-xit}\dd t\right|
    \left(1+O\left(\lambda^{-1}\right)\right)\\
    &\quad= \frac{2^{r/2}}{\sqrt{2\pi}\,\lambda^{r/2}}
    \int_{-\infty}^\infty \left|H_r(x)\right| e^{-x^2/2}\dd{x}
    \left(1+O\left(\lambda^{-1}\right)\right),
\end{align*}
proving (\ref{hj2}) by (\ref{hm}).
\end{proof}

Note that when $r=1$, we have the closed form expression
\[
    \sum_{0\le k\le \lambda}\left|\binom{\lambda}k
    -\binom{\lambda}{k-1}\right|
    = 2\binom{\lambda}{\tr{\lambda/2}}-1.
\]
For higher values of $r$, a closed-form expression can be derived
for\break  $\sum_{0\le k\le \lambda}|\Delta^r \binom{\lambda}k|$ in terms
of the zeros of Krawtchouk polynomials; see \cite{roos01a} for
$r=2$.

\subsubsection{Proof of Theorem \ref{thm-digits}}
\begin{proof}
Applying Proposition~\ref{loc-exp}, we get
\[
    \sum_{0\le k\le \lambda}\Biggl|\mathbb{P}(X_n=k)
    -2^{-\lambda}\sum_{0\le r\le m+1}
    (-1)^r a_r(n)\Delta_r\binom{\lambda}{k}\Biggr|
    =O\biggl(\frac{2^m\Gamma((m+3)/2)}
    {\lambda^{(m+1)/2}}\biggr).
\]
Note that the sum over all $k$ for the terms corresponding to
$r=m+1$ is of order $\lambda^{-(m+1)/2}$. Theorem~\ref{thm-digits}
then follows from (\ref{hj2}).
\end{proof}

We will later formulate a simple framework of numeration systems for
which the same type of results as $X_n$ hold, using the same method
of proofs.

\subsection{Elementary probability approach: Proof of
Theorem~\ref{spec-seq}} \label{sec:ele}

A crucial observation that will be elaborated here is the fact that
$X_n$ is itself a mixture of binomial distributions. More precisely,
by the decomposition of K\'atai and Mogyor\'odi (\ref{km}),
\begin{align} \label{km-2}
    \mathbb{P}\bigl(X_n=k\bigr)
    =\frac1n\sum_{1\le j\le s}2^{\lambda_j}
    \mathbb{P}(Y_{\lambda_j}=k-j+1).
\end{align}

A direct probabilistic proof of the above relation is as follows.
Suppose $U_n$ is a uniformly distributed of $[0,n-1]$, then by
definition $X_n=\nu_2(U_n)$. First, we have the relation
\[
    X_{2^k}\stackrel{d}{=} Y_k \qquad(k=1,2,\dots).
\]
On the other hand, since $\nu_2(2^r+j)=1+\nu_2(j)$ if $0\le j<2^r$,
we also have
\[
    j+X_{2^{\lambda_{j+1}}} \stackrel{d}{=}
    j+Y_{\lambda_{j+1}}\qquad(0\le j<s).
\]
We can now split the interval $\{0,1,\ldots,n-1\}
=\bigcup_{j=0}^{s-1}A_s$, where $A_0=[0,2^{\lambda})$ and
\begin{align} \label{Aj}
    A_j=\left[\sum_{1\le r\le j} 2^{\lambda_r},
    \sum_{1\le r\le j+1} 2^{\lambda_r}\right),
\end{align}
for $1\le j\le s-1$. Clearly, $\mathbb{P}(U_n\in A_j)=
2^{\lambda_{j+1}}/n$. We then obtain (\ref{km-2}).

We group in the following lemma a few simple properties of the total
variation distances involving $Y_k$, which will be needed later.
\begin{lem}\label{bernoulli}
Let $Y_k$ be a binomial random variable with mean parameters $k$ and
$1/2$. Then
\begin{align*}
    \dtv(\mathscr{L}(Y_k),\mathscr{L}(Y_k+1))
    &=O\left(k^{-1/2}\right),\\
    \dtv(\mathscr{L}(Y_k),\mathscr{L}(Y_{k+1}))
    &=\tfrac{1}{2}\dtv(\mathscr{L}(Y_k),
    \mathscr{L}(Y_k+1))=O\left(k^{-1/2}\right),\\
    \dtv(\mathscr{L}(Y_k),\mathscr{L}(Y_{k+j}+\ell))
    &= O\left((j+\ell)k^{-1/2}\right).
\end{align*}
\end{lem}
\begin{proof}
Since $\mathbb{P}(Y_k=j)$ increases monotonically in the interval
from $\bigl[0,\tr{k/2}\bigr)$ and decreases monotonically in the
interval $\bigl(\tr{k/2},k\bigr]$, we have
\[
    \dtv(\mathscr{L}(Y_k),\mathscr{L}(Y_k+1))
    =2\mathbb{P}\bigl(Y_k=\tr{k/2}\bigr)
    =O\left(k^{-1/2}\right).
\]
In a similar way, since $\mathbb{P}(Y_{k+1}=j) =\mathbb{P}(Y_k+I=j)
=\bigl(\mathbb{P}(Y_k=j-1) +\break \mathbb{P}(Y_k=j)\bigr)/2$, where $I$ is
Bernoulli with parameter $1/2$, we get
\begin{align*}
    \dtv(\mathscr{L}(Y_k),\mathscr{L}(Y_{k+1}))
    &=\tfrac{1}{4}\sum_{j\ge0}
    \bigl|\mathbb{P}(Y_k=j)-\mathbb{P}(Y_k=j-1)\bigr| \\
    &=\tfrac{1}{2}\dtv(\mathscr{L}(Y_k),\mathscr{L}(Y_k+1)).
\end{align*}
This proves the lemma. \end{proof}

\subsubsection{Proof of Theorem \ref{spec-seq} when
\texorpdfstring{$\lambda-\lambda_2\le \sqrt{\lambda}$}{lambda-lambda2<=square root of lambda}}

Consider first the case when $\lambda-\lambda_2\le \sqrt{\lambda}$.
By (\ref{km-2}) and Lemma~\ref{bernoulli}, we have
\begin{align}
    &\dtv(\mathscr{L}(X_n),\mathscr{L}(Y_{\lambda}))\nonumber \\
    &\quad= \frac1{2n}\sum_{\ell\ge0} \left|
    \sum_{1\le j\le s} 2^{\lambda_j}\left(
    \mathbb{P}(Y_{\lambda_j}=\ell-j+1)
    -\mathbb{P}(Y_{\lambda}=\ell)\right)\right|\label{dtv-pa}\\
    &\quad\le\frac1{n}\sum_{2\le j\le s} 2^{\lambda_j}
    \dtv(\mathscr{L}(Y_{\lambda_j}+j-1),
    \mathscr{L}(Y_{\lambda}))\nonumber \\
    &\quad\le 2\sum_{2\le j\le s} \frac{\lambda-\lambda_j}
    {2^{\lambda-\lambda_j}\sqrt{\lambda_j}}\nonumber\\
    &\quad\le 2\sum_{\lambda-\lambda_2\le k\le \lambda-\lambda_s}
    \frac{k}{2^k\sqrt{\lambda-k}} \nonumber\\
    &\quad= O\left(\frac{\lambda-\lambda_2}
    {2^{\lambda-\lambda_2}
    \sqrt{\lambda_2}}\right),\nonumber
\end{align}
giving an upper bound for the total variation distance.\eject

To obtain a lower bound, we apply again (\ref{dtv-pa}) and
Lemma~\ref{bernoulli}.
\begin{align*}
    \dtv(\mathscr{L}(X_n),\mathscr{L}(Y_{\lambda}))
    &\ge \frac{n-2^\lambda}{2n}
    \dtv(\mathscr{L}(Y_{\lambda_2}+1),
    \mathscr{L}(Y_{\lambda}))\\
    &\qquad - \frac1{2n}\sum_{3\le j\le s} 2^{\lambda_j}
    \dtv(\mathscr{L}(Y_{\lambda_j+j-1}),
    \mathscr{L}(Y_{\lambda_2+1}))\\
    &\ge \frac{n-2^\lambda}{2n}
    \dtv(\mathscr{L}(Y_{\lambda_2}+1),
    \mathscr{L}(Y_{\lambda})) \\
    &\qquad +O\left(\frac1n\sum_{3 \le j\le s}
    \frac{2^{\lambda_j}}{\sqrt{\lambda_j}}\,
    (\lambda_2-\lambda_j+j)\right)\\
    &\ge \frac{n-2^\lambda}{2n}
    \dtv(\mathscr{L}(Y_{\lambda_2}+1),
    \mathscr{L}(Y_{\lambda})) \\
    &\qquad+O\left(\frac{ 2^{\lambda_3}(\lambda_2-\lambda_3)}
    {2^{\lambda}\sqrt{\lambda_3}}\right).
\end{align*}
Now, by Lemma~\ref{bernoulli},
\[
    \dtv(\mathscr{L}(Y_{\lambda_2}+1),
    \mathscr{L}(Y_{\lambda}))
    =\dtv(\mathscr{L}(Y_{\lambda_2}),
    \mathscr{L}(Y_{\lambda}))+O(\lambda^{-1/2}).
\]

\subsubsection{Proof of Theorem \ref{spec-seq} when
\texorpdfstring{$c\le \lambda-\lambda_2\le \sqrt{\lambda}$}{c<=lambda-lambda2<=square root of lambda}}

For $c\le \lambda-\lambda_2\le \sqrt{\lambda}$,
where $c>0$ is sufficiently large, we have
\begin{align*}
    \dtv(\mathscr{L}(Y_{\lambda_2}),
    \mathscr{L}(Y_{\lambda}))
    &\ge\mathbb{P}\left(Y_{\lambda_2}\le
    \sqrt{\lambda\lambda_2}/2\right)-\mathbb{P}
    \left(Y_{\lambda}\le\sqrt{\lambda\lambda_2}/2\right)
    \\&=\Phi\left(\frac{\sqrt{\lambda\lambda_2}/2-\lambda_2/2}
    {\sqrt{\lambda_2}/2}\right)-\Phi\left(
    \frac{\sqrt{\lambda\lambda_2}/2-\lambda/2}
    {\sqrt{\lambda}/2}\right)\\ &\qquad+O(\lambda^{-1/2})\\
    &=\Phi\left(\sqrt{\lambda}-\sqrt{\lambda_2}\right)
    -\Phi\left(\sqrt{\lambda_2}-\sqrt{\lambda}\right)
    +O\left(\lambda^{-1/2}\right)\\
    &\ge \ve\frac{\lambda-\lambda_2}{\sqrt{\lambda}},
\end{align*}
for $\ve>0$, by the central limit theorem of the binomial
distribution (with rate). Combining the upper- and the lower-bounds,
we get
\[
    \dtv(\mathscr{L}(X_n),\mathscr{L}(Y_{\lambda}))
    \asymp\frac{\lambda-\lambda_2}{2^{\lambda-\lambda_2}
    \sqrt{\lambda}},
\]
if  $c\le \lambda-\lambda_2\le \sqrt{\lambda}$ and $c$ is
sufficiently large.

\subsubsection{Proof of Theorem~\ref{spec-seq} when
\texorpdfstring{$(\lambda-\lambda_2) /\sqrt{\lambda}\to \infty$}{(lambda-lambda2)/square root of lambda->infinity}}

The lower bound becomes less precise if $(\lambda-\lambda_2)
/\sqrt{\lambda}\to \infty$. In this case, we first observe that the
total variation does not exceed $1$; thus
\[
    \dtv(\mathscr{L}(X_n),\mathscr{L}(Y_{\lambda}))
    \le\frac{n-2^\lambda}{n}\le
    \frac{2}{2^{\lambda-\lambda_2}}.
\]
Take $C=(\lambda-\lambda_2)/\sqrt{\lambda}$. We have
\begin{align*}
    &\dtv(\mathscr{L}(Y_{\lambda_2}+1),
    \mathscr{L}(Y_{\lambda})) \\
    &\quad\ge\mathbb{P}\left(Y_{\lambda}\ge \frac{\lambda}2-
    \frac{C}4\sqrt{\lambda}\right)
    -\mathbb{P}\left(Y_{\lambda_2}+1
    \ge \frac\lambda2- \frac{C}4\sqrt{\lambda}\right)\\
    &\quad\ge \mathbb{P}\left(\left|Y_{\lambda}-\frac\lambda2
    \right|\le \frac C4\sqrt{\lambda}\right)
    -\mathbb{P}\left(\left|Y_{\lambda_2}-\frac{\lambda_2}2
    \right|\ge \frac{\lambda-\lambda_2}{2}-1-
    \frac C4\sqrt{\lambda}\right)\\
    &\quad= \mathbb{P}\left(\left|Y_{\lambda}-\frac\lambda2\right|
    \le \frac C4\sqrt{\lambda}\right)-
    \mathbb{P}\left(\left|Y_{\lambda_2}-\frac{\lambda_2}2
    \right|\ge \frac C4\sqrt{\lambda}-1\right).
\end{align*}
Applying Chebyshev's inequality, we get
\[
    1\ge \dtv(\mathscr{L}(Y_{\lambda_2}+1),
    \mathscr{L}(Y_{\lambda})) \ge1+O\left(C^{-1}\right),
\]
if $C\ge 8$.

When $\lambda_3<\lambda_2-1$, we have the lower bound
\begin{align*}
    \dtv(\mathscr{L}(X_n),\mathscr{L}(Y_{\lambda}))
    &\ge\frac{2^{\lambda_2}\dtv(\mathscr{L}(Y_{\lambda_2}+1),
    \mathscr{L}(Y_{\lambda}))-2^{\lambda_3}
    -\cdots-2^{\lambda_s}}{n}\\
    &\ge\frac{\dtv(\mathscr{L}(Y_{\lambda_2}+1),
    \mathscr{L}(Y_{\lambda}))-1/2}{2^{\lambda-\lambda_2}}\\
    &\ge\frac{1}{2^{\lambda-\lambda_2+1}}
    \bigl(1+O\left(C^{-1}\right)\bigr).
\end{align*}
On the other hand, when $\lambda_3=\lambda_2-1$, we use
(\ref{dtv-pa}) and get
\begin{align*}
    &2\dtv(\mathscr{L}(X_n),\mathscr{L}(Y_{\lambda})) \\
    &\quad\ge \frac1n\sum_{\ell\ge0}\left|
    2^{\lambda_2}\bigl(\mathbb{P}(Y_{\lambda_2}=\ell-1)
    -\mathbb{P}(Y_{\lambda}=\ell)\bigr)
    \right. \\ &\qquad\left.{}
    +2^{\lambda_3}\bigl(\mathbb{P}(Y_{\lambda_3}=\ell-2)
    -\mathbb{P}(Y_{\lambda}=\ell)\bigr)\right|
    -\frac{2^{\lambda_4}+\cdots+2^{\lambda_s}}{n}\\
    &\quad\ge \frac1n\sum_{\ell\ge0}\left|
    (2^{\lambda_2}+2^{\lambda_3})
    \bigl(\mathbb{P}(Y_{\lambda_2}=\ell-1)
    -\mathbb{P}(Y_{\lambda}=\ell)\bigr) \right.\\
    &\quad\qquad \left.
    {}+2^{\lambda_3}\bigl(\mathbb{P}(Y_{\lambda_3}=\ell-2)
    -\mathbb{P}(Y_{\lambda_2}=\ell-1)\bigr)\right|
    -\frac{2^{\lambda_4+1}}{n}\\
    &\quad\ge \frac1n\left((2^{\lambda_2}+2^{\lambda_3})
    \dtv(\mathscr{L}(Y_{\lambda_2}+1),
    \mathscr{L}(Y_{\lambda}))-2^{\lambda_3+1}\right)\\
    &\quad\ge\frac{1}{2^{\lambda-\lambda_2+1}}
    \bigl(1+O\left(C^{-1}\right)\bigr).
\end{align*}
This completes the proof of Theorem \ref{spec-seq}.

\subsection{Stein's method: An alternative proof of
Theorem~\ref{spec-seq}}
\label{sec:stein}

The sum-of-digits function was among one of the first instances used
to demonstrate the effectiveness of Stein's method (see
\cite{diaconis77a,stein1986a}) with an optimal approximation rate.
This method centers on exploiting an equation that characterizes the
limiting measure, which, in the case of binomial distribution, is
given by \eqref{bino-stein-eq} and can be derived in the following
way.

\subsubsection{Stein's equation for binomial distribution}

Since $Y_k$ is binomially distributed with parameters $k$ and
$p\in(0,1)$, we see that the probabilities $\mathbb{P}(Y_k=j)$
satisfy the difference equation
\begin{align}\label{bino-df}
    q\bigl((j+1)\mathbb{P}(Y_k=j+1)-j\mathbb{P}(Y_k=j)\bigr)
    = (p(k-j)-jq)\mathbb{P}(Y_k=j).
\end{align}
Following Stein's idea \cite{stein72a} for deriving the
characteristic equation for the normal distribution
\[
    \mathbb{E}\left(f'(\mathscr{N})\right)
    =\mathbb{E}\left(\mathscr{N}\!\!
    f(\mathscr{N})\right)\qquad (f\in C^1(\mathbb{R})),
\]
by using integration by parts, we consider the average
\[
    q\mathbb{E}\bigl(g(Y_k)-g(Y_k-1)\bigr)Y_k
    =q\sum_{0\le j\le k} \bigl(g(j)-g(j-1)\bigr)\mathbb{P}(Y_k=j),
\]
and apply summation by parts, which yields, by \eqref{bino-df},
\[\begin{split}
    &q\mathbb{E}\bigl(g(Y_k)-g(Y_k-1)\bigr)Y_k \\
    &\quad=q\sum_{0\le j\le k}g(j)q\bigl(j\mathbb{P}(Y_k=j)
    -(j+1)\mathbb{P}(Y_k=j)\bigr)\\
    &\quad=\sum_{0\le j\le k}g(j)\mathbb{P}(Y_k=j)
    (jq-p(k-j))\mathbb{P}(Y_k=j)\\
    &\quad=q\mathbb{E}\left(Y_kg(Y_k)\right)
    -p\mathbb{E}\left((k-Y_k)g(Y_k)\right).
\end{split}
\]
Thus the identity
\begin{align}\label{bino-st-eq}
    q\mathbb{E}\left(Y_kg(Y_k-1)\right)
    =p\mathbb{E}\left((k-Y_k)g(Y_k)\right)
\end{align}
holds for any function $g:\{0,1,2,\ldots,k\}\to \mathbb{R}$.

A simpler proof of \eqref{bino-st-eq} starts with the relation
\begin{align}\label{bino-alt}
    q(j+1)\mathbb{P}(Y_k=j+1)=p(k-j)\mathbb{P}(Y_k=j)
    \qquad(0\le j<k),
\end{align}
multiply both sides by $g(j)$, and sum over all indices $j$, giving
rise to
\[
    q\sum_{0\le j\le k}(j+1)g(j)\mathbb{P}(Y_k=j+1)
    =p\sum_{0\le j\le k}(k-j)g(j)\mathbb{P}(Y_k=j),
\]
which is nothing but \eqref{bino-st-eq}.

Conversely, if for some discrete random variable $Z$ the identity
\[
    \mathbb{E}\bigl(qZg(Z-1)+p(Z-k)g(Z)\bigr)=0
\]
holds for any function $g(j)$, then the probabilities
$\mathbb{P}(Z=j)$ satisfy the equation (\ref{bino-alt}) as
$\mathbb{P}(Y_k=j)$. Thus
\[
    \mathbb{P}(Z=j)=\mathbb{P}(Y_k=j).
\]

\subsubsection{Binomial approximation}

In the special case when $p=q=1/2$, we have
\begin{align} \label{stein-eq0}
    \mathbb{E}\bigl((Y_k-k)g(Y_k)+Y_kg(Y_k-1)\bigr)
    =0.
\end{align}
Thus we expect that the above quantity will be small for any random
variable whose distribution is close to $\Bi(k,1/2)$. Assume
$h:\{0,1,\ldots,\lambda\}\to \mathbb{C}$ is an arbitrary function.
Let $g$ be a solution to the recurrence relation
\begin{equation}\label{stein-eq}
    (x-\lambda)g(x)+xg(x-1)
    =h(x)-\mathbb{E}(h(Y_\lambda)).
\end{equation}
Then we can represent the difference of means as
\[
    \mathbb{E}(h(X_n))-\mathbb{E}(h(Y_\lambda))
    =\mathbb{E}\bigl((X_n-\lambda)g(X_n)+X_ng(X_n-1)\bigr).
\]
By Stein's equation (\ref{stein-eq}), the expectation on the
right-hand side of the above identity will be zero if $X_n$ were
distributed according to binomial distribution $B(\lambda,1/2)$.
Thus we expect that this quantity will be small if the distribution
of $X_n$ is close to $B(\lambda,1/2)$.

Recalling that $A_j$ is defined in (\ref{Aj}), we see that
$\mathbb{P}(X_n<x|U_n\in A_j)=\mathbb{P}(Y_{\lambda_j}+j-1<x)$. It
follows that
\begin{align*}
    &\mathbb{E}\left(h\bigl(X_{n}\bigr)\right)
    -\mathbb{E}(h(Y_\lambda)) \\
    &\quad=\mathbb{E}\bigl((X_n-\lambda)g(X_n)
    +X_ng(X_n-1)\bigr)\\
    &\quad=\sum_{0\le j<s}\mathbb{P}(U_n\in A_j)
    \mathbb{E}\bigl((X_n-\lambda)
    g(X_n)+X_n g(X_n-1)|U_n\in A_j\bigr)\\
    &\quad=\sum_{1\le j\le s}\frac{2^{\lambda_j}}{n}
    \mathbb{E}\bigl((Y_{\lambda_j}+j-1)
    -\lambda)g(Y_{\lambda_j}+j-1) \\
    &\qquad\qquad \qquad
    +(Y_{\lambda_j}+j-1)g(Y_{\lambda_j}+j-2)\bigr).
\end{align*}
The term with $j=1$ in the last sum is zero since $Y_{\lambda}$ is
binomially distributed. Hence, with $g_j(x):=g(x+j-1)$, we then
obtain
\begin{align*}
    &\mathbb{E}\left(h\bigl(X_{n}\bigr)\right)
    -\mathbb{E}(h(Y_\lambda)) \\
    &\quad=\sum_{2\le j\le s}\frac{2^{\lambda_j}}{n}\mathbb{E}
    \bigl(({\lambda_{j}}-\lambda+j-1)g_j(Y_{\lambda_{j}})
    +(j-1)g_j(Y_{\lambda_{j}}-1)\bigr)\\
    &\qquad+\sum_{2\le j\le s}\frac{2^{\lambda_j}}{n}
    \mathbb{E}\bigl((Y_{\lambda_{j}}
    -\lambda_j)g_j(Y_{\lambda_{j}})
    +Y_{\lambda_{j}}g_j(Y_{\lambda_{j}}-1)\bigr).
\end{align*}
But the second sum is identically zero by (\ref{stein-eq0}).
It follows that
\begin{align*}
    &\mathbb{E}\bigl(h(X_n)\bigr)-\mathbb{E}(h(Y_\lambda))\\
    &\quad=\sum_{2\le j\le s}\frac{2^{\lambda_j}}{n}
    \mathbb{E}\bigl(({\lambda_j}-\lambda+j-1)
    g(Y_{\lambda_j}+j-1)+(j-1)g(Y_{\lambda_j}+j-2)\bigr),
\end{align*}
which can be alternatively rewritten as
\begin{equation}\label{stein_eq_new}
    \mathbb{E}\bigl(h(X_n)\bigr)-\mathbb{E}(h(Y_\lambda))
    =\mathbb{E}\bigl(Q_1g(X_n)\bigr)
    +\mathbb{E}\bigl(Q_2g(X_n-1)\bigr),
\end{equation}
where $Q_1$ is a random variable taking value $\lambda_j -\lambda
+j-1$ if $U_n\in A_j$ and $Q_2$ takes value $j-1$ if $U_n\in A_j$
for $1\le j\le s$. Note that
\begin{align*}
    \mathbb{E}(Q_1)&=\sum_{1\le j\le s}\mathbb{P}(U_n\in
    A_j)(\lambda_j-\lambda+j-1) \\
    &=\sum_{1\le j\le s}
    \frac{2^{\lambda_j}}{n}(\lambda-\lambda_j+j-1)\\
    &=O\left(\frac{\lambda-\lambda_2}{2^{\lambda-\lambda_j}}\right),
\end{align*}
and, similarly, $\mathbb{E}(Q_2)
=O\left(1/2^{\lambda-\lambda_j}\right)$.

\subsubsection{Solving the equation \texorpdfstring{$(x-\lambda)g(x)+xg(x-1)=h(x)-\mathbb{E}(h(Y_\lambda))$}{
(x-lambda)g(x)+xg(x-1)=h(x)-E(h(Y lambda))}}

Solving the equation (\ref{stein-eq}) is equivalent to finding the
solution $x_m$ of the difference equation
\[
    (k-m)x_m-mx_{m-1}=\delta_m,
\]
for $1\le m\le k$ (note that $x_{-1}$ and $x_n$ do not affect the
solution of this equation and therefore can be assumed to be equal
to zero), where the $\delta_m$'s are given and satisfy the condition
$$\delta_0\binom{k}{0}+\delta_1\binom{k}{1}+\cdots
+\delta_k \binom{k}{k} =0.$$ The solution is obtained by introducing
new variables $z_m=(k-m)\binom{k}{m}x_m$ for which our difference
equation takes form
\[
    z_m-z_{m-1}=\binom{k}{m}\delta_m,
\]
Iterating this, we obtain the following solution to Stein's equation
\eqref{stein-eq}.
\begin{lem}[\!\!\cite{stein1986a}]
\label{stein_eq-sol} Let $Y_\lambda\sim \Bi(\lambda,1/2)$.
Define the function $g:\{0,1,\ldots,\break \lambda-1\}\to
\mathbb{C}$ by
\begin{align*}
    g(m) =\frac{1}{(\lambda-m)\binom{\lambda}{m}}
    \sum_{0\le r\le m}\binom{\lambda}{r}
    \bigl(h(r)-\mathbb{E}(h(Y_\lambda))\bigr).
\end{align*}
Then $g$ is the solution to the recurrence equation
\[
    (x-\lambda)g(x)+xg(x-1)
    =h(x)-\mathbb{E}(h(Y_\lambda)),
\]
for all $x\in \{1,\ldots,\lambda-1\}$.
\end{lem}
Note that
\begin{align*}
    g(m)
    =-\frac{1}{(\lambda-m)\binom{\lambda}{m}}
    \sum_{m<r\le \lambda}\binom{\lambda}{r}
    \bigl(h(r)-\mathbb{E}(h(Y_\lambda))\bigr).
\end{align*}
\begin{lem} The sequence
\[
    y_m
    :=\frac{1}{\binom{\lambda}{m}}
    \sum_{0\le r\le m}\binom{\lambda}{r}
\]
is monotonically increasing in $m$.
\end{lem}
\begin{proof}
By induction using the recurrence relation
\[
    y_{m+1}
    =\frac{m+1}{\lambda-m}y_m+1,
\]
and the monotonicity of $\frac{m+1}{\lambda-m}$.
\end{proof}

\begin{lem}[\!\!\cite{loh92a},\cite{soon1993a}]
\label{lem_g_bound} If $0\le h(m)\le 1$, then the solution of
Stein's equation provided by Lemma~\ref{stein_eq-sol} satisfies the
uniform estimate
\[
    |g(m)| = O\left(\lambda^{-1/2}\right)
    \qquad(m=0,1,\dots).
\]
\end{lem}
\begin{proof}
If $m\le \lambda/2$, then, by the monotonicity of the sequence
$y_m$, we obtain
\begin{align*}
    |g(m)|
    &\le \frac{1}{(\lambda-m)\binom{\lambda}{m}}
    \sum_{0\le r \le m}\binom{\lambda}{r}
    =\frac{y_m}{\lambda-m}
    \le \frac{y_{\tr{\lambda/2}}}{\lambda-{\tr{\lambda/2}}} \\
    &=\frac{1}{(\lambda-\tr{\lambda/2})
    \binom{\lambda}{\tr{\lambda/2}}}
    \sum_{0\le r\le \tr{\lambda/2}}\binom{\lambda}{r}\\
    &\le \frac{2^{\lambda}}{\tr{\lambda/2}
    \binom{\lambda}{\tr{\lambda/2}}}
    = O\left(\lambda^{-1/2}\right).
\end{align*}
The case when $m>\tr{\lambda/2}$ is treated similarly. Indeed, if
$m>\tr{\lambda/2}$, then, using the identity
\[
    {(k-m)}\binom{k}{m}=(m+1)\binom{k}{k-m-1},
\]
we have
\begin{align*}
    |g(m)|
    &\le \frac{1}{(\lambda-m)\binom{\lambda}{m}}
    \sum_{m< r \le \lambda}\binom{\lambda}{r}\\
    &= \frac{1}{(m+1)\binom{\lambda}{\lambda-m-1}}
    \sum_{0\le r< \lambda-m}\binom{\lambda}{r}
    =\frac{y_{\lambda-m-1}}{m+1}
    \le \frac{y_{\lambda-\tr{\lambda/2}}}{m+1}\\
    &\le \frac{2^{\lambda}}{\tr{\lambda/2}
    \binom{k}{\tr{\lambda/2}}}
    = O\left(\lambda^{-1/2}\right).\qedhere
\end{align*}
\end{proof}

\subsubsection{Proof of Theorem~\ref{spec-seq} by Stein's method}
Assume now that $A\subset \mathbb{R}$ is an arbitrary set. Define
\[
    h(m) := I_A(m) = \left\{
    \begin{array}{ll}
        1, & \text{if } m\in A;\\
        0, & \text{if } m \not\in A.
    \end{array}
    \right.
\]
Then
\begin{align*}
    P\bigl(X_{n}\in A\bigr)-P\bigl(Y_{\lambda}\in A\bigr)
    &=\mathbb{E}\bigl(Q_1g(X_n)\bigr)
    +\mathbb{E}\bigl(Q_2g(X_n-1)\bigr)\\
    &=O\left(\frac{\mathbb{E}(Q_1)+\mathbb{E}(Q_2)}
    {\sqrt{\lambda}}\right)\\
    &= O\left(\frac{\lambda-\lambda_2}
    {2^{\lambda-\lambda_2}\sqrt{\lambda}}\right).
\end{align*}
Thus
\[
    \dtv(\mathscr{L}(X_n),\mathscr{L}(Y_{\lambda}))
    =O\left(\frac{\lambda-\lambda_2}
    {2^{\lambda-\lambda_2}\sqrt{\lambda}}\right).
\]

\subsubsection{A refinement of Theorem~\ref{spec-seq}}

A finer result can be obtained by using the following lemma.
\begin{lem}[\!\!\cite{barbour1992a}] If $0\le h(x)\le 1$ and $g$ is
defined in (\ref{stein_eq-sol}), then
\[
    \max_{1\le j\le \lambda-1}|g(j)-g(j-1)|
    \le 2\min\left\{\frac{1}{j},\frac{1}{\lambda-j}\right\}
    \le \frac{4}{\lambda}.
\]
\end{lem}
\begin{proof}
If $m\le \lambda/2$, then
\[
\begin{split}
    &g(m)-g(m-1) \\
    &\quad =\left(\frac{1}{(\lambda-m)\binom{\lambda}{m}}
    -\frac{1}{(\lambda-m+1)\binom{\lambda}{m-1}}\right)
    \sum_{0\le r< m}\binom{\lambda}{r}
    \bigl(h(r)-\mathbb{E}(h(Y_\lambda))\bigr)\\
    &\qquad +\frac{h(m)-\mathbb{E}(h(Y_\lambda))}{\lambda-m}.
\end{split}
\]
By the elementary inequality
\[
    \binom{\lambda}{m-r}
    \le\left(\frac{m}{\lambda-m+1}\right)^r\binom{\lambda}{m},
\]
we see that
\begin{align*}
    |g(m)-g(m-1)|
    &\le \left(\frac{1}{m}-\frac{1}{\lambda-m}
    \right)\sum_{0\le r\le m-1}
    \frac{\binom{\lambda}{r}}
    {\binom{\lambda}{m}} +\frac{1}{\lambda-m}\\
    &\le \frac{\lambda-2m}{m(\lambda-m)}
    \sum_{1\le r\le m }
    \left(\frac{m}{\lambda-m+1}\right)^r
    +\frac{1}{\lambda-m}\\
    &\le \frac{\lambda-2m}{m(\lambda-m)}\cdot
    \frac{\frac{m}{\lambda-m+1}}
    {1-\frac{m}{\lambda-m+1}}
    +\frac{1}{\lambda-m}\\
    &= \frac{\lambda-2m}{(\lambda-m)(\lambda-2m+1)}
    +\frac{1}{\lambda-m}\\
    &\le \frac{2}{\lambda-m}.
\end{align*}
In a similar way we obtain the estimate
\[
    |g(m)-g(m-1)|\le \frac{2}{m},
\]
in the case when $m>\lambda/2$.
\end{proof}
The following result is similar in nature to that obtained by Soon
\cite{soon1993a} for unbounded function $h(x)$ that he later applied
to derive several large and moderate deviations results for $X_n$.
\begin{prop} Assume $h$ is any real function such that $0\le h(x)
\le 1$. Then \label{prop-st-diff}
\[
    \mathbb{E}(h(X_n))-\mathbb{E}(h(Y_\lambda))
    =4a_1(n)\mathbb{E}\left(h(Y_\lambda)\frac{\lambda/2-Y_\lambda}
    {\lambda}\right)+O\left(\frac{(\lambda-\lambda_2)^2}
    {2^{\lambda-\lambda_2}\lambda}\right),
\]
where $a_1(n)=F(\log_2n)$ is defined in \eqref{Flogn}.
\end{prop}
\begin{proof}
The lemma implies that $g(x+j-1)=g(x)+O(j/k)$. Since $Y_{k+s}$ has
the same distribution as $Y_{k}+W_s$, where $W_s$ is independently
and binomially distributed $W_s\sim B(s,1/2)$, we can replace the
mean $\mathbb{E}(g(Y_{k+s}))$ by $\mathbb{E}(g(Y_{k}+W_s))$, the
error so introduced being bounded above by
\[
    \mathbb{E}(g(Y_{k+s}))-\mathbb{E}(g(Y_k))
    =\mathbb{E}\bigl(g(Y_{k}+W_s)-\mathbb{E}(g(Y_k))\bigr)
    =O\left(\frac{s}{k}\right),
\]
where we used the estimate $|W_s|\le s$. Thus
\begin{align*}
    &\mathbb{E}(h(X_n))-\mathbb{E}(h(Y_\lambda))\\
    &\quad=\sum_{2\le j\le s}\frac{2^{\lambda_j}}{n}
    \mathbb{E}\bigl(({\lambda_j}-\lambda+j-1)
    g(Y_{\lambda_{j}}+j-1) \\ &\qquad+(j-1)
    g(Y_{\lambda_{j}}+j-2)\bigr)\\
    &\quad=\sum_{2\le j\le s}\frac{2^{\lambda_j}}{n}
    \bigl({\lambda_{j}}-\lambda+2(j-1)\bigr)
    \mathbb{E}\left(g(Y_{\lambda_{j}}+j-1)\right) \\
    &\qquad+O\left(\frac{1}{2^{\lambda-\lambda_2}\lambda}
    \right)\\
    &\quad=\mathbb{E}(g(Y_{\lambda}))
    \sum_{2\le j\le s}\frac{2^{\lambda_j}}{n}
    \bigl({\lambda_{j}}-\lambda+2(j-1)\bigr)
    +O\left(\frac{(\lambda-\lambda_2)^2}
    {2^{\lambda-\lambda_2}\lambda}\right)\\
    &\quad= 2a_1(n)\mathbb{E}(g(Y_{\lambda}))
    +O\left(\frac{(\lambda-\lambda_2)^2}
    {2^{\lambda-\lambda_2}\lambda}\right).
\end{align*}

We now evaluate the quantity $\mathbb{E}(g(Y_\lambda))$ appearing in
the last expression
\[
\begin{split}
    \mathbb{E}(g(Y_\lambda))
    &=\frac{1}{2^\lambda}\sum_{0\le m<\lambda}
    \binom{\lambda}{m}\frac{1}{(\lambda-m)
    \binom{\lambda}{m}}
    \sum_{0\le r\le m}\binom{\lambda}{r}
    \bigl(h(r)-\mathbb{E}(h(Y_\lambda))\bigr)\\
    &=\frac{1}{2^\lambda}\sum_{0\le r<\lambda}\binom{\lambda}{r}
    \bigl(h(r)-\mathbb{E}(h(Y_\lambda))\bigr)
    \sum_{r\le m<\lambda}\frac{1}{\lambda-m}\\
    &=\mathbb{E}\Biggl(\bigl(h(Y_\lambda)
    -\mathbb{E}(h(Y_\lambda))\bigr)
    \sum_{Y_\lambda\le m<\lambda}\frac{1}{\lambda-m}\Biggr)\\
    &=\mathbb{E}\Biggl(\bigl(h(Y_\lambda)
    -\mathbb{E}(h(Y_\lambda))\bigr)
    \Biggl(\sum_{Y_\lambda\le m<\lambda}
    \frac{1}{\lambda-m}-
    \sum_{\lambda/2\le m< \lambda}
    \frac{1}{\lambda-m}\Biggr)\Biggr),
\end{split}
\]
since $\mathbb{E}\bigl(h(Y_\lambda)-\mathbb{E}h(Y_\lambda)\bigr)=0$
and the sum $\sum_{\lambda/2\le m<\lambda}(\lambda-m)^{-1}$ is a
constant independent of $Y_\lambda$. Then
\begin{align*}
    \mathbb{E}(g(Y_\lambda))
    =-\sum_{0\le r<\lambda}
    \mathbb{P}(Y_\lambda=r)\bigl(h(r)-\mathbb{E}(h(Y_\lambda))\bigr)
    \sum_{m=\lambda/2}^r\frac{1}{\lambda-m},
\end{align*}
where we use the convention that $\sum_{m=a}^{b}=-\sum_{m=b}^{a}$.
We then split the sum into two parts and obtain
\begin{align*}
    \mathbb{E}(g(Y_\lambda))
    &=-\sum_{|r-\lambda/2|\le \lambda^{3/4}}
    \mathbb{P}(Y_\lambda=r)\bigl(h(r)-\mathbb{E}(h(Y_\lambda))\bigr)
    \sum_{m=\lambda/2}^r\frac{1}{\lambda-m} \\
    &\qquad+O(\mathbb{P}(|Y_\lambda-\lambda/2|>\lambda^{3/4})).
\end{align*}
If $\lambda/2\le r\le \lambda/2+\lambda^{3/4}$, then
\[
\begin{split}
    \sum_{m=\lambda/2}^r\frac{1}{\lambda-m}
    &=\sum_{m=\lambda/2}^r\left(\frac{1}{\lambda-m}
    -\frac{1}{\lambda/2}\right)+\frac{r-\lceil\lambda/2\rceil+1}
    {\lambda/2}\\
    &=\frac{r-\lambda/2}{\lambda/2}
    +\sum_{m=\lambda/2}^r\frac{m-\lambda/2}
    {(\lambda-m)\lambda/2} +O(\lambda^{-1})\\
    &=\frac{r-\lambda/2}{\lambda/2}
    +O\left(\frac{\left(r-\lambda/2\right)^2}
    {\lambda^2}+\lambda^{-1}\right).
\end{split}
\]
The same estimate holds when $r$ lies in the range $\lambda/2-
\lambda^{3/4} \le r\le \lambda/2$. Thus
\begin{align*}
    \mathbb{E}(g(Y_\lambda))
    &=2\mathbb{E}\left(\bigl(h(Y_\lambda)
    -\mathbb{E}(h(Y_\lambda))\bigr)
    \frac{\lambda/2-Y_\lambda}{\lambda}\right)\\
    &\qquad+O\left(\frac{1}{\lambda^2}
    \mathbb{E}\bigl|h(Y_\lambda)
    -\mathbb{E}(h(Y_\lambda))
    \bigr|(Y_\lambda-\lambda/2)^2\right)\\
    &\qquad+O\left(\lambda^{-1}
    +\mathbb{P}(|Y_\lambda-\lambda/2|>\lambda^{3/4})\right)\\
    &=2\mathbb{E}\left(\bigl(h(Y_\lambda)
    -\mathbb{E}(h(Y_\lambda))\bigr)
    \frac{\lambda/2-Y_\lambda}{\lambda}\right)\\
    &\qquad+O\left(\frac{1}{\lambda^2}\mathbb{E}
    (Y_\lambda-\lambda/2)^2\right)+O(\lambda^{-1})\\
    &=2\mathbb{E}\left(\bigl(h(Y_\lambda)
    -\mathbb{E}(h(Y_\lambda))\bigr)
    \frac{\lambda/2-Y_\lambda}{\lambda}\right) +O(\lambda^{-1})\\
    &= 2\mathbb{E}\left(h(Y_\lambda)
    \frac{\lambda/2-Y_\lambda}{\lambda}\right) +O(\lambda^{-1}),
\end{align*}
since \(\mathbb{E}\left(\lambda/2-Y_\lambda\right) =0\). This proves
the proposition.
\end{proof}

\subsubsection{Corollaries of Proposition~\ref{prop-st-diff}}

\begin{cor} We have
\begin{align}\label{cor-dtv-h}
    \dtv(\mathscr{L}(X_n),\mathscr{L}(Y_{\lambda}))
    =|a_1(n)|\mathbb{E}\left|\frac{Y_\lambda -\lambda/2}
    {\lambda/2}\right|+O\left(\frac{(\lambda-\lambda_2)^2}
    {2^{\lambda-\lambda_2}\lambda}\right).
\end{align}
\end{cor}
\begin{proof} By the definition of the total variation distance
\[
    \dtv(\mathscr{L}(X_n),\mathscr{L}(Y_{\lambda}))
    =\sup_{h}\bigl|\mathbb{E}(h(X_n))
    -\mathbb{E}(h(Y_\lambda))\bigr|,
\]
where the supremum is taken over all functions $h$ assuming only
binary values $\{0,1\}$. It is easy to see that the supremum of the
average containing $h$ in the above relation is reached by the
function
\[
    h(x)
    =\begin{cases}
    1,\quad \hbox{if}\quad x\le \lambda/2,\\
    0,\quad \hbox{if}\quad x> \lambda/2,
    \end{cases}
\]
and we thus get, by Proposition~\ref{prop-st-diff}, the estimate
\[
    \dtv(\mathscr{L}(X_n),\mathscr{L}(Y_{\lambda}))
    = |a_1(n)|\mathbb{E}
    \left|\frac{Y_\lambda -\lambda/2}
    {\lambda/2}\right|+O\left(\frac{(\lambda-\lambda_2)^2}
    {2^{\lambda-\lambda_2}\lambda}\right).\qedhere
\]
\end{proof}
\begin{cor} For all $c\le\lambda-\lambda_2$ with $c$ large
enough, we have
\[
    \dtv(\mathscr{L}(X_n),\mathscr{L}(Y_{\lambda}))
    =\frac{|a_1(n)|}{\sqrt{2\pi \lambda}}
    \left(1+ O\left(\frac{\lambda-\lambda_2}
    {\sqrt{\lambda}}\right)\right).
\]
\end{cor}
\begin{proof} This follows from the estimate \eqref{cor-dtv-h}
because
\[
    \mathbb{E}\left|\frac{Y_\lambda -\lambda/2}
    {\sqrt{\lambda}/2}\right|
    =\frac{1}{\sqrt{2\pi}}\int_{-\infty}^\infty
    |x|e^{-x^2/2}\,dx +O(\lambda^{-1/2})
    =\sqrt{\frac{2}{\pi}}+O(\lambda^{-1/2}),
\]
and the quantity $|a_1(n)|$ can be bounded from bellow by
\[
    \frac{\lambda-\lambda_2}{2^{\lambda-\lambda_2}}
    = O(a_1(n)),
\]
if $\lambda-\lambda_2>c$ with $c>0$ large enough.
\end{proof}

\begin{rem}
The Stein method we adopted above for
the analysis of sum-of-digits function differs from the original
approach by Stein \cite{stein1986a}. First, he used $Y_{\lambda+1}$
in lieu of $Y_\lambda$, as a good approximation to $X_n$, and
derived a uniform bound for the point probability. Second, instead
of exploiting the fact that $X_n$ is a mixture of binomial
distributions his analysis is more subtle and is based on the
construction of an exchangeable pair. By doing so he managed to
simplify the right-hand side of \eqref{stein-eq} to
\begin{align} \label{stein-eqo}
\begin{split}
    \mathbb{E}(h(X_n))-\mathbb{E}(h(Y_{\lambda+1}))
    &=\mathbb{E}\bigl((X_n-(\lambda+1))g_h(X_n)+X_ng_h(X_n-1)\bigr)\\
    &=\mathbb{E}(Qg_h(X_n)),
\end{split}
\end{align}
where $Q$ is a random variable such that $0\le \mathbb{E}(Q)\le 2$
and $g_h$ is the solution to the recurrence equation
\[
    h(x)-\mathbb{E}(h(Y_{\lambda+1}))
    =\bigl(x-(\lambda+1)\bigr)g(x)+xg(x-1),
\]
for $x\in\{0,1,\ldots,\lambda\}$ whose precise expression is given
by Lemma \ref{stein_eq-sol} with $k=\lambda+1$. The estimate of
Lemma \ref{lem_g_bound} together with the property $0\le
\mathbb{E}(Q)\le 2$ now immediately give the estimate of the total
variation distance $\dtv(\mathscr{L}(X_n),\break  \mathscr{L}
(Y_{\lambda+1})) =O(\lambda^{-1/2})$. Further applications of
similar ideas will be explored elsewhere.
\end{rem}

\subsection{The Krawtchouk-Parseval approach: \texorpdfstring{$\chi^2$}{chi2}-distance}
\label{sec:kraw}

We examine in this section yet another approach based on properties
of the Krawtchouk polynomials and the Parseval identity (or more
generally Plan\-cherel's~formula). The approach is the binomial
analogue of the Charlier-Parseval approach we developed and explored
earlier in \cite{zacharovas10a}. We consider only the simplest case
of deriving the $\chi^2$-distance, leaving the extension to other
distances to the interested reader, which follows readily from the
framework developed in \cite{zacharovas10a}.

\subsubsection{Krawtchouk polynomials}

Krawtchouk (or Kravchuk) polynomials, introduced in the late 1920s,
are polynomials orthogonal with respect to the binomial
distribution. Over the years, they were frequently encountered in a
variety of areas, including combinatorics, number theory, asymptotic
analysis, image analysis, coding theory, etc. In probability theory,
their appearance is perhaps even more anticipated than in other
areas due to the prevalence of binomial distribution, and sometimes
without noticing the explicit connection; see Diaconis's monograph
\cite{diaconis1988a} for more information and applications. See also
\cite{ismail2009,macwilliams1977,schoutens2000,terras1999} for more
recent update. Despite the large literature on diverse properties of
Krawtchouk polynomials and the high usefulness of the Parseval
identity, we did not find application of the corresponding Parseval
identity similar to ours; see however \cite{diaconis90a,morrison86a}
for a direct manipulation of Fourier integrals.

We start with reviewing the definition of Krawtchouk polynomials and
some of their well-known properties (see \cite[pp.\
35--37]{szego1975a}).\vadjust{\goodbreak}

Assume that $p$ and $q$ are nonnegative integers such that $p+q=1$.

Introduce the notation
\[
    B(N,t)
    =\binom{N}{t}p^tq^{N-t}.
\]
The Krawtchouk polynomials $K_n(t)=K_n(N,t)$ are defined by
\begin{equation}\label{def-kravtchuk}
    \sum_{0\le j\le N}K_j(t)w^j
    =\left(\frac{1+qw}{1-pw}\right)^t(1-pw)^N.
\end{equation}
Multiplying both sides by $B(N,t)z^t$ and summing over all $t$ from
$0$ to $N$, we obtain
\begin{align*}
    \sum_{0\le t\le N}B(N,t)z^t\sum_{0\le j\le N}
    K_j(x)w^j
    &=\bigl(pz(1+qw)+q(1-pw)\bigr)^N\\
    &=\bigl(pz+q+wpq(z-1)\bigr)^N\\
    &=\sum_{0\le j\le N}\binom{N}{j}w^j(pq)^j(z-1)^j(pz+q)^{N-j}.
\end{align*}
Taking the coefficients of $w^n$ on both sides, we get
\[
    \sum_{0\le t\le N}B(N,t)K_n(t)z^t
    =\binom{N}{n}(pq)^n(z-1)^n(pz+q)^{N-n}.
\]
On the other hand, by (\ref{def-kravtchuk}), we have
\begin{align*}
    \sum_{0\le n,m\le N}
    &\left(\sum_{0\le t\le N}B(N,t)K_n(t)K_m(t)\right)w^nz^n\\
    &=\sum_{0\le t\le N}B(N,t)\bigl(({1+qw})({1+qz})\bigr)^t
    \bigl((1-pw)(1-pw)\bigr)^{N-x}\\
    &=\bigl(p({1+qw})({1+qz})+q(1-pw)(1-pw)\bigr)^N\\
    &=(1+pqzw)^N.
\end{align*}
Accordingly, we obtain the orthogonality relation
\[
    \sum_{0\le t\le N}B(N,t)K_n(t)K_m(t)
    =\delta_{m,n}\binom{N}{n}(pq)^n.
\]

\subsubsection{The Parseval identity for Krawtchouk polynomials}

Let $F(z)$ be a polynomial of degree not greater than $N$. Thus
\[
    f(z)
    =\sum_{0\le t\le N}f_t z^t,
\]
and we have the expansion
\begin{equation}\label{k-over-bino}
    \frac{f_t}{B(N,t)}
    =\sum_{0\le j\le N}b_jK_j(t).
\end{equation}
Taking square of the above identity, multiplying it by $B(N,t)$ and
summing the resulting identity with respect to $t$, we obtain
\begin{equation}\label{krav_par}
    \sum_{0\le t\le N}\left|\frac{f_t}{B(N,t)}\right|^2B(N,t)
    =\sum_{0\le j\le N}|b_j|^2\binom{N}{j}(pq)^j.
\end{equation}
By the definition (\ref{def-kravtchuk}), we deduce that
\[
    (1-pw)^Nf\left(\frac{1+qw}{1-pw}\right)
    =\sum_{j=0}^Nb_j\binom{N}{j}(pq)^jw^j.
\]
Comparing this identity with (\ref{krav_par}), we conclude that
\[
    \sum_{0\le t\le N}\left|\frac{f_t}{B(N,t)}\right|^2B(N,t)
    =\sum_{0\le j\le N}\frac{|c_j|^2}{\binom{N}{j}(pq)^j},
\]
where $c_j$ is defined by
\[
    (1-pw)^Nf\left(\frac{1+qw}{1-pw}\right)
    =\sum_{0\le j\le N} c_jw^j.
\]

Now by the Parseval identity
\begin{equation}
\label{define_J}
\begin{split}
    J(f,N;r)
    &:=\frac{1}{2\pi}\int_{-\pi}^\pi
    \left|(1-pre^{it})^N
    f\left(\frac{1+qre^{it}}{1-pre^{it}}\right)\right|^2\dd t\\
    &=\sum_{0\le j\le N}|c_j|^2 r^{2j}.
\end{split}
\end{equation}
Comparing (\ref{define_J}) with (\ref{krav_par}), and using the
relation
\[
    (N+1)\int_0^\infty\frac{u^j}{(1+u)^{N+2}}
    =\binom{N}{j}^{-1},
\]
we obtain the \emph{Krawtchouk-Parseval identity}
\begin{align} \label{kp}
    \sum_{0\le t\le N}\left|\frac{f_t}{B(N,t)}\right|^2B(N,t)
    =(N+1)\int_0^\infty\frac{J\left(f,N;\sqrt{\frac{u}{pq}}\right)}
    {(1+u)^{N+2}}\dd u,
\end{align}
which is crucial for deriving the asymptotics of the
$\chi^2$-distance.

\subsubsection{The \texorpdfstring{$\chi^2$}{chi2}-distance}

For any non-negative integer-valued random variables $Z$ and $W$,
the $\chi^2$-distance is defined by
\[
    \chi^2(\mathscr{L}(Z),\mathscr{L}(W))
    =\sum_{j\ge0}\left(\frac{\mathbb{P}(Z=j)}
    {\mathbb{P}(W=j)}-1\right)^2\mathbb{P}(W=j),
\]
provided that the series on the right-hand side has a meaning. It
possesses two important properties. First, its square root
upper-bounds the total variation distance
\[
    \dtv(\mathscr{L}(Z),\mathscr{L}(W))
    \le  \frac{1}{2}\sqrt{\chi^2(\mathscr{L}(Z),\mathscr{L}(W))}.
\]
Second, it also provides an effective upper bound for the
Kullback-Leibler divergence (or information divergence)
\[
    d_{\text{KL}}(\mathscr{L}(Z),\mathscr{L}(W))
    :=\sum_{j\ge0} \mathbb{P}(Z=j)\log
    \frac{\mathbb{P}(Z=j)}{\mathbb{P}(W=j)}
    \le\chi^2(\mathscr{L}(Z),\mathscr{L}(W)),
\]
a very useful measure in information theory and related applications.

\begin{thm}\label{chi-square}
The $\chi^2$-distance between the distribution of $X_n$ and the
binomial distribution $Y_\lambda$ satisfies
\[
    \chi^2(\mathscr{L}(X_n),\mathscr{L}(Y_\lambda))
    = O\left(\lambda^{-1}\right).
\]
\end{thm}
\begin{proof}
Let
\[
    f(z)
    :=P_n(z)-\left(\frac{1+z}{2}\right)^{\lambda},
\]
where $P_n$ is the probability generating function \eqref{Pny}
of $X_n$. Then, by (\ref{BGF}),
\[
    \left(1-\frac w2\right)^{\lambda}
    f\left(\frac{1+\frac w2}{1-\frac w2}\right)
    = \frac1n\sum_{2\le j\le s} 2^{\lambda_j}
    \left(\left(1+\frac w2\right)^{j-1}
    \left(1-\frac w2\right)^{\lambda-\lambda_j-j+1}-1\right).
\]
We need the elementary inequality
\[
    \left|(1+z)^a(1-z)^b-1\right|
    \le (1+|z|)^{a+b}-1 \le (a+b)|z|(1+|z|)^{a+b-1},
\]
for nonnegative integers $a, b$ with $a+b\ge1$. Applying this
inequality, we get, with $p=q=1/2$ and $N=\lambda$,
\[
    J(f,\lambda;r)
    \le\frac{r^2}{4}\left|
    \sum_{2\le j\le s} \frac{\lambda-\lambda_j}
    {2^{\lambda-\lambda_j}}
    \left(1+\frac r2\right)^{\lambda
    -\lambda_j-1}\right|^2.
\]\eject\noindent
Substituting this estimate into the Krawtchouk-Parseval identity
(\ref{kp}), we have
\begin{align*}
    &\left(\sum_{0\le m\le \lambda}\left|
    \frac{\mathbb{P}\bigl(X_n=m\bigr)}
    {\frac{1}{2^{\lambda}}\binom{\lambda}{m}}-1\right|^2
    \frac{1}{2^{\lambda}}\binom{\lambda}{m}\right)^{1/2}\\
    &\qquad\le\left((\lambda+1)
    \int_0^\infty\frac{J\left(f,\lambda;2\sqrt{u}\right)}
    {(1+u)^{\lambda+2}}\dd u\right)^{1/2}\\
    &\qquad\le\sum_{2\le j\le s}\frac{\lambda-\lambda_j}
    {2^{\lambda-\lambda_j}}\left( (\lambda+1)\int_0^\infty
    \frac{u^2}{(1+u)^{\lambda_j+3}}\cdot
    \frac{(1+\sqrt{u})^{2(\lambda-\lambda_j-1)}}
    {(1+u)^{\lambda-\lambda_j-1}}\dd u\right)^{1/2}\\
    &\qquad=\sum_{2\le j\le s}\frac{\lambda-\lambda_j}
    {2^{\lambda-\lambda_j}}\left( (\lambda+1)
    \int_0^\infty \frac{u^2}{(1+u)^{\lambda_j+3}}
    \left(1+\frac{2\sqrt{u}}{1+u}\right)
    ^{\lambda-\lambda_j-1}\dd u\right)^{1/2}.
\end{align*}
Since the function $\sqrt{u}/(1+u)$ reaches the maximum at the point
$u=1$, we see that
$$
    1+2\sqrt{u}/(1+u) \le 2.
$$
It follows that
\begin{align*}
    &\left(\chi^2(\mathscr{L}(X_n),\mathscr{L}(Y_\lambda))
    \right)^{1/2} \\
    &\quad\le\sum_{2\le j\le s}\frac{\lambda-\lambda_j}
    {2^{(\lambda-\lambda_j+1)/2}}
    \left( (\lambda+1)\int_0^\infty
    \frac{u^2}{(1+u)^{\lambda_j+3}}\dd u\right)^{1/2}\\
    &\quad\le\sum_{2\le j\le s}\frac{\lambda-\lambda_j}
    {2^{(\lambda-\lambda_j+1)/2}}\left( \frac{\lambda+1}
    {(\lambda_j+2)(\lambda_j+1)}\right)^{1/2}\\
    &\quad\le\sqrt{\lambda+1}\sum_{2\le j\le s}
    \frac{\lambda-\lambda_j}
    {2^{(\lambda-\lambda_j+1)/2}(1+\lambda_j)}\\
    &\quad\le \frac{1}{\sqrt{\lambda+1}}
    \sum_{\lambda-\lambda_2\le k\le \lambda-\lambda_s}
    \frac{k(\lambda+1)}{2^{k/2}(1+\lambda-k)}.
\end{align*}
It is clear that
\[
    \sum_{\lambda-\lambda_2\le k< \lambda}
    \frac{k(\lambda+1)}{2^{k/2}(1+\lambda-k)}
    =O(1).
\]
This proves the theorem.
\end{proof}

Finer results can be derived by developing similar techniques
as those used in \cite{zacharovas10a} for Poisson approximation.

\section{A general numeration system and applications}

The properties we studied above can be readily extended to a more
general framework of numeration system in which we encode each
integer by a different binary string and impose the sole condition
that
\begin{align} \label{Zn}
    Z_{2n} \stackrel{d}{=} Z_n + I \qquad(n\ge1),
\end{align}
where $Z_n$ denotes the number of $1$s in the resulting coding
string for a random integer, assuming that each of the first $n$
nonnegative integers is equally likely, and $I\sim
\text{Bernoulli}(1/2)$. For definiteness, let $Z_0=Z_1=0$.
This simple scheme covers in particular the
binary coding of $X_n$ above (as can be easily checked) and binary
reflected Gray code, which will be discussed in more detail later.
Let $\mu(n)$ denote the number of $1$s in the coding of $n$ in such
a numeration system. All our results below roughly say that this
numeration system does not differ much from the binary coding
although the codings inside each $2^k$ block can be rather flexible.

\begin{thm}[Local limit theorem]
Assume that $Z_n$ satisfies (\ref{Zn}). Then $Z_n$ is asymptotically
normally distributed:
\begin{align} \label{Zn-llt}
    \mathbb{P}\left(Z_n = \tr{\frac
    \lambda2+x\frac{\sqrt{\lambda}}2}\right)
    = \frac{\sqrt{2}\,e^{-x^2/2}}{\sqrt{\pi\lambda}}
    \left(1+O\left(\frac{1+|x|^3}{\sqrt{\lambda}}
    \right)\right),
\end{align}
uniformly for $x=o(\lambda^{1/6})$, with mean and variance
satisfying
\begin{align}\label{Zn-mu-var}
\begin{split}
    \mathbb{E}(Z_n) &= \frac{\log_2n}2 + G_1(\log_2n),\\
    \mathbb{V}(Z_n) &= \frac{\log_2n}4 + G_2(\log_2n).
\end{split}
\end{align}
Here $G_1, G_2$ are bounded periodic functions.
\end{thm}
For results related to moderate deviations, see \cite{chen13a}. We
can derive more precise Fourier expansions for the periodic
functions $G_1, G_2$ when more information is available.

\begin{thm} \label{thm-gen}
Assume that $Z_n$ satisfies (\ref{Zn}). Then
\begin{align*}
    &\sum_{0\le k\le \lambda}\left|\mathbb{P}
    \bigl(Z_n=k\bigr)-\sum_{0\le r<m}(-1)^r b_r(n)
    2^{-\lambda}\Delta^r \binom{\lambda}{k}\right| \\
    &\quad= \frac{h_m |b_m(n)|}{\lambda^{m/2}}+
    O\left(\lambda^{-(m+1)/2}\right),
\end{align*}
for $m=1,2,\dots$, where the sequence $b_r(n)=b_r(2n)$ is defined by
(see (\ref{BGF}))
\begin{align} \label{brn1}
    \mathbb{E}(y^{Z_n})
    \left(\frac{1+y}{2}\right)^{-\lambda}
    = \sum_{r\ge0} b_r(n) (y-1)^r.
\end{align}
\end{thm}
In particular, $b_0=1$, $b_1(n)=\mathbb{E}(Z_n)-\lambda/2$, and
\[
    b_2(n) = \frac{\mathbb{E}(Z_n^2)}2-\frac{\lambda+1}2\,
    \mathbb{E}(Z_n) +\frac{\lambda(\lambda+1)}8.
\]

\begin{cor} \label{Tdigits}
\begin{equation}
    \dtv\bigl(\mathscr{L}(Z_n),\mathscr{L}(Y_\lambda)\bigr)
    =\frac{\sqrt{2}|\bar{G}_1(\log_2 n)|}{\sqrt{\pi\log_2 n}}
    +O\left(\frac{1}{\log n}\right),
\end{equation}
where $\bar{G}_1(\log_2n)=\mathbb{E}(Z_n)-\lambda/2$ is periodic
$\bar{G}_1(x+1)=\bar{G}_1(x)$ and continuous on the set
$\mathbb{R}\setminus \mathbb{N}$.
\end{cor}
The periodic function $\bar{G}_1(x)$ can be defined as follows.
Write $2^x=\break \sum_{j\ge0}\xi_j 2^{-j}\in[1,2)$.
\[
    \bar{G}_1(x) = 2^{-x}\sum_{j\ge0} \frac{1-(-1)^{\tr{2^{j+x}}}}2
    2^{-j} \left(\mu\left(\tr{2^{j-1+x}}\right) - \frac j2\right).
\]
Note that $\bar{G}_1(x)= G_1(x)-\{x\}/2$.

\begin{thm} \label{spe-seq-mu} Assume, as above, that
$n=2^{\lambda}+2^{\lambda_2}+\cdots+2^{\lambda_s}$ with
$\lambda>\lambda_2>\cdots>\lambda_s\ge0$. Then
\[
    \dtv(\mathscr{L}(Z_n),\mathscr{L}(Y_\lambda))
    \asymp \frac1{2^{\lambda-\lambda_2}}
    \min\left\{1, \frac{\lambda-\lambda_2}
    {\sqrt{\lambda}}\right\},
\]
whenever $\lambda-\lambda_2\ge c$, where $c$ is sufficiently large.
\end{thm}

\subsection{Sketches of proofs}

Most of our analysis is based on the following explicit expression;
cf.\ (\ref{BGF}).
\begin{lem} If $Z_n$ satisfies the condition (\ref{Zn}), then
the probability generating function of $Z_n$ satisfies
\begin{align}\label{pgf-Zn}
    \mathbb{E}\left(y^{Z_n}\right)
    =\frac{1}{n}\sum_{1\le j\le s}
    y^{\mu\left(\tr{n/2^{\lambda_k}}-1\right)} (1+y)^{\lambda_k},
\end{align}
where $n=2^{\lambda}+2^{\lambda_2}+\cdots+2^{\lambda_s}$ with
$\lambda>\lambda_2>\cdots>\lambda_s\ge0$.
\end{lem}
\begin{proof}
Observe that the crucial condition (\ref{Zn}) implies the recurrence
\[
    \mathbb{E}\left(y^{Z_{2n}}\right)
    = \frac{1+y}2\,\mathbb{E}\left(y^{Z_n}\right),
\]
the same as (\ref{Jnz}) for $X_n$. Consequently, we also have,
following the same analysis there,
\begin{align*}
    \mathbb{E}\left(y^{Z_n}\right)
    &=\frac1n\sum_{0\le j\le \lambda} \frac{1-(-1)^{\tr{n/2^j}}}2
    y^{\mu\left(\tr{n/2^{j}}-1\right)} (1+y)^j,
\end{align*}
for $n\ge1$; compare (\ref{Qnz}). Since $2$ divides $\tr{n/2^j}$ if
and only if $j\not\in\{\lambda_1,\lambda_2,\ldots,\lambda_s\}$, we
obtain (\ref{pgf-Zn}).
\end{proof}

From the expression (\ref{pgf-Zn}), we easily obtain
\begin{align*}
    \mathbb{E}(Z_n) &= \frac1n\sum_{1\le j\le s}
    2^{\lambda_j} \left(\bar{\mu}_j+\frac{\lambda_j}2\right),\\
    \mathbb{E}(Z_n^2) &= \frac1n\sum_{1\le j\le s}
    2^{\lambda_j} \left(\bar{\mu}_j^2
    +\lambda_j\bar{\mu}_j
    +\frac{\lambda_j^2+\lambda_j}4\right),
\end{align*}
where $\bar{\mu}_j := \mu(\tr{n/2^{\lambda_j}}-1)$.
The identity for the mean in (\ref{Zn-mu-var}) then follows with
\[
    G_1(\log_2n) = \frac{2^\lambda}n
    \sum_{1\le j\le s} 2^{-(\lambda-\lambda_j)}
    \left(\bar{\mu}_j-\frac{\lambda-\lambda_j}2\right)
    -\frac{\{\log_2n\}}2,
\]
which is periodic and bounded since $\bar{\mu}_j\le j+1$. Also the
definition of $G_1$ here for $\log_2n$ can be readily extended to
all reals.

We now prove the identity in (\ref{Zn-mu-var}) for the variance as
the proof is very simple. Let
\[
    \bar{Z}_n := Z_n - \frac\lambda2 =
    \frac{1}n\sum_{1\le j\le s} 2^{\lambda_j}
    \left(\bar{\mu}_j-\frac{\lambda-\lambda_j}2\right).
\]
Then
\begin{align*}
    &\mathbb{V}(Z_n) -\frac{\lambda}4\\
    &\quad= \mathbb{E}(Z_n^2) -
    \left(\mathbb{E}(\bar{Z}_n)+\frac \lambda2\right)^2
    -\frac{\lambda}4\\
    &\quad= \frac1n\sum_{1\le j\le s}2^{\lambda_j}
    \left(\bar{\mu}_j(1+\lambda_j-\lambda) +
    \frac{\lambda_j(\lambda_j+1)- \lambda(\lambda+1)
    +2\lambda(\lambda-\lambda_j)}4\right)\\
    &\qquad-\left(\mathbb{E}(\bar{Z}_n)\right)^2,
\end{align*}
and we obtain the identity for the variance in (\ref{Zn-mu-var})
with
\begin{align*}
    G_2(\log_2n) &= \frac{2^\lambda}n
    \sum_{1\le j\le s} 2^{-(\lambda-\lambda_j)}
    \left(\bar{\mu}_j^2-\bar{\mu}_j
    \left(\lambda-\lambda_j\right)
    +\frac{(\lambda-\lambda_j)
    (\lambda-\lambda_j-1)}4\right)\\
    &\qquad -G_1(\log_2n)^2-G_1(\log_2n)\{\log_2n\}
    -\frac{\{\log_2n\}^2+\{\log_2n\}}{4},
\end{align*}
which is also bounded and periodic, and extendible to all $x\in
\mathbb{R}$.

The local limit theorem (\ref{Zn-llt}) is proved in a way similar to
the proof of Proposition~\ref{loc-exp}.

In terms of probabilities, the identity (\ref{pgf-Zn}) means that
the random variable $Z_n$ can be expressed as a mixture of shifted
binomial random variables. Its distribution can be described in the
following way. Let $\zeta_n$ be a random variable defined by
\[
    \mathbb{P}(\zeta_n=j)
    =\frac{2^{\lambda_j}}{n}\qquad(1\le j\le s).
\]
Then
\[
    \mathbb{P}(Z_n\in A|\zeta_n=j)
    =\mathbb{P}(Y_{\lambda_j}+r_j\in A),
\]
for any $A\subset \mathbb{R}$, where $r_j
:=\mu(\tr{n/2^{\lambda_j}}-1) \le \lambda-\lambda_j+1$ and $r_0:=0$.
By the same arguments used above, we see that the identity
\begin{align*}
    \mathbb{E}\left(h\bigl(\mu(X_{n})\bigr)\right)
    =\sum_{2\le j\le s}\frac{2^{\lambda_j}}{n}\mathbb{E}
    \bigl(({\lambda_{j}}-\lambda+r_j)g(Y_{\lambda_{j}}+r_j)
    +r_jg(Y_{\lambda_{j}}+r_j-1)\bigr)
\end{align*}
holds for any function $h:\mathbb{R}\to \mathbb{R}$, where $g$ is
the solution to Stein's equation~(\ref{stein-eq}).

We skip all details of the proofs as they are almost identical to
those for $X_n$.

\subsection{Gray code}
\label{sec-gray}

The Gray code is characterized by the property that the codings of
any two successive integers differ by exactly one bit. It is named
after Frank Gray's 1947 patent, although the same construction had
been introduced in telegraphy in the late nineteenth century by the
French engineer \'Emile Baudot; see Wikipedia's page on Gray code
for more information. The coding notion with two neighboring objects
differing at one location has turned out to be extremely useful in
many scientific disciplines beyond the original communication
motivations such as experimental designs, job scheduling in computer
systems, and combinatorial generation; see the survey paper
\cite{savage97a} and the references therein.

The binary reflected Gray code is constructed by reflecting (or
mirroring) the first $2^k$ codings of the first $2^k$ nonnegative
integers and then adding $1$ at the beginning for each coding,
resulting in the Gray code for the first $2^{k+1}$ nonnegative
integers; see Figure~\ref{fig12} for an illustration.%

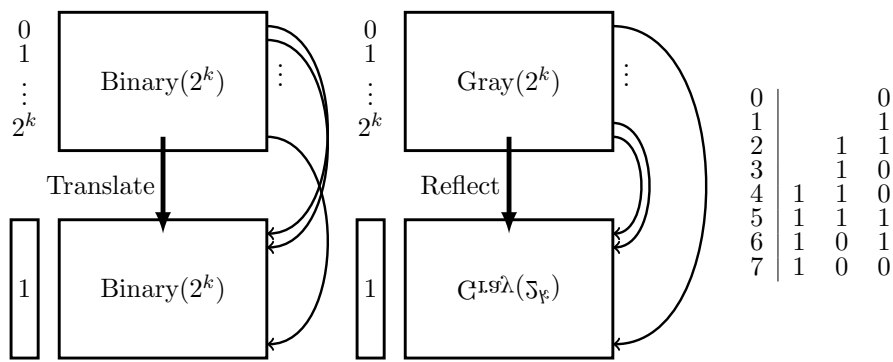
\begin{figure}
\scalebox{0.92}{
\begin{tikzpicture}[
  s1/.style={
  rectangle,
  minimum width=30mm,
  minimum height=20mm,
  very thick,
  draw=black,
},s2/.style={ rectangle,
  minimum width=0.5mm,
  minimum height=20mm,
  very thick,
  draw=black,
}]
\newcommand{\FONTSIZE}{\fontsize{4mm}{\baselineskip}\selectfont}
\node [s1,text centered] at (-3.5,3)
{\FONTSIZE{Binary$(2^{k})$}};%
\node [s1, text centered] at (1.5,3)
{\FONTSIZE{Gray$(2^{k})$}};%
\node [s1,text centered] at (-3.5,0)
{\FONTSIZE{Binary$(2^{k})$}};%
\node [s1, text centered,rotate=180] at (1.5,0)
{\reflectbox{\FONTSIZE{Gray$(2^{k})$}}};%
\node [s2,text centered] at (-5.5,0)
{\FONTSIZE{1}};%
\node [s2,text centered] at (-0.5,0)
{\FONTSIZE{1}};%
\node [text centered] at (-4.4,1.5)
{\FONTSIZE{Translate}};%
\node [text centered] at (.8,1.5)
{\FONTSIZE{Reflect}};%

\node at (-5.5,3){
\begin{tabular}{cccc}
\FONTSIZE{0} \\
\FONTSIZE{1} \\
\FONTSIZE{$\vdots$} \\
\FONTSIZE{$2^{k}$}
\end{tabular}
};%

\node at (-0.5,3){
\begin{tabular}{cccc}
\FONTSIZE{0} \\
\FONTSIZE{1} \\
\FONTSIZE{$\vdots$} \\
\FONTSIZE{$2^{k}$}
\end{tabular}
};%

\node at (6,1.5){
\begin{tabular}{c|cccccc}
\FONTSIZE{0}& & &\FONTSIZE{0}\\
\FONTSIZE{1}& & &\FONTSIZE{1}\\
\FONTSIZE{2}& &\FONTSIZE{1}&\FONTSIZE{1}\\
\FONTSIZE{3}& &\FONTSIZE{1}&\FONTSIZE{0}\\
\FONTSIZE{4}&\FONTSIZE{1}&\FONTSIZE{1}&\FONTSIZE{0}\\
\FONTSIZE{5}&\FONTSIZE{1}&\FONTSIZE{1}&\FONTSIZE{1}\\
\FONTSIZE{6}&\FONTSIZE{1}&\FONTSIZE{0}&\FONTSIZE{1}\\
\FONTSIZE{7}&\FONTSIZE{1}&\FONTSIZE{0}&\FONTSIZE{0}\\
\end{tabular}
};%

\draw [->,line width=1pt] (-2,3.8) to [out=360,in=0] (-2,0.8);%
\draw [->,line width=1pt] (-2,3.6) to [out=360,in=0] (-2,0.6);%
\draw [->,line width=1pt] (-2,2.2) to [out=360,in=0] (-2,-0.8);%
\node at (-1.8,3.2){\FONTSIZE{$\vdots$}};
\draw [->,line width=1pt] (3,3.8) to [out=360,in=0] (3,-0.8);%
\draw [->,line width=1pt] (3,2.4) to [out=360,in=0] (3,0.6);%
\draw [->,line width=1pt] (3,2.2) to [out=360,in=0] (3,0.8);%
\node at (3.2,3.2){\FONTSIZE{$\vdots$}};
\draw[-latex,line width=2pt] (-3.5,2.2) -- (-3.5,0.8);%
\draw[-latex,line width=2pt] (1.5,2.2) -- (1.5,0.8);%

\end{tikzpicture}}
\caption{Constructions of binary code (left) and Gray code (middle),
and the Gray code of the first few integers (right).}\label{fig12}
\end{figure}

By construction, the Gray code, say $\mathcal{G}(2^k+j)$ of $2^k+j$
with $0\le j<2^k$ is equal to $10^\ell\mathcal{G}(2^k-1-j)$ (string
concatenation), where $\ell :=k-1-\tr{\log_2(2^k-1-j)}$ and $0^\ell$
means $0$ written $\ell$ times. For example,
\[
    \mathcal{G}(19) = 1\mathcal{G}(12) = 110\mathcal{G}(3)
    = 11010.
\]
Thus the number of $1$s, denoted by $\gamma(n)$, of $n$ under such a
coding system satisfies the recurrence
\begin{equation}\label{gamma}
    \gamma(2^k+j) = 1+\gamma(2^k-1-j),
\end{equation}
for $0\le j<2^k$ and $k\ge1$. Another interesting type of recurrence
is (by induction)
\[
    \gamma(n)
    = \gamma\left(\tr{n/2}\right)
    + \frac{1-(-1)^{\rd{n/2}}}2,
\]
for $k\ge1$, in contrast to
\[
    \nu(n) = \nu\left(\tr{n/2}\right)
    +\frac{1-(-1)^n}2,
\]
for binary coding.

\begin{figure}[t!]
\begin{tikzpicture}[yscale=1,xscale=0.1]
\newcommand{\FONTSIZE}{\fontsize{8pt}{\baselineskip}\selectfont}
\draw[line width=1pt] (0,0) -- (0,6.2)--(65,6.2)--(65,0)--cycle;%
\foreach \x in {2,4,8,16,32,64  }%
\draw[shift={(\x, 0)}] (0, 0) -- (0,-0.05)node[below]{\FONTSIZE{\x}};%
\foreach \y in {1,2,3,4,5,6}%
\draw[shift={(0.000, \y)}] (0,0) -- (-0.5, 0) node[left]{\FONTSIZE{\y}};%
\definecolor{gen}{rgb}{0,0,1}%
\draw[ycomb,color=blue,line width=2pt] plot coordinates{(1, 1)
(2,2)(3, 1)(4, 2)(5, 3)(6, 2)(7, 1)(8, 2)(9, 3)(10, 4)(11, 3)
(12,2)(13, 3)(14, 2)(15, 1)(16, 2)(17, 3)(18, 4)(19, 3)(20, 4)
(21,5)(22, 4)(23, 3)(24, 2)(25, 3)(26, 4)(27, 3)(28, 2)(29, 3)
(30,2)(31, 1)(32, 2)(33, 3)(34, 4)(35, 3)(36, 4)(37, 5)(38, 4)
(39,3)(40, 4)(41, 5)(42, 6)(43, 5)(44, 4)(45, 5)(46, 4)(47, 3)
(48,2)(49, 3)(50, 4)(51, 3)(52, 4)(53, 5)(54, 4)(55, 3)(56, 2)
(57,3)(58, 4)(59, 3)(60, 2)(61, 3)(62, 2)(63, 1)(64, 2)};
\draw node[text width=10pt] at (32.5,-0.8) {\FONTSIZE $n$};%
\draw (4,5.7) node[right]{
};%
\end{tikzpicture}
\caption{$\gamma(n)$.}\label{fig13}
\end{figure}

Let now
$$
    R_n(z)
    =\sum_{0\le j <n}z^{\gamma(j)}.
$$
Then, obviously,
$$
    R_{2^k}(z)
    =(1+z)^k,
$$
and, by (\ref{gamma}),
\begin{align*}
    R_{2^k+j}(z)
    &=R_{2^k}(z)+z(R_{2^k}(z)-R_{2^k-j}(z))\\
    &=R_{2^k}(z)(1+z)-zR_{2^k-j}(z)\\
    &=(1+z)^{k+1}-zR_{2^k-j}(z).
\end{align*}
From this recurrence relation, we deduce by induction that
\[
    R_{2n}(z)
    =(1+z)R_n(z),
\]
for all $n\ge 1$. Thus the the sum-of-digits function $Z_n$ of
random integers under the Gray coding satisfies (\ref{Zn}), and thus
Theorems~\ref{Zn-llt}, \ref{thm-gen} and \ref{spe-seq-mu} and
Corollary~\ref{Tdigits} all hold. In addition to the mean and the
variance, all results are new. The mean of $Z_n$ was first studied
by Flajolet and Ramshaw \cite{flajolet80a} where more precise
characterizations of $G_1$ (including a Fourier series expansion)
are given. A closed-form expression for $\mathbb{E}(y^{Z_n})$ was
derived by Kobayashi et al.\ \cite{kobayashi02a} by singular
measures, together with exact expressions for all moments
(non-centered).\looseness=1

\begin{figure}
\begin{tikzpicture}[scale=0.8]
\newcommand{\FONTSIZE}{\fontsize{8pt}{\baselineskip}\selectfont}
\draw[-latex] (0, -0.5) -- (0, 5.5) node[right]
{
};
\foreach \y/\ytext in { 0.000, 0.500, ..., 5.000  }
\draw[shift={(0.000, \y)}] (0.025, 0) -- (-0.025, 0);
\foreach \y/\ytext in { 0.000, 0.100, ..., 5.000  }
\draw[shift={(0.000, \y)}] (0.0125, 0) -- (-0.0125, 0);
\draw (0.050 , 0) -- (-0.050 , 0) node[left] {\FONTSIZE{$-0.11$}}
(0.050 , 2) -- (-0.050 , 2) node[left] {\FONTSIZE{$0.07$}}
(0.050 , 3) -- (-0.050 , 3) node[left] {\FONTSIZE{$0.16$}}
(0.050 , 4) -- (-0.050 , 4) node[left] {\FONTSIZE{$0.25$}}
(0.050 , 5) -- (-0.050 , 5) node[left] {\FONTSIZE{$0.34$}};%
\draw[color=black,fill=blue,line width=1pt,
top color=blue,bottom color=blue!20] plot coordinates{
(0, 1.250)(0.028, 1.318)(0.056, 1.260)(0.084, 1.160)(0.112, 1.102)
(0.140, 1.004)(0.167, 0.949)(0.195, 0.852)(0.222, 0.796)(0.249, 0.700)
(0.276, 0.647)(0.303, 0.717)(0.330, 0.663)(0.357, 0.569)(0.384, 0.516)
(0.411, 0.423)(0.437, 0.371)(0.464, 0.279)(0.490, 0.227)(0.516, 0.294)
(0.543, 0.242)(0.569, 0.306)(0.595, 0.562)(0.621, 0.621)(0.646, 0.566)
(0.672, 0.476)(0.698, 0.422)(0.723, 0.479)(0.749, 0.424)(0.774, 0.335)
(0.799, 0.281)(0.825, 0.193)(0.850, 0.148)(0.875, 0.062)(0.900, 0.009)
(0.924, 0.064)(0.949, 0.010)(0.974, 0.063)(0.998, 0.294)(1.023, 0.343)
(1.047, 0.287)(1.072, 0.335)(1.096, 0.556)(1.120, 1.019)(1.144, 1.232)
(1.168, 1.270)(1.192, 1.477)(1.216, 1.512)(1.240, 1.449)(1.263, 1.363)
(1.287, 1.302)(1.310, 1.335)(1.334, 1.274)(1.357, 1.306)(1.381, 1.500)
(1.404, 1.529)(1.427, 1.466)(1.450, 1.382)(1.473, 1.320)(1.496, 1.348)
(1.519, 1.285)(1.542, 1.203)(1.564, 1.141)(1.587, 1.060)(1.610, 1.026)
(1.632, 0.946)(1.655, 0.886)(1.677, 0.913)(1.699, 0.853)(1.721, 0.879)
(1.744, 1.057)(1.766, 1.081)(1.788, 1.019)(1.810, 1.042)(1.832, 1.213)
(1.853, 1.595)(1.875, 1.760)(1.897, 1.776)(1.919, 1.936)(1.940, 1.951)
(1.962, 1.885)(1.983, 1.899)(2.004, 2.053)(2.026, 2.411)(2.047, 2.560)
(2.068, 2.909)(2.089, 3.518)(2.110, 3.857)(2.131, 3.993)(2.152, 3.991)
(2.173, 4.124)(2.194, 4.450)(2.215, 4.577)(2.235, 4.571)(2.256, 4.695)
(2.277, 4.687)(2.297, 4.611)(2.318, 4.531)(2.338, 4.456)(2.358, 4.448)
(2.379, 4.374)(2.399, 4.366)(2.419, 4.484)(2.439, 4.474)(2.459, 4.400)
(2.479, 4.390)(2.499, 4.503)(2.519, 4.798)(2.539, 4.907)(2.559, 4.893)
(2.578, 5.000)(2.598, 4.985)(2.618, 4.909)(2.637, 4.832)(2.657, 4.757)
(2.676, 4.742)(2.696, 4.668)(2.715, 4.653)(2.734, 4.754)(2.754, 4.738)
(2.773, 4.664)(2.792, 4.590)(2.811, 4.516)(2.830, 4.500)(2.849, 4.427)
(2.868, 4.354)(2.887, 4.281)(2.906, 4.209)(2.925, 4.195)(2.944, 4.292)
(2.962, 4.277)(2.981, 4.205)(3.000, 4.189)(3.018, 4.117)(3.037, 4.047)
(3.055, 3.977)(3.074, 3.961)(3.092, 3.891)(3.110, 3.822)(3.129, 3.810)
(3.147, 3.741)(3.165, 3.672)(3.183, 3.605)(3.201, 3.536)(3.219, 3.522)
(3.237, 3.454)(3.255, 3.387)(3.273, 3.376)(3.291, 3.310)(3.309, 3.298)
(3.327, 3.394)(3.344, 3.382)(3.362, 3.316)(3.380, 3.250)(3.397, 3.185)
(3.415, 3.172)(3.433, 3.107)(3.450, 3.043)(3.467, 2.978)(3.485, 2.914)
(3.502, 2.903)(3.520, 2.839)(3.537, 2.776)(3.554, 2.765)(3.571, 2.702)
(3.588, 2.690)(3.605, 2.781)(3.623, 2.769)(3.640, 2.706)(3.657, 2.693)
(3.674, 2.781)(3.690, 3.020)(3.707, 3.105)(3.724, 3.090)(3.741, 3.173)
(3.758, 3.157)(3.774, 3.093)(3.791, 3.033)(3.808, 2.970)(3.824, 2.954)
(3.841, 2.891)(3.857, 2.875)(3.874, 2.955)(3.890, 2.939)(3.907, 2.876)
(3.923, 2.818)(3.940, 2.756)(3.956, 2.739)(3.972, 2.678)(3.988, 2.620)
(4.004, 2.559)(4.021, 2.501)(4.037, 2.496)(4.053, 2.591)(4.069, 2.586)
(4.085, 2.528)(4.101, 2.522)(4.117, 2.465)(4.133, 2.406)(4.149, 2.349)
(4.164, 2.343)(4.180, 2.287)(4.196, 2.228)(4.212, 2.215)(4.227, 2.156)
(4.243, 2.101)(4.259, 2.043)(4.274, 1.988)(4.290, 1.983)(4.305, 1.928)
(4.321, 1.871)(4.336, 1.858)(4.352, 1.802)(4.367, 1.789)(4.383, 1.866)
(4.398, 1.852)(4.413, 1.796)(4.428, 1.743)(4.444, 1.687)(4.459, 1.673)
(4.474, 1.618)(4.489, 1.565)(4.504, 1.510)(4.519, 1.458)(4.534, 1.455)
(4.549, 1.544)(4.564, 1.540)(4.579, 1.488)(4.594, 1.484)(4.609, 1.432)
(4.624, 1.378)(4.639, 1.327)(4.654, 1.323)(4.668, 1.271)(4.683, 1.218)
(4.698, 1.207)(4.713, 1.154)(4.727, 1.104)(4.742, 1.051)(4.756, 1.001)
(4.771, 0.997)(4.786, 1.083)(4.800, 1.078)(4.814, 1.028)(4.829, 1.024)
(4.843, 0.974)(4.858, 0.923)(4.872, 0.873)(4.886, 0.868)(4.901, 0.951)
(4.915, 0.945)(4.929, 0.896)(4.943, 0.890)(4.958, 0.971)(4.972, 0.964)
(4.986, 1.044)(5.000, 1.250)};%
\draw[-latex] (-0.75, 1.25) -- (5.5 , 1.25) node[right] {$x$};
\foreach \x/\xtext in { 0.000, 0.500, ..., 5.000  }
\draw[shift={(\x,1.250)}] (0,0.025) -- (0,-0.025);
\foreach \x/\xtext in { 0.000, 0.100, ..., 5.000  }
\draw[shift={(\x,1.250)}] (0,0.0125) -- (0,-0.0125);
\draw (1 , 1.29956) -- (1 , 1.2) node[below] {\FONTSIZE{$0.2$}}
(2 , 1.29956) -- (2 , 1.200) node[below] {\FONTSIZE{$0.4$}}
(3 , 1.29956) -- (3 , 1.200) node[below] {\FONTSIZE{$0.6$}}
(4 , 1.29956) -- (4 , 1.200) node[below] {\FONTSIZE{$0.8$}}
(5 , 1.29956) -- (5 , 1.200) node[below=6pt, right=-1pt]{\FONTSIZE{$1$}};%
\end{tikzpicture}
\caption{$G_1(x)$.}\label{fig14}
\end{figure}

\begin{figure}
\begin{tikzpicture}[scale=0.8]
\newcommand{\FONTSIZE}{\fontsize{8pt}{\baselineskip}\selectfont}
\draw[-latex] (0, -0.75) -- (0, 5.5) node[right] {
};
\foreach \y/\ytext in { 0.000, 0.500, ..., 5.000  }
\draw[shift={(0.000, \y)}] (0.025, 0) -- (-0.025, 0);
\foreach \y/\ytext in { 0.000, 0.100, ..., 5.000  }
\draw[shift={(0.000, \y)}] (0.0125, 0) -- (-0.0125, 0);
\draw (0.050 , 0) -- (-0.050 , 0) node[left] {\FONTSIZE{$-0.05$}}
(0.050 , 2) -- (-0.050 , 2) node[left] {\FONTSIZE{$0.06$}}
(0.050 , 3) -- (-0.050 , 3) node[left] {\FONTSIZE{$0.12$}}
(0.050 , 4) -- (-0.050 , 4) node[left] {\FONTSIZE{$0.17$}}
(0.050 , 5) -- (-0.050 , 5) node[left] {\FONTSIZE{$0.22$}};%
\definecolor{gen}{rgb}{0.647059,0.164706,0.164706}%
\draw[color=black,fill=gen,line width=1pt,fill opacity=0.5,
top color=gen!80,bottom color=gen!20] plot coordinates{(0, 0.859)
(0.028, 0.663)(0.056, 0.539)(0.084, 0.489)(0.112, 0.367)
(0.140, 0.317)(0.167, 0.337)(0.195, 0.287)(0.222, 0.167)
(0.249, 0.118)(0.276, 0.138)(0.303, 0.227)(0.330, 0.247)
(0.357, 0.197)(0.384, 0.216)(0.411, 0.167)(0.437, 0.049)
(0.464, 0.000)(0.490, 0.019)(0.516, 0.104)(0.543, 0.122)
(0.569, 0.207)(0.595, 0.357)(0.621, 0.440)(0.646, 0.455)
(0.672, 0.405)(0.698, 0.420)(0.723, 0.501)(0.749, 0.516)
(0.774, 0.465)(0.799, 0.479)(0.825, 0.429)(0.850, 0.314)
(0.875, 0.264)(0.900, 0.279)(0.924, 0.356)(0.949, 0.370)
(0.974, 0.446)(0.998, 0.585)(1.023, 0.659)(1.047, 0.671)
(1.072, 0.745)(1.096, 0.880)(1.120, 1.076)(1.144, 1.209)
(1.168, 1.279)(1.192, 1.411)(1.216, 1.480)(1.240, 1.487)
(1.263, 1.434)(1.287, 1.441)(1.310, 1.508)(1.334, 1.515)
(1.357, 1.581)(1.381, 1.707)(1.404, 1.772)(1.427, 1.777)
(1.450, 1.723)(1.473, 1.728)(1.496, 1.792)(1.519, 1.796)
(1.542, 1.742)(1.564, 1.747)(1.587, 1.693)(1.610, 1.581)
(1.632, 1.528)(1.655, 1.533)(1.677, 1.594)(1.699, 1.598)
(1.721, 1.659)(1.744, 1.776)(1.766, 1.836)(1.788, 1.839)
(1.810, 1.898)(1.832, 2.012)(1.853, 2.182)(1.875, 2.294)
(1.897, 2.351)(1.919, 2.462)(1.940, 2.517)(1.962, 2.517)
(1.983, 2.571)(2.004, 2.680)(2.026, 2.843)(2.047, 2.950)
(2.068, 3.111)(2.089, 3.325)(2.110, 3.483)(2.131, 3.587)
(2.152, 3.636)(2.173, 3.738)(2.194, 3.893)(2.215, 3.994)
(2.235, 4.041)(2.256, 4.141)(2.277, 4.187)(2.297, 4.180)
(2.318, 4.121)(2.338, 4.114)(2.358, 4.159)(2.379, 4.152)
(2.399, 4.197)(2.419, 4.293)(2.439, 4.337)(2.459, 4.329)
(2.479, 4.372)(2.499, 4.466)(2.519, 4.611)(2.539, 4.703)
(2.559, 4.745)(2.578, 4.836)(2.598, 4.877)(2.618, 4.867)
(2.637, 4.807)(2.657, 4.797)(2.676, 4.837)(2.696, 4.826)
(2.715, 4.866)(2.734, 4.954)(2.754, 4.993)(2.773, 4.982)
(2.792, 4.922)(2.811, 4.911)(2.830, 4.949)(2.849, 4.938)
(2.868, 4.879)(2.887, 4.868)(2.906, 4.808)(2.925, 4.701)
(2.944, 4.546)(2.962, 4.440)(2.981, 4.382)(3.000, 4.276)
(3.018, 4.219)(3.037, 4.209)(3.055, 4.152)(3.074, 4.048)
(3.092, 3.991)(3.110, 3.982)(3.129, 4.019)(3.147, 4.010)
(3.165, 3.953)(3.183, 3.944)(3.201, 3.888)(3.219, 3.786)
(3.237, 3.730)(3.255, 3.721)(3.273, 3.757)(3.291, 3.748)
(3.309, 3.784)(3.327, 3.866)(3.344, 3.901)(3.362, 3.891)
(3.380, 3.836)(3.397, 3.826)(3.415, 3.861)(3.433, 3.851)
(3.450, 3.796)(3.467, 3.786)(3.485, 3.731)(3.502, 3.632)
(3.520, 3.578)(3.537, 3.568)(3.554, 3.602)(3.571, 3.592)
(3.588, 3.626)(3.605, 3.704)(3.623, 3.737)(3.640, 3.727)
(3.657, 3.760)(3.674, 3.836)(3.690, 3.956)(3.707, 4.031)
(3.724, 4.063)(3.741, 4.137)(3.758, 4.168)(3.774, 4.157)
(3.791, 4.102)(3.808, 4.090)(3.824, 4.121)(3.841, 4.109)
(3.857, 4.139)(3.874, 4.212)(3.890, 4.242)(3.907, 4.229)
(3.923, 4.175)(3.940, 4.163)(3.956, 4.192)(3.972, 4.180)
(3.988, 4.126)(4.004, 4.113)(4.021, 4.060)(4.037, 3.965)
(4.053, 3.829)(4.069, 3.735)(4.085, 3.682)(4.101, 3.589)
(4.117, 3.536)(4.133, 3.525)(4.149, 3.473)(4.164, 3.380)
(4.180, 3.329)(4.196, 3.318)(4.212, 3.347)(4.227, 3.336)
(4.243, 3.284)(4.259, 3.273)(4.274, 3.222)(4.290, 3.131)
(4.305, 3.080)(4.321, 3.070)(4.336, 3.098)(4.352, 3.087)
(4.367, 3.116)(4.383, 3.184)(4.398, 3.212)(4.413, 3.200)
(4.428, 3.150)(4.444, 3.139)(4.459, 3.166)(4.474, 3.155)
(4.489, 3.105)(4.504, 3.093)(4.519, 3.043)(4.534, 2.955)
(4.549, 2.828)(4.564, 2.741)(4.579, 2.692)(4.594, 2.605)
(4.609, 2.556)(4.624, 2.545)(4.639, 2.497)(4.654, 2.411)
(4.668, 2.362)(4.683, 2.352)(4.698, 2.380)(4.713, 2.369)
(4.727, 2.321)(4.742, 2.311)(4.756, 2.264)(4.771, 2.179)
(4.786, 2.057)(4.800, 1.973)(4.814, 1.926)(4.829, 1.842)
(4.843, 1.796)(4.858, 1.786)(4.872, 1.740)(4.886, 1.657)
(4.901, 1.538)(4.915, 1.455)(4.929, 1.410)(4.943, 1.328)
(4.958, 1.210)(4.972, 1.129)(4.986, 1.012)(5.000, 0.859)};%
\draw[-latex] (-0.75, 0.859) -- (5.5, 0.859) node[right] {$x$};
\foreach \x/\xtext in { 0.000, 0.500, ..., 5.000  }
\draw[shift={(\x,0.859)}] (0,0.025) -- (0,-0.025);
\foreach \x/\xtext in { 0.000, 0.100, ..., 5.000  }
\draw[shift={(\x,0.859)}] (0,0.0125) -- (0,-0.0125);
\draw (1 , 0.908722) -- (1 , 0.809) node[below] {\FONTSIZE{$0.2$}}
(2 , 0.908722) -- (2 , 0.809) node[below] {\FONTSIZE{$0.4$}}
(3 , 0.908722) -- (3 , 0.809) node[below] {\FONTSIZE{$0.6$}}
(4 , 0.908722) -- (4 , 0.809) node[below] {\FONTSIZE{$0.8$}}
(5 , 0.908722) -- (5 , 0.809) node[below] {\FONTSIZE{$1$}};%
\end{tikzpicture}
\caption{$G_2(x)$.}\label{fig15}
\end{figure}

For other properties related to $\gamma(n)$ and $Z_n$, see
\cite{prodinger83b,kirschenhofer84a,
prodinger02a,kobayashi02a,kobayashi02b,
doran07a,hofer08a,klavzar07a}.

So far, we considered only the goodness of approximations to
$\mathscr{L}(X_n)$ and $\mathscr{L}(Z_n)$ by the binomial
distribution $Y_\lambda$. It is also natural to consider
approximations of $\mathscr{L}(Z_n)$ by $\mathscr{L}(X_n)$, and the
result is as follows.
\begin {equation}
    \dtv\left(\mathscr{L}(X_n),\mathscr{L}(Z_n)\right)
    =\frac{\sqrt{2}|F(\log_2 n)-G_1(\log_2n)|}
    {\sqrt{\pi \log_2n}}+O\left(\frac{1}{\log_2n}\right),
\end{equation}
where the difference $F(x)-G_1(x)$ is a \emph{continuous function
for all} $x$.

\subsection{Beyond binary and Gray codings}

We give here another simple binary coding system for integers
satisfying the condition (\ref{Zn}). We start with the observation
that binary coding can be constructed not only in the usual
translation way, but also using reflection and complement (first
reflect the whole block of $2^k$ numbers as in Gray code, and then
change every $1$ to $0$ and every $0$ to $1$);
see Figure~\ref{fig16}.

\begin{figure}
\begin{tikzpicture}[auto,text badly centered,
   s1/.style={
   rectangle,
   minimum width=30mm,
   minimum height=20mm,
   very thick,
   draw=black,
},s2/.style={ rectangle,
   minimum width=0.5mm,
   minimum height=20mm,
   very thick,
   draw=black,
}]
\newcommand{\FONTSIZE}{\fontsize{4mm}{\baselineskip}\selectfont}
\node [s1,text centered] at (1.5,3)
{\FONTSIZE{Binary$(2^{k})$}};%
\node [s1, text centered] at (-3.5,3)
{\FONTSIZE{Binary$(2^{k})$}};%
\node [s1, text centered,text width=2cm] (n1) at
(1.5,0) {\FONTSIZE{Binary$(2^{k})$}};%
\node [s1,text centered] at (-3.5,0) {\FONTSIZE{Binary$(2^{k})$}};%
\node [text centered] at (0.35,1.5) {\small{Reflect}};%
\draw[-latex,line width=1pt] (1,1.5) -- (2.1,1.5);%
\node [text centered,text width=3cm] at
(3.3,1.5) { \small Complement ($0\rightarrow1$;$1\rightarrow0$)};%
\node [text centered] at (-2.6,1.5) {\small{Translate}};%
\node at (-1.1,1.5){\FONTSIZE{$=$}};%
\node at(-1.5,-1.5){
};
\node [s2,text centered] (n1) at (-5.5,0) {\FONTSIZE{1}};%
\node [s2,text centered] (n1) at (-0.5,0) {\FONTSIZE{1}};%
\node at (-5.5,3){
\begin{tabular}{cccc}
\FONTSIZE{0} \\
\FONTSIZE{1} \\
\FONTSIZE{$\vdots$} \\
\FONTSIZE{$2^{k}$}
\end{tabular}
};%

\node at (-0.5,3){
\begin{tabular}{cccc}
\FONTSIZE{0} \\
\FONTSIZE{1} \\
\FONTSIZE{$\vdots$} \\
\FONTSIZE{$2^{k}$}
\end{tabular}
};%
\draw[-latex,line width=2pt] (-3.5,2.2) -- (-3.5,0.8);%
\draw[-latex,line width=2pt] (1.5,2.2) -- (1.5,0.8);%
\end{tikzpicture}
\caption{Two different ways of constructing the same binary code.}\label{fig16}
\end{figure}
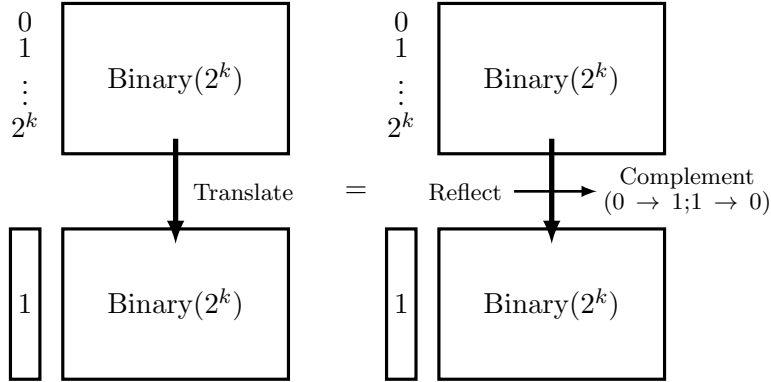

We now consider a coding system using translation and
complement. Let $\mu(n)$ denote the number of $1$s in the coding of
$n$. Then by construction
\[
    \mu(2^k+j) = k+1-\mu(j),
\]
for $0\le j<2^k$ and $k\ge1$.
From this recurrence, we see that
\begin{align*}
    &\sum_{0\le \ell<2^k+j}y^{\mu(\ell)}
 = (1+y)^k + y^{k+1}\sum_{0\le \ell<j}y^{-\mu(\ell)},
\end{align*}
and it is straightforward to see that (\ref{Zn}) holds in such a
coding system. Thus $Z_n$ satisfies all properties stated in the
beginning of this section.

\begin{figure}
\begin{tikzpicture}[auto,text badly centered,
   s1/.style={
   rectangle,
   minimum width=30mm,
   minimum height=20mm,
   very thick,
   draw=black,
},s2/.style={ rectangle,
   minimum width=0.5mm,
   minimum height=20mm,
   very thick,
   draw=black,
}]
\newcommand{\FONTSIZE}{\fontsize{3.5mm}{\baselineskip}\selectfont}
\node [s1,text centered] at (7.5,3)
{\FONTSIZE{TransComplm$(2^{k})$}};%
\node [s1, text centered] (n1) at
(7.5,0) {\FONTSIZE{$\overline{\text{TransComplm}}(2^{k})$}};%
\node [text centered,text width=2cm] at (6.2,1.5) {\FONTSIZE Translate};%
\draw[-latex,line width=1pt] (7,1.5) -- (8.1,1.5);%
\node [text centered,text width=2cm] at (9,1.5) {\FONTSIZE Complement};%
\node [s2,text centered] (n1) at (5.5,0)
{\FONTSIZE{1}};%
\node at (5.5,3){
\begin{tabular}{cccc}
\FONTSIZE{0} \\
\FONTSIZE{1} \\
\FONTSIZE{$\vdots$} \\
\FONTSIZE{$2^{k}$}
\end{tabular}
};%
\draw[-latex,line width=2pt] (7.5,2.2) -- (7.5,0.8);%
\node at (11,1.3){
\begin{tabular}{cccc|c}
\ & & &\FONTSIZE{0}&\FONTSIZE{0}\\
\ & & &\FONTSIZE{1}&\FONTSIZE{1}\\
\ & &\FONTSIZE{1}&\FONTSIZE{1}&\FONTSIZE{2}\\
\ & &\FONTSIZE{1}&\FONTSIZE{0}&\FONTSIZE{3}\\
\ &\FONTSIZE{1}&\FONTSIZE{1}&\FONTSIZE{1}&\FONTSIZE{4}\\
\ &\FONTSIZE{1}&\FONTSIZE{1}&\FONTSIZE{0}&\FONTSIZE{5}\\
\ &\FONTSIZE{1}&\FONTSIZE{0}&\FONTSIZE{0}&\FONTSIZE{6}\\
\ &\FONTSIZE{1}&\FONTSIZE{0}&\FONTSIZE{1}&\FONTSIZE{7}\\
\FONTSIZE{1}&\FONTSIZE{1}&\FONTSIZE{1}&\FONTSIZE{1}&\FONTSIZE{8}\\
\FONTSIZE{1}&\FONTSIZE{1}&\FONTSIZE{1}&\FONTSIZE{0}&\FONTSIZE{9}\\
\FONTSIZE{1}&\FONTSIZE{1}&\FONTSIZE{0}&\FONTSIZE{0}&\FONTSIZE{10}\\
\FONTSIZE{1}&\FONTSIZE{1}&\FONTSIZE{0}&\FONTSIZE{1}&\FONTSIZE{11}\\
\FONTSIZE{1}&\FONTSIZE{0}&\FONTSIZE{0}&\FONTSIZE{0}&\FONTSIZE{12}\\
\FONTSIZE{1}&\FONTSIZE{0}&\FONTSIZE{0}&\FONTSIZE{1}&\FONTSIZE{13}\\
\end{tabular}
};%
\node at (11,-2){ \FONTSIZE{$\vdots$}};%
\end{tikzpicture}
\caption{Yet another code.}\label{fig17}
\end{figure}
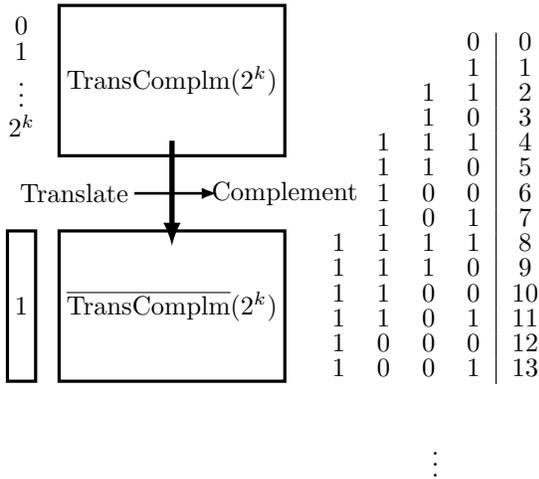

To obtain other examples for which (\ref{Zn}) holds, one may combine
more block operations (such as translation, horizontal or vertical
reflection, reversal, flip,~etc.) and string operations (complement,
reversal, cyclic rotation, rewriting, etc.). A simple example is the
block translation or reflection followed by any cyclic rotation of
each coding (which does not change the number of $1$s). Such a
coding scheme also satisfies (\ref{Zn}).

\section*{Acknowledgement}

We thank the referee for their helpful comments.

\bibliographystyle{abbrv}

\begin{thebibliography}{100}

\bibitem{agnarsson13a}
\textsc{Agnarsson,~G.},
\newblock On the number of hypercubic bipartitions of an integer.
\newblock \emph{Discrete Math.}, 313(24):2857--2864, 2013.
\MR{3115296}

\bibitem{alkauskas04a}
\textsc{Alkauskas,~G.},
\newblock Dirichlet series associated with strongly {$q$}-multiplicative functions.
\newblock \emph{Ramanujan J.}, 8(1):13--21, 2004.
\MR{2068427}

\bibitem{allouche2003a}
\textsc{Allouche,~J.-P.} and \textsc{Shallit,~J.},
\newblock \emph{Automatic Sequences}.
\newblock Cambridge University Press, Cambridge, 2003.
\MR{1997038}

\bibitem{barat06a}
\textsc{Barat,~G., Berth{\'e},~V., Liardet,~P.}, and \textsc{Thuswaldner,~J.},
\newblock Dynamical directions in numeration.
\newblock \emph{Ann. Inst. Fourier (Grenoble)}, 56(7):1987--2092, 2006.
\MR{2290774}

\bibitem{barbour88a}
\textsc{Barbour,~A.~D.},
\newblock Stein's method and poisson process convergence.
\newblock \emph{J. Appl. Probab.}, 25:175--184, 1988.
\MR{0974580}

\bibitem{barbour90a}
\textsc{Barbour,~A.~D.},
\newblock Stein's method for diffusion approximations.
\newblock \emph{Probab. Th. Related Fields}, 84(3):297--322, 1990.
\MR{1035659}

\bibitem{barbour92a}
\textsc{Barbour,~A.~D.} and \textsc{Chen,~L.~H.~Y.},
\newblock On the binary expansion of a random integer.
\newblock \emph{Statist. Probab. Lett.}, 14(3):235--241, 1992.
\MR{1173624}

\bibitem{barbour1992a}
\textsc{Barbour,~A.~D., Holst,~L.}, and \textsc{Janson,~S.},
\newblock \emph{Poisson Approximation}.
\newblock The Clarendon Press, Oxford University Press, New York, 1992.
\MR{1163825}

\bibitem{bassily95a}
\textsc{Bassily,~N.~L.} and \textsc{K{\'a}tai,~I.},
\newblock Distribution of the values of {$q$}-additive functions on polynomial sequences.
\newblock \emph{Acta Math. Hungar.}, 68(4):353--361, 1995.
\MR{1333478}\vadjust{\goodbreak}

\bibitem{bassily96a}
\textsc{Bassily,~N.~L.} and \textsc{K{\'a}tai,~I.},
\newblock Distribution of consecutive digits in the {$q$}-ary expansions of some subsequences of integers.
\newblock In \emph{Proceedings of the {XVI} {S}eminar on {S}tability {P}roblems
for {S}tochastic {M}odels, {P}art {II} ({E}ger, 1994)}, volume~78, pages 11--17, 1996.
\MR{1381030}

\bibitem{bellman48a}
\textsc{Bellman,~R.} and \textsc{Shapiro,~H.~N.},
\newblock On a problem in additive number theory.
\newblock \emph{Ann. of Math. (2)}, 49:333--340, 1948.
\MR{0023864}

\bibitem{berthe-rigo2010a}
\textsc{Berth{\'e},~V.} and \textsc{Rigo,~M.}, editors,
\newblock \emph{Combinatorics, Automata and Number Theory}.
\newblock Cambridge University Press, 2010.
\MR{2742574}

\bibitem{bowden1936a}
\textsc{Bowden,~J.},
\newblock \emph{Special Topics in Theoretical Arithmetic}.
\newblock J. Bowden, Garden City, New York, 1936.

\bibitem{brown94a}
\textsc{Brown,~T.~C.},
\newblock Powers of digital sums.
\newblock \emph{Fibonacci Quart.}, 32(3):207--210, 1994.
\MR{1285747}

\bibitem{bush40a}
\textsc{Bush,~L.~E.},
\newblock An asymptotic formula for the average sum of the digits of integers.
\newblock \emph{Amer. Math. Monthly}, 47:154--156, 1940.
\MR{0001225}

\bibitem{chen04a}
\textsc{Chen,~F.-J.},
\newblock A problem in the $r$-adic representation of positive integers ({C}hinese).
\newblock \emph{J. Nanjing University (Natural Sciences)}, 40(1):89--93, 2004.
\MR{2370534}

\bibitem{chen75a}
\textsc{Chen,~L.~H.~Y.},
\newblock Poisson approximation for dependent trials.
\newblock \emph{Ann. Probability}, 3(3):534--545, 1975.
\MR{0428387}

\bibitem{chen13a}
\textsc{Chen,~L.~H.~Y., Fang,~X.}, and \textsc{Shao,~Q.-M.},
\newblock From {S}tein identities to moderate deviations.
\newblock \emph{Ann. Probab.}, 41(1):262--293, 2013.
\MR{3059199}

\bibitem{chen05a}
\textsc{Chen,~L.~H.~Y.} and \textsc{Shao,~Q.-M.},
\newblock Stein's method for normal approximation.
\newblock In \emph{An Introduction to {S}tein's Method},
volume~4 of \emph{Lect. Notes Ser. Inst. Math. Sci. Natl. Univ. Singap.}, pages 1--59.
Singapore Univ. Press, Singapore, 2005.
\MR{2235448}

\bibitem{chen94a}
\textsc{Chen,~L.~H.~Y.} and \textsc{Soon,~S.~Y.~T.},
\newblock On the number of ones in the binary expansion of a random integer.
\newblock Unpublished manuscript, 1994.

\bibitem{chen99a}
\textsc{Chen,~W.-M., Hwang,~H.-K.}, and \textsc{Chen,~G.-H.},
\newblock The cost distribution of queue-mergesort, optimal mergesorts, and power-of-2 rules.
\newblock \emph{J. Algorithms}, 30(2):423--448, 1999.
\MR{1671856}

\bibitem{cheo55a}
\textsc{Cheo,~P.-H.} and \textsc{Yien,~S.-C.},
\newblock A problem on the {$k$}-adic representation of positive integers.
\newblock \emph{Acta Math. Sinica}, 5:433--438, 1955.
\MR{0075979}

\bibitem{clements65a}
\textsc{Clements,~G.~F.} and \textsc{Lindstr{\"o}m,~B.},
\newblock A sequence of {$(\pm 1)-$}determinants with large values.
\newblock \emph{Proc. Amer. Math. Soc.}, 16:548--550, 1965.
\MR{0178001}

\bibitem{cooper92a}
\textsc{Cooper,~C.} and \textsc{Kennedy,~R.~E.},
\newblock Digital sum sums.
\newblock \emph{J. Inst. Math. Comput. Sci. Math. Ser.}, 5(1):45--49, 1992.
\MR{1182467}

\bibitem{cooper86a}
\textsc{Cooper,~C.~N.} and \textsc{Kennedy,~R.~E.},
\newblock A generalization of a theorem by {C}heo and {Y}ien concerning digital sums.
\newblock \emph{Internat. J. Math. Math. Sci.}, 9(4):817--820, 1986.
\MR{0870542}

\bibitem{coquet86a}
\textsc{Coquet,~J.},
\newblock Power sums of digital sums.
\newblock \emph{J. Number Theory}, 22(2):161--176, 1986.
\MR{0826949}

\bibitem{dartyge09a}
\textsc{Dartyge,~C., Luca,~F.}, and \textsc{St\u{a}nic\u{a},~P.},
\newblock On digit sums of multiples of an integer.
\newblock \emph{J. Number Theory}, 129(11):2820--2830, 2009.
\MR{2549536}

\bibitem{deheuvels86a}
\textsc{Deheuvels,~P.} and \textsc{Pfeifer,~D.},
\newblock A semigroup approach to {P}oisson approximation.
\newblock \emph{Ann. Probab.}, 14(2):663--676, 1986.
\MR{0832029}

\bibitem{delange72a}
\textsc{Delange,~H.},
\newblock Sur les fonctions {$q$}-additives ou {$q$}-multiplicatives.
\newblock \emph{Acta Arith.}, 21:285--298 (errata insert), 1972.
\MR{0309891}

\bibitem{delange75a}
\textsc{Delange,~H.},
\newblock Sur la fonction sommatoire de la fonction ``somme des chiffres''.
\newblock \emph{Enseignement Math. (2)}, 21(1):31--47, 1975.
\MR{0379414}

\bibitem{diaconis77a}
\textsc{Diaconis,~P.},
\newblock The distribution of leading digits and uniform distribution mod~1.
\newblock \emph{Ann. Probability}, 5(1):72--81, 1977.
\MR{0422186}

\bibitem{diaconis1988a}
\textsc{Diaconis,~P.},
\newblock \emph{Group Representations in Probability and Statistics}.
\newblock Institute of Mathematical Statistics Lecture Notes---Monograph Series, 11.
Institute of Mathematical Statistics, Hayward, CA, 1988.
\MR{0964069}

\bibitem{diaconis90a}
\textsc{Diaconis,~P., Graham,~R.~L.}, and \textsc{Morrison,~J.~A.},
\newblock Asymptotic analysis of a random walk on a hypercube with many dimensions.
\newblock \emph{Random Structures Algorithms}, 1(1):51--72, 1990.
\MR{1068491}

\bibitem{dickson1966a}
\textsc{Dickson,~L.~E.},
\newblock \emph{History of the Theory of Numbers. Vol. I: Divisibility and Primality}.
\newblock Chelsea Publishing Co. (unaltered reprintings of the 1919 original), New York, 1966.

\bibitem{docagne86a}
\textsc{d'Ocagne,~M.},
\newblock Sur certaines sommations arithm{\'e}tiques.
\newblock \emph{Jornal de Sciencias Mathematicas e Astronomicas (de M. Gomes Teixeira. Coimbre)}, 7:117--128, 1886.

\bibitem{doran07a}
\textsc{Doran,~R.},
\newblock The {G}ray code.
\newblock \emph{J. Universal Comput. Sci.}, 13(11):1573--1597, 2007.
\MR{2390238}

\bibitem{drazin52a}
\textsc{Drazin,~M.~P.} and \textsc{Griffith,~J.~S.},
\newblock On the decimal representation of integers.
\newblock \emph{Proc. Cambridge Philos. Soc.}, 48:555--565, 1952.
\MR{0049959}

\bibitem{drmota01a}
\textsc{Drmota,~M.},
\newblock The joint distribution of {$q$}-additive functions.
\newblock \emph{Acta Arith.}, 100(1):17--39, 2001.
\MR{1864623}

\bibitem{drmota03a}
\textsc{Drmota,~M., Fuchs,~M.}, and \textsc{Manstavi\v{c}ius,~E.},
\newblock Functional limit theorems for digital expansions.
\newblock \emph{Acta Math. Hungar.}, 98(3):175--201, 2003.
\MR{1956755}

\bibitem{drmota98a}
\textsc{Drmota,~M.} and \textsc{Gajdosik,~J.},
\newblock The distribution of the sum-of-digits function.
\newblock \emph{J. Th\'eor. Nombres Bordeaux}, 10(1):17--32, 1998.
\MR{1827283}

\bibitem{dumont89a}
\textsc{Dumont,~J.-M.} and \textsc{Thomas,~A.},
\newblock Syst\`emes de num\'eration et fonctions fractales relatifs aux substitutions.
\newblock \emph{Theoret. Comput. Sci.}, 65(2):153--169, 1989.
\MR{1020484}

\bibitem{dumont93a}
\textsc{Dumont,~J.-M.} and \textsc{Thomas,~A.},
\newblock Digital sum moments and substitutions.
\newblock \emph{Acta Arith.}, 64(3):205--225, 1993.
\MR{1225425}

\bibitem{dumont97a}
\textsc{Dumont,~J.~M.} and \textsc{Thomas,~A.},
\newblock Gaussian asymptotic properties of the sum-of-digits function.
\newblock \emph{J. Number Theory}, 62(1):19--38, 1997.
\MR{1430000}

\bibitem{ettestad10a}
\textsc{Ettestad,~D.~J.} and \textsc{Carbonara,~J.~O.},
\newblock Formulas for the number of states of an interesting finite cellular
  automaton and a connection to {P}ascal's triangle.
\newblock \emph{J. Cell. Autom.}, 5(1--2):157--166, 2010.
\MR{2583067}

\bibitem{fang02a}
\textsc{Fang,~Y.},
\newblock A theorem on the {$k$}-adic representation of positive integers.
\newblock \emph{Proc. Amer. Math. Soc.}, 130(6):1619--1622 (electronic), 2002.
\MR{1887007}

\bibitem{flajolet94b}
\textsc{Flajolet,~P.} and \textsc{Golin,~M.},
\newblock Mellin transforms and asymptotics. {T}he mergesort recurrence.
\newblock \emph{Acta Inform.}, 31(7):673--696, 1994.
\MR{1300060}

\bibitem{flajolet94a}
\textsc{Flajolet,~P., Grabner,~P., Kirschenhofer,~P., Prodinger,~H.}, and \textsc{Tichy,~R.~F.},
\newblock Mellin transforms and asymptotics: Digital sums.
\newblock \emph{Theoret. Comput. Sci.}, 123(2):291--314, 1994.
\MR{1256203}

\bibitem{flajolet80a}
\textsc{Flajolet,~P.} and \textsc{Ramshaw,~L.},
\newblock A note on {G}ray code and odd-even merge.
\newblock \emph{SIAM J. Comput.}, 9(1):142--158, 1980.
\MR{0557835}

\bibitem{foster87a}
\textsc{Foster,~D.~M.~E.},
\newblock Estimates for a remainder term associated\break  with the sum of digits function.
\newblock \emph{Glasgow Math. J.}, 29(1):109--129, 1987.
\MR{0876156}

\bibitem{foster91a}
\textsc{Foster,~D.~M.~E.},
\newblock A lower bound for a remainder term associated with the sum of digits function.
\newblock \emph{Proc. Edinburgh Math. Soc. (2)}, 34(1):121--142, 1991.
\MR{1093181}

\bibitem{foster92a}
\textsc{Foster,~D.~M.~E.},
\newblock Averaging the sum of digits function to an even base.
\newblock \emph{Proc. Edinburgh Math. Soc. (2)}, 35(3):449--455, 1992.
\MR{1187007}

\bibitem{gelfond67a}
\textsc{Gel'fond,~A.~O.},
\newblock Sur les nombres qui ont des propri\'et\'es additives et multiplicatives donn\'ees.
\newblock \emph{Acta Arith.}, 13:259--265, 1967/1968.
\MR{0220693}

\bibitem{gilbert62a}
\textsc{Gilbert,~E.~N.},
\newblock Games of identification or convergence.
\newblock \emph{SIAM Review}, 4(1):16--24, 1962.

\bibitem{gittenberger00a}
\textsc{Gittenberger,~B.} and \textsc{Thuswaldner,~J.~M.},
\newblock Asymptotic normality of $b$-additive functions on polynomial
  sequences in the gaussian number field.
\newblock \emph{Journal of Number Theory}, 84(2):317--341, 2000.
\MR{1796518}

\bibitem{glaisher99a}
\textsc{Glaisher,~J.~W.~L.},
\newblock On the residue of a binomial-theorem coefficient with respect to a  prime modulus.
\newblock \emph{Quart. J. Pure and Appl. Math.}, 30:150--156, 1899.

\bibitem{glaser1981a}
\textsc{Glaser,~A.},
\newblock \emph{History of Binary and Other Nondecimal Numeration}.
\newblock Tomash Publishers, Los Angeles, Calif., second edition, 1981.
\MR{0666393}

\bibitem{grabner93a}
\textsc{Grabner,~P.~J.},
\newblock Completely {$q$}-multiplicative functions: the {M}ellin transform approach.
\newblock \emph{Acta Arith.}, 65(1):85--96, 1993.
\MR{1239244}

\bibitem{grabner05a}
\textsc{Grabner,~P.~J.} and \textsc{Hwang,~H.-K.},
\newblock Digital sums and divide-and-conquer recurrences: {F}ourier expansions and absolute convergence.
\newblock \emph{Constr. Approx.}, 21(2):149--179, 2005.
\MR{2107936}

\bibitem{grabner93b}
\textsc{Grabner,~P.~J., Kirschenhofer,~P., Prodinger,~H.}, and \textsc{Tichy,~R.~F.},
\newblock On the moments of the sum-of-digits function.
\newblock In \emph{Applications of {F}ibonacci Numbers, {V}ol. 5 ({S}t. {A}ndrews, 1992)}, pages 263--271.
Kluwer Acad. Publ., Dordrecht, 1993.
\MR{1271366}

\bibitem{graham70a}
\textsc{Graham,~R.~L.},
\newblock On primitive graphs and optimal vertex assignments.
\newblock \emph{Ann. New York Acad. Sci.}, 175:170--186, 1970.
\MR{0269533}

\bibitem{greene2008a}
\textsc{Greene,~D.~H.} and \textsc{Knuth,~D.~E.},
\newblock \emph{Mathematics for the Analysis of Algorithms}.
\newblock Modern Birkh\"auser Classics. Birkh\"auser Boston Inc., Boston, MA, 2008.
\MR{2381155}

\bibitem{hadjicostas11a}
\textsc{Hadjicostas,~P.} and \textsc{Lakshmanan,~K.~B.},
\newblock Recursive merge sort with erroneous comparisons.
\newblock \emph{Discrete Appl. Math.}, 159(14):1398--1417, 2011.
\MR{2823899}

\bibitem{hart76a}
\textsc{Hart,~S.},
\newblock A note on the edges of the {$n$}-cube.
\newblock \emph{Discrete Math.}, 14(2):157--163, 1976.
\MR{0396293}

\bibitem{hata84a}
\textsc{Hata,~M.} and \textsc{Yamaguti,~M.},
\newblock The {T}akagi function and its generalization.
\newblock \emph{Japan J. Appl. Math.}, 1(1):183--199, 1984.
\MR{0839313}

\bibitem{heppner76a}
\textsc{Heppner,~E.},
\newblock \"{U}ber die {S}umme der {Z}iffern nat\"urlicher {Z}ahlen.
\newblock \emph{Ann. Univ. Sci. Budapest. E\"otv\"os Sect. Math.}, 19:41--43 (1977),
1976.\break
\MR{0506024}

\bibitem{hofer08a}
\textsc{Hofer,~R., Larcher,~G.}, and \textsc{Pillichshammer,~F.},
\newblock Average growth-behavior and distribution properties of generalized
  weighted digit-block-counting functions.
\newblock \emph{Monatsh. Math.}, 154(3):199--230, 2008.\break
\MR{2413302}

\bibitem{holmes04a}
\textsc{Holmes,~S.},
\newblock Stein's method for birth and death chains.
\newblock In \emph{Stein's Method: Expository Lectures and Applications},
volume~46 of \emph{IMS Lecture Notes Monogr. Ser.}, pages 45--67.
Inst. Math. Statist., Beachwood, OH, 2004.
\MR{2118602}

\bibitem{hong82a}
\textsc{Hong,~Z.} and \textsc{Sedgewick,~R.},
\newblock Notes on merging networks (preliminary version).
\newblock In \emph{Proc. ACM Symposium on Theory of Computing}, pages 296--302, 1982.

\bibitem{ifrah2000a}
\textsc{Ifrah,~G.},
\newblock \emph{The Universal History of Numbers}.
\newblock John Wiley \& Sons Inc., New York, 2000.
\newblock From prehistory to the invention of the computer,
Translated from the 1994 French original by
David Bellos, E. F. Harding, Sophie Wood and Ian Monk.
\MR{1725387}

\bibitem{ismail2009}
\textsc{Ismail,~M.~E.~H.},
\newblock \emph{Classical and Quantum Orthogonal Polynomials in One Variable},
volume~98 of \emph{Encyclopedia of Mathematics and Its Applications}.
\newblock Cambridge University Press, Cambridge, 2009.
\MR{2542683}

\bibitem{kano91a}
\textsc{Kano,~H.},
\newblock On the sums of digits in integers.
\newblock \emph{Proc. Japan Acad. Ser. A Math. Sci.}, 67(5):148--150, 1991.
\MR{1114959}

\bibitem{katai77a}
\textsc{K{\'a}tai,~I.},
\newblock On the sum of digits of primes.
\newblock \emph{Acta Math. Acad. Sci. Hungar.}, 30(1--2):169--173, 1977.
\MR{0472747}

\bibitem{katai68a}
\textsc{K{\'a}tai,~I.} and \textsc{Mogyor{\'o}di,~J.},
\newblock On the distribution of digits.
\newblock \emph{Publ. Math. Debrecen}, 15:57--68, 1968.
\MR{0236139}

\bibitem{kennedy91a}
\textsc{Kennedy,~R.~E.} and \textsc{Cooper,~C.~N.},
\newblock An extension of a theorem by {C}heo and {Y}ien concerning digital  sums.
\newblock \emph{Fibonacci Quart.}, 29(2):145--149, 1991.
\MR{1119401}

\bibitem{kirschenhofer90a}
\textsc{Kirschenhofer,~P.},
\newblock On the variance of the sum of digits function.
\newblock In \emph{Number-Theoretic Analysis ({V}ienna, 1988--89)},
volume 1452 of \emph{Lecture Notes in Math.}, pages 112--116.
Springer, Berlin, 1990.
\MR{1084640}

\bibitem{kirschenhofer84a}
\textsc{Kirschenhofer,~P.} and \textsc{Prodinger,~H.},
\newblock Subblock occurrences in positional number systems and {G}ray code  representation.
\newblock \emph{J. Inform. Optim. Sci.}, 5(1):29--42, 1984.
\MR{0737164}

\bibitem{klavzar07a}
\textsc{Klav{\v{z}}ar,~S., Milutinovi{\'c},~U.}, and \textsc{Petr,~C.},
\newblock Stern polynomials.
\newblock \emph{Adv. in Appl. Math.}, 39(1):86--95, 2007.
\MR{2319565}

\bibitem{knuth1997a}
\textsc{Knuth,~D.~E.},
\newblock \emph{Art of Computer Programming, Volume 2: Seminumerical Algorithms}.
\newblock Addison-Wesley, third edition, November 1997.

\bibitem{kobayashi02a}
\textsc{Kobayashi,~Z.},
\newblock Digital sum problems for the {G}ray code representation of natural  numbers.
\newblock \emph{Interdiscip. Inform. Sci.}, 8(2):167--175, 2002.
\MR{1972038}

\bibitem{kobayashi02b}
\textsc{Kobayashi,~Z.} and \textsc{Sekiguchi,~T.},
\newblock On a characterization of the standard {G}ray code by using it edge  type on a hypercube.
\newblock \emph{Inform. Process. Lett.}, 81(5):231--237, 2002.
\MR{1879645}

\bibitem{kruppel09a}
\textsc{Kr{\"u}ppel,~M.},
\newblock De {R}ham's singular function, its partial derivatives with respect
  to the parameter and binary digital sums.
\newblock \emph{Rostock. Math. Kolloq.}, 64:57--74, 2009.
\MR{2605000}

\bibitem{kummer52a}
\textsc{Kummer,~E.~E.},
\newblock {\"Uber die Erg\"anzungss\"atze zu den allgemeinen Reciprocit\"atsgesetzen}.
\newblock \emph{J. Reine Angew. Math.}, 44:93--146, 1852.

\bibitem{laczay07a}
\textsc{Laczay,~B.} and \textsc{Ruszink{\'o},~M.},
\newblock Collision channel with multiplicity feedback.
\newblock In E.~Biglieri and L.~Gy{\"o}rfi, editors,
\emph{Proceedings of the NATO Advanced Study Institute on Coding and Analysis of Multiple Access Channels.
Theory and Practice},
volume D. 10, pages 250--270. IOS Press, 2007.

\bibitem{lagarias12a}
\textsc{Lagarias,~J.~C.},
\newblock The {T}akagi function and its properties.
\newblock In \emph{Functions in Number Theory and Their Probabilistic Aspects},
RIMS K\^oky\^uroku Bessatsu, B34, pages 153--189.
Res. Inst. Math. Sci. (RIMS), Kyoto, 2012.
\MR{3014845}

\bibitem{legendre1900a}
\textsc{Legendre,~A.},
\newblock \emph{Th{\'e}orie des Nombres}.
\newblock Firmin Didot Fr{\`e}res, fourth edition, 1900.

\bibitem{li86a}
\textsc{Li,~S.-Y.~R.},
\newblock Binary trees and uniform distribution of traffic cutback.
\newblock \emph{J. Comput. System Sci.}, 32(1):1--14, 1986.
\MR{0844201}

\bibitem{lindstrom65a}
\textsc{Lindstr{\"o}m,~B.},
\newblock On a combinatorial problem in number theory.
\newblock \emph{Canad. Math. Bull.}, 8:477--490, 1965.
\MR{0181604}

\bibitem{loh92a}
\textsc{Loh,~W.-L.},
\newblock Stein's method and multinomial approximation.
\newblock \emph{Ann. Appl. Probab.}, 2(3):536--554, 1992.
\MR{1177898}

\bibitem{lucas78a}
\textsc{Lucas,~\'E.},
\newblock Sur les congruences des nombres eul{\'e}riens et des coefficients
  diff{\'e}rentiels des fonctions trigonom{\'e}triques, suivant un module premier.
\newblock \emph{Bull. Soc. Math. France}, 6:49--54, 1878.
\MR{1503769}

\bibitem{macwilliams1977}
\textsc{MacWilliams,~F.~J.} and \textsc{Sloane,~N.~J.~A.},
\newblock \emph{The Theory of Error-Correcting Codes}.
\newblock North-Holland Publishing Co., Amsterdam-New York-Oxford, 1977.

\bibitem{madritsch11a}
\textsc{Madritsch,~M.} and \textsc{Peth\H{o},~A.},
\newblock Asymptotic normality of additive functions on polynomial sequences
in canonical number systems.
\newblock \emph{J. Number Theory}, 131(9):1553--1574, 2011.
\MR{2802135}

\bibitem{madritsch10a}
\textsc{Madritsch,~M.~G.},
\newblock Asymptotic normality of {$b$}-additive functions on polynomial sequences in number systems.
\newblock \emph{Ramanujan J.}, 21(2):181--210, 2010.
\MR{2593247}

\bibitem{manstavicius97a}
\textsc{Manstavi{\v{c}}ius,~E.},
\newblock Probabilistic theory of additive functions related to systems of numeration.
\newblock In \emph{New {T}rends in {P}robability and {S}tatistics, {V}ol.~4 ({P}alanga, 1996)},
pages 413--429.
VSP, Utrecht, 1997.
\MR{1653626}

\bibitem{mauclaire93a}
\textsc{Mauclaire,~J.-L.},
\newblock Sur la r\'epartition des fonctions {$q$}-additives.
\newblock \emph{J. Th\'eor. Nombres Bordeaux}, 5(1):79--91, 1993.
\MR{1251228}

\bibitem{mauclaire83a}
\textsc{Mauclaire,~J.-L.} and \textsc{Murata,~L.},
\newblock On {$q$}-additive functions. {I}.
\newblock \emph{Proc. Japan Acad. Ser. A Math. Sci.}, 59(6):274--276, 1983.
\MR{0718620}

\bibitem{mauclaire83b}
\textsc{Mauclaire,~J.-L.} and \textsc{Murata,~L.},
\newblock On {$q$}-additive functions. {II}.
\newblock \emph{Proc. Japan Acad. Ser. A Math. Sci.}, 59(9):441--444, 1983.
\MR{0732606}

\bibitem{mauduit05a}
\textsc{Mauduit,~C.} and \textsc{Rivat,~J.},
\newblock Propri\'et\'es {$q$}-multiplicatives de la suite {$\lfloor n^c\rfloor$}, {$c>1$}.
\newblock \emph{Acta Arith.}, 118(2):187--203, 2005.
\MR{2141049}

\bibitem{mauduit09a}
\textsc{Mauduit,~C.} and \textsc{Rivat,~J.},
\newblock La somme des chiffres des carr\'es.
\newblock \emph{Acta Math.}, 203(1):107--148, 2009.
\MR{2545827}

\bibitem{mauduit10a}
\textsc{Mauduit,~C.} and \textsc{Rivat,~J.},
\newblock Sur un probl{\`e}me de {G}elfond: la somme des chiffres des nombres premiers.
\newblock \emph{Ann. of Math.}, 171(3):1591--1646, 2010.
\MR{2680394}

\bibitem{mcilroy74a}
\textsc{McIlroy,~M.~D.},
\newblock The number of {$1$}'s in binary integers: Bounds and extremal properties.
\newblock \emph{SIAM J. Comput.}, 3:255--261,\break  1974.
\MR{0436687}

\bibitem{mehrabian13a}
\textsc{Mehrabian,~A., Mitsche,~D.}, and \textsc{Pra{\l}at,~P.},
\newblock On the maximum density of graphs with unique-path labelings.
\newblock \emph{SIAM J. Discrete Math.}, 27(3):1228--1233, 2013.
\MR{3072758}

\bibitem{mirsky49a}
\textsc{Mirsky,~L.},
\newblock A theorem on representations of integers in the scale of {$r$}.
\newblock \emph{Scripta Math.}, 15:11--12, 1949.
\MR{0030991}

\bibitem{morrison86a}
\textsc{Morrison,~J.~A.},
\newblock Weighted averages of {R}adon transforms on {$Z^k_2$}.
\newblock \emph{SIAM J. Algebraic Discrete Methods}, 7(3):404--413, 1986.
\MR{0844043}

\bibitem{muramoto00a}
\textsc{Muramoto,~K., Okada,~T., Sekiguchi,~T.}, and \textsc{Shiota,~Y.},
\newblock Digital sum problems for the {$p$}-adic expansion of natural numbers.
\newblock \emph{Interdiscip. Inform. Sci.}, 6(2):105--109, 2000.
\MR{1839805}

\bibitem{muramoto03a}
\textsc{Muramoto,~K., Okada,~T., Sekiguchi,~T.}, and \textsc{Shiota,~Y.},
\newblock Power and exponential sums of digital sums with information per  digits.
\newblock \emph{Math. J. Toyama Univ.}, 26:35--44, 2003.
\MR{2048391}

\bibitem{murata88a}
\textsc{Murata,~L.} and \textsc{Mauclaire,~J.-L.},
\newblock An explicit formula for the average of some {$q$}-additive functions.
\newblock In \emph{Prospects of {M}athematical {S}cience ({T}okyo, 1986)}, pages 141--156.
World Sci. Publishing, Singapore, 1988.
\MR{0948466}

\bibitem{okada95a}
\textsc{Okada,~T., Sekiguchi,~T.}, and \textsc{Shiota,~Y.},
\newblock Applications of binomial measures to power sums of digital sums.
\newblock \emph{J. Number Theory}, 52(2):256--266, 1995.
\MR{1336748}

\bibitem{okada95b}
\textsc{Okada,~T., Sekiguchi,~T.}, and \textsc{Shiota,~Y.},
\newblock An explicit formula of the exponential sums of digital sums.
\newblock \emph{Japan J. Indust. Appl. Math.}, 12(3):425--438, 1995.
\MR{1356664}

\bibitem{okada96a}
\textsc{Okada,~T., Sekiguchi,~T.}, and \textsc{Shiota,~Y.},
\newblock A generalization of {H}ata-{Y}amaguti's results on the {T}akagi function. {II}.
{M}ultinomial case.
\newblock \emph{Japan J. Indust. Appl. Math.}, 13(3):435--463, 1996.\break
\MR{1415064}

\bibitem{osbaldestin91a}
\textsc{Osbaldestin,~A.~H.},
\newblock Digital sum problems.
\newblock In \emph{Fractals in the Fundamental and Applied Sciences}, pages 307--328.
Elsevier Science, B. V., North-Holland, Amsterdam, 1991.

\bibitem{panny95a}
\textsc{Panny,~W.} and \textsc{Prodinger,~H.},
\newblock Bottom-up mergesort---A detailed analysis.
\newblock \emph{Algorithmica}, 14(4):340--354, 1995.
\MR{1343320}

\bibitem{prodinger82a}
\textsc{Prodinger,~H.},
\newblock Generalizing the sum of digits function.
\newblock \emph{SIAM J. Algebraic Discrete Methods}, 3(1):35--42, 1982.
\MR{0644955}

\bibitem{prodinger83b}
\textsc{Prodinger,~H.},
\newblock Nonrepetitive sequences and {G}ray code.
\newblock \emph{Discrete Math.}, 43(1):113--116, 1983.
\MR{0680311}

\bibitem{prodinger83a}
\textsc{Prodinger,~H.},
\newblock A subword version of d'{O}cagne's formula.
\newblock \emph{Utilitas Math.}, 24:125--129, 1983.
\MR{0724766}

\bibitem{prodinger02a}
\textsc{Prodinger,~H.},
\newblock Digits and beyond.
\newblock In \emph{Mathematics and {C}omputer {S}cience, {II} ({V}ersailles, 2002)}, Trends Math.,
pages 355--377.
Birkh\"auser, Basel, 2002.
\MR{1940147}

\bibitem{roberts57a}
\textsc{Roberts,~J.~B.},
\newblock On binomial coefficient residues.
\newblock \emph{Canad. J. Math.}, 9:363--370, 1957.
\MR{0086828}

\bibitem{roos01a}
\textsc{Roos,~B.},
\newblock Binomial approximation to the {P}oisson binomial distribution: The {K}rawtchouk expansion.
\newblock \emph{Theory Probab. Appl.}, 45(2):258--272, 2001.
\MR{1967760}

\bibitem{sandor2004a}
\textsc{S{\'a}ndor,~J.} and \textsc{Crstici,~B.},
\newblock \emph{Handbook of Number Theory. {II}}.
\newblock Kluwer Academic Publishers, Dordrecht, 2004.
\MR{2119686}

\bibitem{savage97a}
\textsc{Savage,~C.},
\newblock A survey of combinatorial {G}ray codes.
\newblock \emph{SIAM Rev.}, 39(4):605--629, 1997.
\MR{1491049}

\bibitem{schmid84a}
\textsc{Schmid,~J.},
\newblock The joint distribution of the binary digits of integer multiples.
\newblock \emph{Acta Arith.}, 43(4):391--415, 1984.
\MR{0756290}

\bibitem{schmidt83a}
\textsc{Schmidt,~W.~M.},
\newblock The joint distribution of the digits of certain integer {$s$}-tuples.
\newblock In \emph{Studies in {P}ure {M}athematics}, pages 605--622.
Birkh\"auser, Basel, 1983.
\MR{0820255}

\bibitem{schoutens2000}
\textsc{Schoutens,~W.},
\newblock \emph{Stochastic Processes and Orthogonal Polynomials},
volume 146 of \emph{Lecture Notes in Statistics}.
\newblock Springer-Verlag, New York, 2000.
\MR{1761401}

\bibitem{shiokawa73a}
\textsc{Shiokawa,~I.},
\newblock On a problem in additive number theory.
\newblock \emph{Math. J. Okayama Univ.}, 16:167--176, 1973/1974.
\MR{0357352}

\bibitem{shiokawa74b}
\textsc{Shiokawa,~I.},
\newblock {$g$}-adical analogues of some arithmetical functions.
\newblock \emph{Math. J. Okayama Univ.}, 17:75--94, 1974.
\MR{0364069}

\bibitem{soon1993a}
\textsc{Soon,~Y.-T.},
\newblock \emph{Some Problems in Binomial and Compound Poisson Approximations}.
\newblock Ph.D. Thesis, National University of Singapore, 1993.

\bibitem{stein80a}
\textsc{Stein,~A.~H.},
\newblock Exponential sums related to binomial coefficient parity.
\newblock \emph{Proc. Amer. Math. Soc.}, 80(3):526--530, 1980.
\MR{0581019}

\bibitem{stein86a}
\textsc{Stein,~A.~H.},
\newblock Exponential sums of sum-of-digit functions.
\newblock \emph{Illinois J. Math.}, 30(4):660--675, 1986.
\MR{0857218}

\bibitem{stein72a}
\textsc{Stein,~C.},
\newblock A bound for the error in the normal approximation to the distribution
  of a sum of dependent random variables.
\newblock In \emph{Proceedings of the {S}ixth {B}erkeley {S}ymposium on
  {M}athematical {S}tatistics and {P}robability ({U}niv. {C}alifornia, {B}erkeley, {C}alif., 1970/1971),
{V}ol. {II}: {P}robability {T}heory}, pages 583--602.
Univ. California Press, Berkeley, Calif., 1972.
\MR{0402873}

\bibitem{stein1986a}
\textsc{Stein,~C.},
\newblock \emph{Approximate Computation of Expectations}.
\newblock Institute of Mathematical Statistics Lecture Notes---Monograph Series, 7.
Institute of Mathematical Statistics, Hayward, CA, 1986.
\MR{0882007}

\bibitem{steiner2002a}
\textsc{Steiner,~W.},
\newblock \emph{The Distribution of Digital Expansions on Polynomial Sequences}.
\newblock Dissertation, TU-Wien, 2002.

\bibitem{stolarsky77a}
\textsc{Stolarsky,~K.~B.},
\newblock Power and exponential sums of digital sums related to binomial coefficient parity.
\newblock \emph{SIAM J. Appl. Math.}, 32(4):717--730, 1977.
\MR{0439735}

\bibitem{stolarsky80a}
\textsc{Stolarsky,~K.~B.},
\newblock Integers whose multiples have anomalous digital frequencies.
\newblock \emph{Acta Arith.}, 38(2):117--128, 1980/81.
\MR{0604228}

\bibitem{szego1975a}
\textsc{Szeg\H{o},~G.},
\newblock \emph{Orthogonal Polynomials}.
\newblock AMS, Providence, R.I., fourth edition, 1975.
\MR{0372517}

\bibitem{tang63a}
\textsc{Tang,~S.~C.},
\newblock An improvement and generalization of {B}ellman-{S}hapiro's theorem on
  a problem in additive number theory.
\newblock \emph{Proc. Amer. Math. Soc.}, 14:199--204, 1963.
\MR{0150082}

\bibitem{tenenbaum97a}
\textsc{Tenenbaum,~G.},
\newblock Sur la non-d\'erivabilit\'e de fonctions p\'eriodiques associ\'ees \`a\ certaines formules sommatoires.
\newblock In \emph{The {M}athematics of {P}aul {E}rd{\H o}s, {I}},
volume~13 of \emph{Algorithms Combin.}, pages 117--128.
Springer, Berlin, 1997.
\MR{1425180}

\bibitem{terras1999}
\textsc{Terras,~A.},
\newblock \emph{Fourier Analysis on Finite Groups and Applications},
volume~43 of \emph{London Mathematical Society Student Texts}.
\newblock Cambridge University Press, Cambridge, 1999.
\MR{1695775}

\bibitem{thim2003a}
\textsc{Thim,~J.},
\newblock \emph{Continuous Nowhere Differentiable Functions}.
\newblock Master Thesis, Lule\aa\ Tekniska Universitet, 2003.

\bibitem{trollope67a}
\textsc{Trollope,~J.~R.},
\newblock Generalized bases and digital sums.
\newblock \emph{Amer. Math. Monthly}, 74:690--694, 1967.
\MR{0211950}

\bibitem{trollope68a}
\textsc{Trollope,~J.~R.},
\newblock An explicit expression for binary digital sums.
\newblock \emph{Math. Mag.}, 41:21--25, 1968.
\MR{0233763}

\bibitem{wolfram83a}
\textsc{Wolfram,~S.},
\newblock Statistical mechanics of cellular automata.
\newblock \emph{Rev. Modern Phys.}, 55(3):601--644, 1983.
\MR{0709077}

\bibitem{wolfram84a}
\textsc{Wolfram,~S.},
\newblock Geometry of binomial coefficients.
\newblock \emph{Amer. Math. Monthly}, 91(9):566--571, 1984.
\MR{0764797}

\bibitem{yu96a}
\textsc{Yu,~X.~Y.},
\newblock On the mean-value of the powers of digital sums.
\newblock \emph{Kexue Tongbao (Chinese)}, 41(7):581--585, 1996.
\MR{1418096}

\bibitem{zacharovas10a}
\textsc{Zacharovas,~V.} and \textsc{Hwang,~H.-K.},
\newblock A {C}harlier-{P}arseval approach to {P}oisson approximation and its
  applications.
\newblock \emph{Lithuanian Math. J.}, 50(1):88--119, 2010.
\MR{2607681}

\end{thebibliography}

\end{document}